\documentclass{article}

\usepackage{amsmath, amssymb}
\usepackage{amsthm}
\usepackage{mathrsfs}
\usepackage{xcolor}
\usepackage[toc,page]{appendix}

\makeatletter
\newcommand{\customlabel}[2]{%
\protected@write \@auxout {}{\string \newlabel {#1}{{#2}{}}}}
\makeatother

\newcommand{\X}{\mathbb{X}}
\newcommand{\Y}{\mathbb{Y}}
\newcommand{\R}{\mathbb{R}}
\newcommand{\Rn}{\mathbb{R}^n}

\newcommand{\C}{\mathcal{C}}

\newcommand{\K}{\mathcal{K}}
\newcommand{\N}{\mathcal{N}}

\newcommand{\graph}[1]{\mathrm{Graph}(#1)}
\newcommand{\cseg}[1]{\mathrm{Seg}^c_{#1}}
\newcommand{\csec}[2]{\mathrm{Sec}_{#1}^{#2}}
\newcommand{\T}{\mathcal{T}}

\newcommand{\diam}[1]{\mathrm{diam}(#1)}
\newcommand{\dist}[2]{\mathrm{dist}(#1,#2)}

\newcommand{\Int}[1]{\mathrm{Int}\left( #1 \right)}
\newcommand{\Dom}[1]{\mathrm{Dom}(#1)}

\newcommand{\cexp}[2]{\mathrm{exp}^c_{#1}(#2)}
\newcommand{\Fx}{F_{X_0 X_1}}
\newcommand{\Fxp}{F_{X_0' X_1'}}

\newcommand{\MTW}{\mathrm{MTW}}
\newcommand{\clip}{\| c \|_{\mathrm{Lip}}}

\numberwithin{equation}{section}

\newtheorem{Thm}{Theorem}[section]
\newtheorem*{Thm*}{Theorem}
\newtheorem{Lem}[Thm]{Lemma}
\newtheorem{Prop}[Thm]{Proposition}
\newtheorem{Cor}[Thm]{Corollary}

\theoremstyle{definition}
\newtheorem{Def}[Thm]{Definition}

\newtheorem*{Prbm*}{Problem}

\newtheorem{Cond}{Condition}

\theoremstyle{remark}
\newtheorem{Rmk}[Thm]{Remark}

\title{An alternative definition for $c$-convex functions and another synthetic statement of the MTW condition}
\author{Seonghyeon Jeong\thanks{E-mail: jeongs1466@gmail.com}}
\date{}

\begin{document}

\maketitle

\begin{abstract}
The main theorem of this paper states that the $c$-convexity and the alternative $c$-convexity are equivalent if and only if the cost function $c$ satisfies MTW condition. The alternative $c$-convex function is an analogy of the definition of the convex function that is using the inequality $\phi(tx_1 + (1-t)x_0) \leq t\phi(x_1) +(1-t) \phi(x_0)$. we study properties of the alternative $c$-convex functions and MTW condition, then prove the main theorem. 
\end{abstract}

\section{Introduction}

The regularity theory of optimal transport problem is closely related to the analysis of Monge-Amp\`ere type equations, which are known to be fully non-linear degenerate elliptic partial differential equations.
\begin{equation*}
\det( D^2 \phi + \mathcal{A}(x,Du) ) = f(x, u, Du).
\end{equation*}
When the cost function $c$ is given by the distance square of the two points $\frac{1}{2} \| x - y \|^2$, the corresponding equation is the Monge-Amp\`ere equation, and the optimal transport plan is induced from the gradient of the convex potential $\phi + \frac{\| x \|^2}{2}$ that solves a second boundary value problem of the Monge-Amp\`ere equation, where the regularity is studied in, for example, but not limited to, \cite{Caf90},\cite{Caf91},\cite{Caf92},\cite{FK10},\cite{PF13monge},\cite{W96}. With general cost functions, the regularity depends delicately on the structure of the cost function.

A key break through for the regularity of optimal transport problems with general costs was made in \cite{MTW}, where a sign condition on a (2,2)-tensor was discovered. Named after the authors of \cite{MTW}, we call this tensor the MTW tensor and the condition the MTW condition (see \eqref{eqn: MTW tensor} and Definition \ref{def: MTW condition}). The MTW condition is also known as the A3 condition in the optimal transport literature. It is proved later by Loeper \cite{Loe09} that, under suitable situation, the MTW condition is in fact a necessary and sufficient condition for the H\"older regularity of optimal transport plans. As such, the MTW condition is regarded as an essential condition for the regularity theory, and is used in many regularity researches for optimal transport problem and Monge-Amp\`ere type equations (for instance, see \cite{Liu09},\cite{GK15},\cite{PF13},\cite{FKM13},\cite{CW16}).

When the MTW condition was first discovered, it was not clear what the MTW condition means due to the form of the condition. The formula of the MTW tensor already gives that the MTW condition implies some convexity of some matrix $\mathcal{A}$ in certain directions. Still it was not clear why the MTW condition matters in the regularity theory of optimal transport. Therefore, there were studies which tried to find another form of the MTW condition that provides more insightful interpretation. In \cite{KM10}, the MTW condition is expressed as a non-negativity condition of the curvature tensor of a pseudo-Riemmanian metric that is induced from a metric which is constructed using the cost function. In addition, \cite{KZ20} shows that the MTW condition can also be expressed using the curvature tensor of a K\"ahler metric in some special cases. Loeper found a synthetic expression of the MTW condition, which we call Loeper's property (see Definition \ref{def: Loeper}), that can be interpreted as the convexity of the $c$-subdifferentials of $c$-convex functions at a point \cite{Loe09}. Another synthetic expression that can be interpreted as a quantitative version of Loeper's property was found by Guillen and Kitagawa  (\cite{GK15}, \cite{GK17}) independently. In addition, Loeper and Trudinger showed that Loeper's property with a small error is still equivalent to the MTW condition. In the recent work of Rankin (\cite{Ran23}), it has been shown that the MTW condition is also equivalent to the implication between the monotonicity and the $c$-monotonicity.

In this paper, we introduce another condition, that is equivalent to Loeper's property, using a new type of functions which we call alternative $c$-convex functions. Motivation of the alternative $c$-convex function is the definition of the convex function that uses the segment connecting two points on the graph of the convex function: $\phi$ is convex if
\begin{equation*}
    \phi(tx_1 + (1-t)x_0) \leq t \phi(x_1) + (1-t) \phi(x_0)
\end{equation*}
for any $x_0, x_1 \in \Dom{\phi}$. This inequality can be extended to the whole domain by defining the function on the right hand side to be $\infty$ outside the segment connecting $x_0$ and $x_1$:
\begin{equation*}
    \phi(x) \leq \left\{ \begin{matrix}
        t \phi(x_0) + (1-t) \phi(x_1) & x = tx_1 + (1-t)x_0 \\
        \infty & \textrm{otherwise}
    \end{matrix}\right. .
\end{equation*}
Then, the function on the right hand side of the above inequality can be expressed as follows
\begin{equation*}
    \sup\{ \langle x, y \rangle + h | \langle x_i, y \rangle + h \leq \phi(x_i), i = 0,1 \}.
\end{equation*}
Hence, recalling that $\langle x,y \rangle$ corresponds to $-c(x,y)$ in the definition of $c$-convex function, we may replace $\langle x, y \rangle$ with $-c(x,y)$ and obtain the definition for the alternative $c$-convex functions. 

It is very well known that the definitions of the convex function using the segments and using the supporting planes are equivalent. However, the alternative $c$-convex functions and the $c$-convex functions does not have to be equivalent definitions. Indeed, if $c(x,y) = \| x-y\|^p$ for $p>2$ defined on $\Rn \times \Rn$, one can easily compute that we have
\begin{equation*}
    \sup\{ -c(x,y) + h | -c(x,y) + h \leq \phi(x_i), i = 0,1 \} = \left\{ \begin{matrix}
        \phi(x_i) & i = 0,1 \\
        \infty & \textrm{otherwise}
    \end{matrix} \right. .
\end{equation*}
In particular, any function is alternative $c$-convex, while it is obvious that $c$-convex functions are much more restrictive. Thus the $c$-convexity and the alternative $c$-convex do not necessarily equivalent, and we can ask when they be equivalent.

Interestingly, the condition on the cost function that gives the equivalence of the $c$-convexity and the alternative $c$-convexity is the MTW condition. The MTW condition arises from the regularity theory of the optimal transport problem, while the alternative $c$-convex function arises from the simple observation about the analogous definition of the convex function. The following theorem, which is the main theorem of this paper, yields, however, that the MTW condition and the two analogous definitions of the convex functions ($c$convex functions and alternative $c$-convex functions) are strongly related.

\begin{Thm}[Main theorem]\label{thm: main theorem}
Let $\X$ and $\Y$ be compact subsets of $\Rn$ with non-empty interior, and let $c: \X \times \Y \to \R$ be a $C^2$ function that satisfies bi-twist, non-degenerate. Also, assume that $\X$ and $\Y$ are $c$-convex with respect to each other. Then we have the following:
\begin{enumerate}
	\item If $c$ satisfies Loeper's property, then a function is $c$-convex if and only if it is alternative $c$-convex.
	\item If $c$ does not satisfy Loeper's property, then there exists an alternative $c$-convex function that is not $c$-convex function.
\end{enumerate}
\end{Thm}

The main theorem states that the definitions of the $c$-convex functions and the alternative $c$-convex functions are equivalent if and only if the cost function $c$ satisfies Loeper's property. Since Loeper's property is equivalent to the MTW condition under enough regularity, we also have that the equivalence of $c$-convexity and alternative $c$-convexity is also equivalent to the original MTW condition under enough regularity. The conditions which are used in the statement of the main theorem Theorem \ref{thm: main theorem} are explained in section \ref{sec: cost function}.

We would like to emphasize a few points about the main theorem. Unlike the existing statement that are equivalent to the MTW condition, Theorem \ref{thm: main theorem} suggest another condition that is equivalent to the MTW condition, but \emph{without any derivatives}. In addition, the proof of the main theorem requires only $C^2$ regularity in the sense that $D^2_{xy}c$ and $D^2_{yx}c$ are exist and continuous, and $D^2_{xy}c=(D^2_{yx}c)^T$. In particular, we do not use $D^2_{xx}c$ and $D^2_{yy}c$ to show the main theorem. 

This paper is organized as follows. In Section \ref{sec: cost function}, we impose conditions on the cost function and the spaces $\X$ and $\Y$ on which the cost function is defined. Then we prove several lemmas about the $c$-affine functions. In Section \ref{sec: c-chord}, we define the $c$-chord and show several properties of the $c$-chord. In Subsection \ref{subsec: c-chord no mtw}, we show properties that does not require the MTW condition. The properties of the $c$-chord that use the MTW condition are in Subsection \ref{subsec: c-chord with mtw}. In Section \ref{sec: alternative c-conv}, we define the alternative $c$-convex function, and explore some properties of the alternative $c$-convex functions with or without the MTW condition. Finally, We prove the main theorem in Section \ref{sec: proof}. The first part of the main theorem is proved in Subsection \ref{subsec: fisrt part}, and the proof of the second part of the main theorem is presented in Subsection \ref{subsec: second part}.

\subsection*{Notations}
For real numbers $a<b$, we use $(a,b)$, $[a,b]$, $[a,b)$, $(a,b]$ to denote the intervals with and without the end points. For $x_0, x_1 \in \Rn$, we use $(x_0, x_1)$, $[x_0, x_1]$, $[x_0, x_1)$, $(x_0, x_1]$ to denote the segment between $x_0$ and $x_1$ with or without the boundary. For instance,
\begin{equation*}
[x_0, x_1) = \{ t x_1 + (1-t)x_0 | t \in [0,1) \}.
\end{equation*}
For any $x \in \Rn$, we use $\| x \|$ to denote the Euclidean norm of $x$. If $M$ is a matrix, we use $\| M \|_F$ to denote the Frobenius norm of $M$. We use $M^T$ and $M^{-T}$ to denote the transpose and the inverse of the transpose of the matrix $M$.\\
For $f: \Rn \to \R$, and a differentiable point $x \in \Rn$ of $f$, we use $D_x f$ to denote the gradient in a column vector:
\begin{equation*}
D_x f(x) = \left( \begin{matrix} \frac{\partial}{\partial x_1} f(x_1, \cdots , x_n ) \\ \vdots \\ \frac{\partial}{\partial x_n} f(x_1, \cdots , x_n ) \end{matrix} \right).
\end{equation*}
For a function $c : \Rn \times \Rn$ with variables $(x,y)$, we use $D^2_{xy} c$ to denote the mixed Hessian matrix of $c$:
\begin{equation*}
D^2_{xy} c = \left( \begin{matrix} \frac{\partial^2}{\partial x_1 \partial y_1} c &\cdots & \frac{\partial^2}{\partial x_1 \partial y_n} c \\ \vdots & \ddots \vdots \\ \frac{\partial^2}{\partial x_n \partial y_1} c & \cdots & \frac{\partial^2}{\partial x_n \partial y_n} c \end{matrix} \right).
\end{equation*}
For a set $S \subset \Rn$ and a point $p \in \Rn$, we denote the Euclidean distance between $S$ and $p$ by $\dist{S}{p}$. Also, we use the following notation for the translation of the set $S$ by $p$:
\begin{equation*}
S+p = \{ x+p | x \in S\}.
\end{equation*}
Moreover, for a matrix $M \in \R^{n \times n}$, we use the following notation for the transform of the set $S$ under $M$:
\begin{equation*}
MS = \{ M x | x \in S \}.
\end{equation*}
We use $B_r^n(p)$ to denote the $n$-dimensional ball of radius $r$ centered at $p$. We will omit the superscript that denotes the dimension when there is no confusion. In addition, for a set $S \subset \R^n$, we also use 
\begin{equation*}
    \begin{aligned}
        B_r(S) & = \left\{ p \in \R^n | p \in B_r(x) \textrm{ for some } x \in S \right\} \\
        & = \bigcup_{x \in S} B_r(x).
    \end{aligned}
\end{equation*}

\section{Conditions on the cost function}\label{sec: cost function}
In this section, we establish assumptions for the base spaces and the cost function that we will be using through out the paper, and introduce basic properties.

\subsection{Conditions on the cost function}\label{subsec: conditions}
Let $\X$ and $\Y$ be compact subsets of $\Rn$ with non-empty interior. Then the tangent spaces $T_x\X$ and $T_y \Y$ for any $x \in \X$ and $y \in \Y$ can be identified with $\Rn$ canonically. Henceforth, we use $\Rn$ instead of the tangent spaces $T_x \X$ and $T_y \Y$. We fix a function $c: \X \times \Y \to \R$ and call it cost function. In this section, we impose several conditions on the spaces $\X$ and $\Y$ and the cost function $c$, and prove some useful lemmas which will be used in the later part of the paper. \\
We first impose some regularity assumptions. we assume that the cost function is $C^2(\X \times \Y)$ in the sense that both mixed Hessian $D^2_{xy}c$ and $D^2_{yx}c$ are continuous and $D^2_{xy}c(x,y) = D^2_{yx}c(x,y)$ for any $(x,y) \in \X \times \Y$. Then the compactness of $\X \times \Y$ implies that $Dc = \begin{pmatrix} D_x c \\ D_y c \end{pmatrix}$ is bounded. In particular, $c$ is Lipschitz. 
\begin{equation}\label{eqn: c Lipschitz}
\clip = \sup_{(x_1, y_1) \neq (x_0, y_0)} \frac{| c(x_1, y_1) - c(x_0, y_0) | }{\| (x_1, y_1) - (x_0, y_0) \|} < \infty .
\end{equation}
where $ \| (x_1,y_1) - (x_0,y_0) \| = ( \| x_1 - x_0 \|^2 + \| y_1 - y_0 \|^2)^{1/2}$.\\

Next, we define so-called twist and non-degeneracy conditions on the cost function.

\begin{Def}\label{def: twist}
The cost function $c$ is said to satisfy \emph{twist condition} if for any $x \in \X$, the function
\begin{equation*}
	-D_x c(x, \cdot) : \Y \to \Rn
\end{equation*}
is injective. The cost function $c$ is said to satisfy \emph{twist* condition} if for any $y \in \Y$, the function
\begin{equation*}
	-D_y c(\cdot, y) : \X \to \Rn
\end{equation*}
is injective. We call $c$ \emph{bi-twisted} if $c$ satisfies both twist and twist* conditions.
\end{Def}

\begin{Def}\label{def: non-deg}
The cost function $c$ is said to be \emph{non-degenerate} if, for any $(x,y) \in \X \times \Y$, the mixed Hessian matrix $-D^2_{xy} c(x,y)$ is invertible.
\end{Def}
We assume that the cost function $c$ is bi-twisted and non-degenerate through out this paper. 

\begin{Cond}
The cost function $c$ is bi-twisted and non-degenerate.
\end{Cond}

Compactness of $\X \times \Y$ implies that the mixed Hessian $D^2_{xy} c(x,y)$ of the cost function has a bounded Frobenius norm:
\begin{equation}\label{eqn: alpha'}
	\| D^2_{xy}c(x,y) \|_F < \alpha' < \infty, \quad \forall (x,y) \in \X \times \Y.
\end{equation}
where $\|M \|_F = (\mathrm{Tr}(MM^T))^{1/2}$. The non-degeneracy condition with compactness of $\X$ and $\Y$ implies that there exists $\beta'>0$ such that
\begin{equation*}
	\|D^2_{xy} c(x,y)\|_F \geq \beta'
\end{equation*}
for any $(x,y) \in \X \times \Y$. In addition, the inverse matrix of the mixed Hessian $(D^2_{xy} c(x,y))^{-1}$ is also continuous and invertible. Hence there exists $\alpha''>0$ and $\beta''>0$ such that
\begin{equation*}
	\frac{1}{\alpha''} \leq \| (D^2_{xy} c(x,y))^{-1} \|_F \leq \beta''.
\end{equation*}
Let $\alpha = \max\{\alpha', 1/\alpha''\}$ and $\beta = \max\{1/\beta',\beta'' \}$, then we have
\begin{equation}\label{eqn: alpha beta}
	\frac{1}{\beta} \leq \|D^2_{xy} c(x,y)\|_F \leq{\alpha} \textrm{ and } \frac{1}{\alpha} \leq \| (D^2_{xy} c(x,y))^{-1} \|_F \leq \beta.
\end{equation}
Compactness of $\X \times \Y$ also implies that the mixed Hessian $D^2_{xy} c(x,y)$ is uniformly continuous. Therefore, there exists $\omega: [0, \diam{\X \times \Y}) \to [0,\infty)$, the modulus of continuity of $D^2_{xy} c$. i.e. there exists $\omega$  that is non-increasing, $\lim_{\rho \to 0} \omega(\rho) = 0$, and
\begin{equation}\label{eqn: mod. of conti. Hessian}
\| D^2_{xy}c(x_1,y_1)-D^2_{xy}c(x_0,y_0) \|_F \leq \omega(\rho ), \textrm{ for any } \| (x_1, y_1) - (x_0, y_0) \| \leq \rho,
\end{equation}
Using this with \eqref{eqn: alpha beta}, we can obtain a modulus of continuity for $(D^2_{xy} c(x,y))^{-1}$.

\begin{Lem}\label{lem: mod. of conti. inverse Hessian}
Let $x_0, x_1 \in \X$ and $y_0, y_1 \in \Y$. If $\| (x_0, y_0) - (x_1, y_1) \| \leq \rho$, then
\begin{equation}\label{eqn: mod. of conti. inverse Hessian}
	\| (D^2_{xy} c(x_0,y_0))^{-1} - (D^2_{xy} c(x_1,y_1))^{-1} \|_F \leq \beta^2 \omega(\rho).
\end{equation}
\end{Lem}
\begin{proof}
\begin{align*}
	&\|(D^2_{xy} c(x_0,y_0))^{-1} - (D^2_{xy} c(x_1,y_1))^{-1}\|_F  \\
	= & \|(D^2_{xy} c(x_0,y_0))^{-1}  (D^2_{xy} c(x_1,y_1))^{-1} ( D^2_{xy} c(x_1,y_1) - D^2_{xy} c(x_0,y_0)) \|_F \\
	\leq & \|(D^2_{xy} c(x_0,y_0))^{-1}\|_F \| (D^2_{xy} c(x_1,y_1))^{-1} \|_F \| (D^2_{xy} c(x_1,y_1)) - (D^2_{xy} c(x_0,y_0)) \|_F \\
	\leq & \beta^2 \omega(\rho)
\end{align*}
where we have used \eqref{eqn: mod. of conti. Hessian} and \eqref{eqn: alpha beta} for the last inequality.
\end{proof}

For any subset $A \subset \X$, we denote its image under the map $-D_y c(\cdot,y ) : \Y \to \Rn$ by $[A]_{y}$, i.e. $[A]_y = -D_y c(A,y)$. Similarly, for any subset $B \subset \Y$, we define $[B]_x = -D_x c(x,B)$. 

\begin{Def}
Let $A \subset \X$ and $y \in \Y$. $A$ is called \emph{$c$-convex with respect to $y$} if $[A]_{y}$ is convex. If $A$ is $c$-convex with respect to $y$ for any $y \in \Y$, then we say that $A$ is \emph{$c$-convex with respect to $\Y$}. Similarly, for $B \subset \Y$ and $x \in \X$, $B$ is said to be \emph{$c$-convex with respect to $x$} if $[B]_x$ is convex. If $B$ is $c$-convex with respect to $x$ for any $x \in \X$, then we say that $B$ is \emph{$c$-convex with respect to $\X$}.
\end{Def}

We assume that $\X$ and $\Y$ are $c$-convex with respect to $\Y$ and $\X$ respectively for the rest of this paper.

\begin{Cond}
$\X$ is $c$-convex with respect to $\Y$ and $\Y$ is $c$-convex with respect to $\X$.
\end{Cond}

Convexity of $[\X]_y$ and $[\Y]_x$ implies that there exists a (outward) normal cone at any points on the boundary of $[\X]_y$ and $[\Y]_x$. We define the notations for the cones in the following definition.

\begin{Def}\label{def: cones}
Let $\theta \in [0, \frac{\pi}{2}]$ and let $w \in \Rn$. We denote the cone with opening $\theta$ and axis $w$ by $\K_{\theta, w}$:
\begin{equation*}
	\K_{\theta, w} = \{ v \in \Rn | \langle v, w \rangle \geq \cos(\theta) \| v \| \| w \| \}
\end{equation*}
and we also denote $\K_{\theta, w}^\rho = \K_{\theta,w} \cap \overline{B_\rho (0)}$. For $S \subset \Rn$, we define \emph{the dual cone of $S$} by 
\begin{equation*}
	\K^*(S) = \{ v \in \Rn | \langle v, p \rangle \geq 0, \forall p \in S\}.
\end{equation*}
If $S\subset \Rn$ is convex, then for any $p \in \partial S$, we denote \emph{the outward normal cone of $S$ at $p$} by $\N(S;p)$:
\begin{equation*}
	\N(S;p) = -\K^*(S-p).
\end{equation*}
\end{Def}

The bi-twist condition on the cost function implies that we can define inverse functions of $-D_y c(\cdot,y ) : \X \to [\X]_y$ and $-D_x c(x, \cdot) : \Y \to [\Y]_x$, which we call $c$-exponential maps. 

\begin{Def}\label{def: cexp}
The inverse map of $-D_x c(x, \cdot ) : \Y \to [\Y]_{x}$ is called the \emph{$c$-exponential map focused at $x$}, and denoted by
\begin{equation*}
	\cexp{x}{\cdot}: [\Y]_{x} \to \Y.
\end{equation*}
We also call the inverse map of $-D_y c(\cdot, y): \X \to [\X]_{y}$ the \emph{$c$-exponential map focused at $y$}, and denote
\begin{equation*}
	\cexp{y}{\cdot}: [\X]_{y} \to \X.
\end{equation*}
\end{Def}

On the other hand, the non-degeneracy condition implies that the $c$-exponential maps $\cexp{x}{\cdot}$ and $\cexp{y}{\cdot}$ are differentiable, and
\begin{equation*}
	D_q \cexp{x}{q} = (-D^2_{xy}c(x,\cexp{x}{q}))^{-1}, \quad D_p \cexp{y}{p} = (-D^2_{yx}c(\cexp{y}{p}, y))^{-1}.
\end{equation*}
Then compactness of $\X$ and $\Y$ and continuity of $-D_x c(x, \cdot)$, $-D_y c(\cdot, y)$, $\cexp{x}{\cdot}$ and $\cexp{y}{\cdot}$ imply that these functions are closed maps. Also, the non-degeneracy implies that these functions are open maps. Moreover, \eqref{eqn: alpha beta} with compactness of $\X$ and $\Y$ implies that the maps $-D_x c(x, \cdot)$, $-D_y c(\cdot, y)$, $\cexp{x}{\cdot}$ and $\cexp{y}{\cdot}$ are bi-Lipschitz. We use $L$ to denote the bi-Lipschitz constant for these functions:
\begin{equation}\label{eqn: Lipschitz}
	\begin{aligned}
	\frac{1}{L} \| x_1 - x_0 \| \leq \| -D_y c(x_1, y) +D_y c(x_0, y) \| \leq L \| x_1 - x_0 \|, \\
	\frac{1}{L} \| y_1 - y_0 \| \leq \| -D_x c(x, y_1) +D_y c(x, y_0) \| \leq L \| y_1 - y_0 \|, \\
	\frac{1}{L} \| p_1 - p_0 \| \leq \| \cexp{y}{p_1} - \cexp{y}{p_0} \| \leq L \| p_1 - p_0 \|, \\
	\frac{1}{L} \| q_1 - q_0 \| \leq \| \cexp{x}{q_1} - \cexp{x}{q_0} \| \leq L \| q_1 - q_0 \|.
\end{aligned}
\end{equation}

\begin{Def}\label{def: c-segment}
Let $x_0, x_1 \in \X$ and $y \in \Y$. Let $p_i = -D_y c(x_i,y)$. We define \emph{$c$-segment with respect to $y$ from $x_0$ to $x_1$} by
\begin{equation*}
	\cexp{y}{[p_0,p_1]} = \{ \cexp{y}{p_t} | p_t = tp_1 + (1-t)p_0, t \in [0,1]\}.
\end{equation*}
We denote the $c$-segment with respect to $y$ from $x_0$ to $x_1$ by $\cseg{y}[x_0,x_1]$. Similarly, for $y_0, y_1 \in \Y$, $x \in \X$, and $q_i = -D_x c(x,y_i)$, we define \emph{$c$-segment with respect to $x$ from $y_0$ to $y_1$} by
\begin{equation*}
	\cexp{x}{[q_0, q_1]} =\{ \cexp{x}{q_t} | q_t = t q_1 + (1-t)q_0, t \in [0,1]\}.
\end{equation*}
We denote the $c$-segment with respect to $x$ from $y_0$ to $y_1$ by $\cseg{x}[y_0,y_1]$.
\end{Def}

\begin{Rmk}\label{rmk: c-segment notation}
We will also use notations like $\cseg{y}(x_0,x_1)$ and $\cseg{y}(x_0,x_1]$ to denote the $c$-segments with or without boundary points. Moreover, we abuse the notation $x_t \in \cseg{y}[x_0,x_1]$ to denote that $x_t$ is not only a point on the $c$-segment with respect to $y$ from $x_0$ to $x_1$, but also
\begin{equation*}
	x_t = \cexp{y}{tp_1 + (1-t)p_0},
\end{equation*}
where $p_i =-D_y c(x_i, y)$ for $i = 0,1$. 
\end{Rmk}

An important condition for the regularity theory of optimal transportation problem is about a tensor so-called MTW tensor. The MTW tensor is defined by the following formula:
\begin{equation*}
	\MTW = D^2_{qq} \mathcal{A}(x,q) = D^2_{qq} ( -D^2_{xx} c(x, \cexp{x}{q})).
\end{equation*}
In coordinate, the MTW tensor has the form
\begin{equation}\label{eqn: MTW tensor}
	\MTW_{ijkl} = ( c_{ij,p}c^{p,q}c_{q,st} -c_{ij,st})c^{s,k}c^{t,l}
\end{equation}
where the indices before comma indicates the derivatives in $x$ variable, and after comma indicates the derivatives $y$ variables, and $(c^{p,q})$ is the inverse matrix of $(c_{p,q}) = D^2_{xy} c$. The condition on the MTW tensor, which we will call the \emph{MTW condition} is a non-negativity on certain directions.

\begin{Def}\label{def: MTW condition}
The cost function $c$ is said to satisfy \emph{(strong) MTW condition} if it is $C^4$ and, for any $y \in \Y$, $p \in [\X]_y$, and any $\eta, \xi \in \Rn\setminus\{0\}$ such that $\eta \perp \xi$, the MTW tensor at $(y,p)$ satisfies
\begin{equation}\label{eqn: MTW condition}
	\MTW[\eta, \eta, \xi, \xi] \geq (>) 0.
\end{equation}
\end{Def}

To define the MTW tensor, we need $C^4$ regularity on the cost function $c$. However, there are synthetic expressions of the MTW condition that requires less regularity to formulate, and equivalent to the MTW condition when the cost function is $C^4$. 

\begin{Def}\label{def: Loeper}
The cost function $c$ is said to satisfy \emph{Loeper's property} if for any $x_0, x_1 \in \X$, $y_0, y \in \Y$ and $x_t \in \cseg{y_0}[x_0,x_1]$, we have
\begin{equation}\label{eqn: Loeper}
	-c(x_t, y) +c(x_t, y_0) \leq \max\{ -c(x_i, y) +c(x_i,y_0) | i = 0,1 \}. 
\end{equation} 
\end{Def}

\begin{Def}\label{def: QQconv}
The cost function $c$ is said to be quantitatively quasi-convex (QQconv) if there exists a constant $C\geq 1$ such that for any $x_0, x_1 \in \X$, $y_0, y \in \Y$ and $x_t \in \cseg{y_0}[x_0,x_1]$, we have
\begin{equation}\label{eqn: QQconv}
\begin{aligned}
	&-c(x_t, y) +c(x_t, y_0) + c(x_0, y)-c(x_0,y_0) \\
	& \leq Ct \max\{-c(x_1, y) +c(x_1, y_0) + c(x_0, y)-c(x_0,y_0), 0 \}.  
\end{aligned}
\end{equation}
\end{Def}

\begin{Rmk}
Due to the symmetry of the indices in \eqref{eqn: MTW tensor}, we can change the roles of $x$ and $y$ in \eqref{eqn: Loeper} and \eqref{eqn: QQconv}:
\begin{equation*}
	-c(x, y_t) + c(x_0, y_t) \leq \max \{ -c(x, y_i) + c(x_0, y_i) | i = 0,1 \},
\end{equation*}
and
\begin{align*}
	&-c(x,y_t) + c(x_0, y_t) + c(x, y_0) - c(x_0, y_0) \\
	&\leq Ct \max\{ -c(x_1, y) +c(x_0, y) + c(x_1, y_0) -c (x_0,y_0) , 0 \}.
\end{align*}
\end{Rmk}

There are other discussions about the meaning of the MTW condition (for example, see \cite{KM10},\cite{LT07} \cite{KZ20}, \cite{Ran23}). In this paper, we establish another synthetic expression of the MTW condition as stated in the main theorem Theorem \ref{thm: main theorem}. Note that we do not have to use the original formulation of the MTW condition \eqref{eqn: MTW condition} which requires $C^4$ regularity, but we can use Loeper's property \eqref{eqn: Loeper} or QQconv \eqref{eqn: QQconv} where $C^2$ regularity is enough for the formulation. Therefore, we will show that the new synthetic expression is equivalent to Loeper's property \eqref{eqn: Loeper}. 

\begin{Rmk}
Loeper's property is not assumed in the whole paper, but is assumed in several lemmas. Hence, we will state that the cost function $c$ satisfies Loeper's property in the statement of lemma or theorem whenever we use Loeper's property.
\end{Rmk}

The constants $\clip, \alpha, \beta, L$, and the modulus of continuity $\omega$ are decided by the fixed cost function $c$, independent of whether the MTW condition holds. Also, the sets $\X$ and $\Y$ are also fixed through out the paper, and therefore, the constant decided by the sets $\X$ and $\Y$, such as $\diam{\X}$ and $\diam{\Y}$, are also fixed. We shall use the term '\emph{universal}' to denote the constants which are only decided by these quantities and $\omega$. 

\subsection{$c$-affine functions}\label{subsec: c-affine functions}
In this subsection, we define the $c$-affine functions and the $c$-convex functions, and provide basic properties of them.

\begin{Def}\label{def: c-affine c-conv}
Let $y \in \Y$ and $h \in \R$. We call functions $f:\X \to \R$ of the following form \emph{$c$-affine functions}:
    \begin{equation*}
        f(x) = -c(x,y)+h.
    \end{equation*} 
\end{Def}

$c$-affine functions can be regarded as basic elements in the construction of the $c$-convex functions. As such, we will need some basic analysis of the relation between one $c$-affine function and another. 

\begin{Def}\label{def: section}
Let $\phi: \X \to \R$ be a function and $f:\X \to \R$ be a $c$-affine function:
    \begin{equation*}
        f(x) = -c(x,y) + h.
    \end{equation*}
\emph{The section of $\phi$ with respect to $f$} is the set where $\phi$ is smaller than or equal to $f$. We denote the section by $\csec{\phi}{f}$:
    \begin{equation}
	   \csec{\phi}{f} = \{ x \in \X | \phi(x) \leq f (x) \}.
    \end{equation}
\end{Def}

We present some basic properties of the sections of $c$-affine function with respect to another $c$-affine function in this subsection.

\begin{Lem}\label{lem: section non empty interior}
Let $x' \in \X$ and $y_0, y_1 \in \Y$, $y_0 \neq y_1$. Also, let
    \begin{equation*}
        f_i(x) = -c(x,y_i) + c(x', y_i).
    \end{equation*}
If $x' \in \Int{\X}$, then $\csec{f_{0}}{f_{1}}$ has non-empty interior.
\end{Lem}
\begin{proof}
    Suppose $x' \in \Int{\X}$. Let 
    \begin{equation*}
        v = D_x f_{1}(x') - D_x f_{0}(x') = -D_x c (x',y_1) + D_x c(x', y_0).
    \end{equation*} 
    The bi-twist condition implies $v \neq 0$. Moreover, $x' \in \Int{\X}$ implies that there exists $\delta >0$ such that $x' + t v \in \X$ for any $t \in (-\delta, \delta)$. Hence, at $t=0$, we compute
    \begin{equation*}
	   \frac{d}{dt}\bigg|_{t=0} (f_{1}(x' + t v)-f_{0}(x' + t v))  = \| v \|^2 >0.
    \end{equation*}
    This implies that for some small enough $\delta'>0$, we have $f_{1}(x' + t v)-f_{0}(x' + t v) >0$ for any $t \in (0,\delta')$. Fix $t_0 \in (0, \delta')$. Using continuity of $c$, we obtain that there exists $r>0$ such that $f_{1} (x) > f_{0}(x)$ for $x \in B_r(x' + t_0v)$. Hence $B_r(x' + tv) \subset \csec{f_{0}}{f_{1}}$ and $\Int{\csec{f_{0}}{f_{1}}} \neq \emptyset$.
\end{proof}

\begin{Lem}\label{lem: c-affine section bdy}
Let $x' \in \X$, $y_0, y_1 \in \Y$, $f_{0}, f_{1}$ be from Lemma \ref{lem: section non empty interior}. If $f_{0}(x_0) = f_{1}(x_0)$ for some $x_0 \in \X$, then $x_0 \in \partial \csec{f_{0}}{f_{1}}$.
\end{Lem}
\begin{proof}
Suppose $x_0 \in \Int{\csec{f_{0}}{f_{1}}} \subset \Int{\X}$. Let 
\begin{equation*}
v = D_x f_{0}(x_0) - D_x f_{1}(x_0) = -D_x c (x_0,y_0) + D_x c(x_0, y_1).
\end{equation*}
Then the proof of Lemma \ref{lem: section non empty interior} using $x_0$ instead of $x'$ and switching the role of $y_0$ and $y_1$ shows that there exists $\delta >0$ such that 
\begin{equation*}
	f_{0}(x_0+tv) > f_{1} (x_0+tv) \textrm{ for any } t \in (0,\delta).
\end{equation*}
On the other hand, $x_0 \in \Int{\csec{f_{0}}{f_{1}}}$ implies that there exists $r>0$ such that $B_r(x_0) \subset \csec{f_{0}}{f_{1}}$. Then, choosing $t_0 \in (0, \min\{\delta, r\})$, we obtain
\begin{equation*}
	x_0 + t_0 v \in \csec{f_{0}}{f_{1}} \quad \textrm{and} \quad f_{0}(x_0+t_0 v) > f_{1} (x_0+t_0v),
\end{equation*}
which contradicts to the definition of $\csec{f_{0}}{f_{1}}$. Therefore, we must have $x_0 \in \partial \csec{f_{0}}{f_{1}}$.
\end{proof}

\begin{Cor}\label{lem: boundary and interior of c-affine section}
Let $x' \in \X$, $y_0, y_1 \in \Y$, $f_{0}, f_{1}$, be from Lemma \ref{lem: section non empty interior}. Then we have
\begin{align*}
	\Int{\csec{f_{0}}{f_{1}}} & \subset \{ x \in \X | f_{0}(x) < f_{1}(x)\},\\
	\partial \csec{f_{0}}{f_{1}} & \supset \{ x \in \X | f_{0}(x) = f_{1}(x) \}.
\end{align*}
\end{Cor}
\begin{proof}
    Lemma \ref{lem: c-affine section bdy} implies 
        \begin{equation*}
	       \{ x \in \X | f_{0}(x) = f_{1}(x) \} \subset \partial \csec{f_{0}}{f_{1}},
        \end{equation*}
    which is the second part of the corollary. Then, the first part is given by
        \begin{align*}
	       \Int{\csec{f_{0}}{f_{1}}} & = \csec{f_{0}}{f_{1}} \setminus \partial \csec{f_{0}}{f_{1}} \\
	       & \subset \csec{f_{0}}{f_{1}} \setminus \{ x \in \X | f_{0}(x) = f_{1}(x) \} \\
	       & = \{ x \in \X | f_{0}(x) < f_{1}(x)\}.
        \end{align*}
\end{proof}

In the next lemma, we find a cone that is in $ \left[ \csec{f_{0}}{f_{1}} \right]_{y_0}$ with the vertex in $\{ x \in \X | f_1(x) = f_0(x)\}$.

\begin{Lem}\label{lem: cone in c-affine section}
We use the notations from Lemma \ref{lem: section non empty interior}. Let 
\begin{equation*}
	x_0 \in \{ x \in \X | f_{0}(x) = f_{1}(x) \} \subset \partial \csec{f_{0}}{f_{1}}.
\end{equation*}
Also, let $p_0 = -D_y c(x_0, y_0)$, $q_i = -D_x c(x_0, y_i)$ for $i = 0,1$ and $M_0 = (-D^2_{xy} c(x_0,y_0))$. Define $v \in \Rn$ by
\begin{equation}\label{eqn: cone in section: w}
	v = M_0^{-1} (q_1 - q_0) .
\end{equation}
Note that $v \neq 0$ by bi-twist condition. Suppose that there exist constants $l_0, l_1 >0$ such that
\begin{equation*}
l_0 \leq \| q_1 - q_0 \| \leq l_1.
\end{equation*}
Fix $\theta \in \left(0,\frac{\pi}{2}\right)$ and let $\rho>0$ satisfy
\begin{equation}\label{eqn: radius of cone in section}
	2\beta L^2 \rho + \beta^2 l_1 \omega(L\rho) \leq \frac{l_0}{\alpha}  \cos(\theta).
\end{equation}
Then we have
\begin{equation}\label{eqn: cone in c-affine section}
	(\K_{\theta,v}^\rho+p_0) \cap [\X ]_{y_0} \subset \left[ \csec{f_{0}}{f_{1}} \right]_{y_0}.
\end{equation}
\end{Lem}
\begin{proof}
    Note first that $f_{0}(x_0) = f_{1}(x_0)$ implies 
        \begin{align*}
	    & -c(x_0, y_0) +c(x' y_0) = -c(x_0, y_1) +c(x', y_1) \\
	       \Leftrightarrow & -c(x', y_1) +c(x', y_0) = -c(x_0, y_1) +c(x_0, y_0).
        \end{align*}
    Therefore
        \begin{align*}
	       & -c(x, y_1) + c(x', y_1) \geq -c(x, y_0) +c(x', y_0) \\
	       \Leftrightarrow & -c(x, y_1) +c(x, y_0) \geq -c(x', y_1) +c(x', y_0) \\
	       \Leftrightarrow & -c(x, y_1) +c(x, y_0) \geq -c(x_0, y_1) +c(x_0, y_0) \\
	       \Leftrightarrow & -c(x, y_1) + c(x_0, y_1) \geq -c(x, y_0) +c(x_0, y_0),
        \end{align*}
    so that
        \begin{equation*}
	       \csec{f_{0}}{f_{1}} = \{ x \in \X | -c(x, y_1) + c(x_0, y_1) \geq -c(x, y_0) +c(x_0, y_0)\}.
        \end{equation*}
    Let $p_1 \in (\K_{\theta,v}^\rho+p_0) \cap [\X ]_{y_0}$ and $x_1 = \cexp{y_0}{p_1}$. Also, let $x_t \in \cseg{y_0}[x_0, x_1]$. Then we compute
        \begin{equation}\label{eqn: difference of c affine}
            \begin{aligned}
	               & -c(x_1, y_1) + c(x_0, y_1) + c(x_1, y_0) - c(x_0, y_0) \\
	               =& \int_0^1 \left\langle (-D^2_{xy} c(x_t, y_0) )^{-1} (-D_x c(x_t, y_1)), p_1 - p_0 \right\rangle dt \\
	               & - \int_0^1 \left\langle (-D^2_{xy} c(x_t, y_0) )^{-1} (-D_x c(x_t, y_0)), p_1 - p_0 \right\rangle dt \\
	               = & \int_0^1 \left\langle (-D^2_{xy} c(x_t, y_0))^{-1} (-D_x c(x_t, y_1)+D_x c(x_t, y_0) ), p_1 - p_0 \right\rangle dt 
            \end{aligned}
        \end{equation}
    To continue from \eqref{eqn: difference of c affine}, we estimate the terms with $x_t$. Thanks to \eqref{eqn: Lipschitz}, we have
    \begin{equation*}
	   \| -D_x c(x_t, y_i) +D_x c(x_0, y_i) \| \leq L \| x_t - x_0 \| ,
    \end{equation*}
    for $i = 0,1$. Therefore,
    \begin{equation}\label{eqn: Dxc(xt)-w est in cone in section lem}
        \begin{aligned}
	       & \| -D_x c(x_t, y_1)+D_x c(x_t, y_0) -(q_1-q_0) \| \\
	       \leq & \| -D_x c(x_t, y_1) + D_x c(x_0, y_1) \| + \| -D_x c(x_t, y_0) + D_x c(x_0,y_0) \| \\
	       \leq & 2L\| x_t - x_0 \| \\
	       \leq & 2L^2 \| (tp_1 + (1-t)p_0) - p_0 \| \\
	       \leq & 2L^2 \rho.
        \end{aligned}
    \end{equation}
    where we have used $\| p_1 - p_0 \| \leq \rho$ for the last inequality. In addition, \eqref{eqn: mod. of conti. inverse Hessian} yields
    \begin{equation}\label{eqn: hessian difference est in cone in section lem}
	   \|(-D^2_{xy} c(x_t, y_0))^{-1} - (-D^2_{xy} c(x_0, y_0))^{-1}\| \leq \beta^2 \omega(\| x_t - x_0 \|) \leq \beta^2  \omega(L\rho).
    \end{equation}
    Therefore, continuing from \eqref{eqn: difference of c affine},
    \begin{equation}\label{eqn: difference of c affine 2}
        \begin{aligned}
	       & \int_0^1 \left\langle (-D^2_{xy} c(x_t, y_0))^{-1} (-D_x c(x_t, y_1)+D_x c(x_t, y_0) ), p_1 - p_0 \right\rangle dt \\
	       = & \int_0^1 \left\langle (-D^2_{xy} c(x_t, y_0))^{-1} (-D_x c(x_t, y_1)+D_x c(x_t, y_0) -(q_1 - q_0) ), p_1 - p_0 \right\rangle dt \\
	       & + \int_0^1 \left\langle ((-D^2_{xy} c(x_t, y_0))^{-1}-(-D^2_{xy} c(x_0, y_0))^{-1}) (q_1 - q_0), p_1 - p_0 \right\rangle dt \\
	       & + \int_0^1 \left\langle M_0^{-1}(q_1 - q_0) , p_1 - p_0 \right\rangle dt \\
	       \geq & -2 \beta L^2 \rho \| p_1 - p_0 \| - \beta^2 \omega(L\rho) \| q_1 - q_0 \| \| p_1 - p_0 \| \\
	       & + \cos (\theta) \| M_0^{-1}(q_1 - q_0) \| \| p_1 - p_0 \|
        \end{aligned}
    \end{equation}
    where we have used \eqref{eqn: Dxc(xt)-w est in cone in section lem}, \eqref{eqn: hessian difference est in cone in section lem}, and $p_1 \in \K_{\theta,v}^\rho+p_0$ to obtain the inequality above. Note that
        \begin{equation*}
	       \frac{l_0}{\alpha} \leq \frac{1}{\alpha} \| q_1 - q_0 \| \leq \| M_0^{-1}(q_1 - q_0) \|.
        \end{equation*}
    Therefore, 
        \begin{equation*}
            \begin{aligned}
                2 \beta L^2 \rho + \beta^2 \omega(\rho) \| q_1 - q_0 \| &\leq 2\beta L^2 \rho + \beta^2 l_1 \omega(\rho)  \\
	           & \leq \frac{l_0}{\alpha} \cos(\theta) \\
	           & \leq \cos(\theta) \| M_0^{-1}(q_1 - q_0) \|.
            \end{aligned}
        \end{equation*}
    where we have used \eqref{eqn: radius of cone in section} for the second line above. Thus \eqref{eqn: difference of c affine 2} is non-negative, and we combine \eqref{eqn: difference of c affine} to obtain
        \begin{equation*}
	       -c(x_1, y_1) + c(x_0, y_1) + c(x_1, y_0) - c(x_0, y_0) \geq 0.
        \end{equation*}
    In particular, $x_1 \in  \csec{f_{0}}{f_{1}}$, and $p_1 \in \left[ \csec{f_{0}}{f_{1}} \right]_{y_0}$. Since $p_1 \in (\K_{\theta,v}^\rho+p_0) \cap [\X ]_{y_0}$ is arbitrary, we obtain \eqref{eqn: cone in c-affine section}.
\end{proof}

\begin{Rmk}\label{rmk: dependency of rho}
We note here that the $\rho$ from Lemma \ref{lem: cone in c-affine section} depends on the universal quantities, $\theta$ and $l_0, l_1$, which are the bounds for $\| q_1 - q_0 \|$. We can choose $l_0$ and $l_1$ that only depends on $L$ and $\| y_1 - y_0 \|$ using \eqref{eqn: Lipschitz}:
    \begin{equation*}
        \frac{1}{L} \| y_1 - y_0 \| \leq \| q_1 - q_0 \| \leq L \| y_1 - y_0 \|.
    \end{equation*}
In particular, we can choose $\rho$ that does not depend on $x_0$ in Lemma \ref{lem: cone in c-affine section}.
\end{Rmk}

For a fixed angle $\theta$, if $y_1$ is clsoe enough to $y_0$, then we can find a cone in the section $\left[\csec{f_{0}}{f_{1}}\right]_{y_0}$ with opening angle $\theta$, but without bound on the radius. 

\begin{Lem}\label{lem: small tilting}
    We use notations from Lemma \ref{lem: section non empty interior}. Fix $\theta \in (0, \frac{\pi}{2})$ and 
    \begin{equation*}
        x_0 \in \{ x \in \X | f_{0} (x) = f_{1}(x)\}\subset \partial \csec{f_{0}}{f_{1}}.
    \end{equation*}
    Let $x_1 \in \X$ and let $p_i = -D_y c(x_i, y_0)$, $q_i = -D_x c(x_0, y_i)$, for $i=0,1$, and $M_0 = -D^2_{xy}c(x_0, y_0)$. Denote $v = M_0^{-1} (q_1 - q_0)$. Then there exists a constant $ \mu(\theta)$ such that if 
    \begin{equation*}
        \| v \| \leq \mu(\theta),
    \end{equation*}
    then we have
    \begin{equation*}
        \begin{aligned}
            &p_1 \in (\K_{\theta, v}+p_0) \cap [\X]_{y_0} \\ 
            \Rightarrow& \frac{1}{2} \| p_1 - p_0 \| \| v \| \cos(\theta)  \leq f_{1}(x_1) - f_{0}(x_1) \leq \frac{3}{2} \| p_1 - p_0 \| \| v \| \cos(\theta),            
        \end{aligned}
    \end{equation*}
    and
    \begin{equation*}
        \begin{aligned}
            &p_1 \in (\K_{\theta, -v}+p_0) \cap [\X]_{y_0} \\
            \Rightarrow& -\frac{3}{2}\| p_1 - p_0 \| \| v \| \cos(\theta) \leq  f_{1}(x_1) - f_{0}(x_1) \leq - \frac{1}{2}\| p_1 - p_0 \| \| v \| \cos(\theta).
        \end{aligned}
    \end{equation*}
\end{Lem}
\begin{proof}
    Define $\C: [\X]_{y_0} \to \R $ by 
    \begin{equation*}
        \C(p,y) = -c(\cexp{y_0}{p}, y) + c(x_0, y) .
    \end{equation*}
    Note that we have $\C(p,y_1)-\C(p,y_0) = f_{1}(\cexp{y_0}{p}) - f_{0}(\cexp{y_0}{p})$. Let $p_1 \in [\X]_{y_0}$. Then 
    \begin{equation}\label{eqn: Cy1-Cy0}
        \begin{aligned}
            \frac{d}{dt}\C(p_1, y_t) & =  \langle D_y \C(p_1, y_t), \frac{d}{dt}y_t \rangle  \\
            & =  \langle D_y \C(p_1,y_t), M_t^{-1}(q_1 - q_0) \rangle
        \end{aligned}
    \end{equation}
    where $y_t \in \cseg{x_0}[y_0,y_1]$ and $M_t = -D^2_{xy}c(x_0, y_t)$. Letting $p_{i,t} = -D_y c(x_i, y_t) $ for $i=0,1$, Then the last line of \eqref{eqn: Cy1-Cy0} is
    \begin{equation*}
         \langle p_{1,t}-p_{0,t}, M_t^{-1}(q_1 - q_0) \rangle.
    \end{equation*} 
    We estimate this using $p_1, p_0$ and $M_0$. We first evaluate the above formula at $t=0$.
    \begin{equation}\label{eqn: dt Cyt t=0}
        \frac{d}{dt}\C(p,y_t)\bigg|_{t=0} = \langle p_1 - p_0, v \rangle. 
    \end{equation}
    Next, we compute
    \begin{equation*}
        \begin{aligned}
            & (p_{1,t} - p_{0,t}) - (p_{1}-p_{0}) \\
            = &  \int_0^1 -D^2_{yx} c(x_s, y_t)(-D^2_{yx}c(x_s, y_0))^{-1}(p_1 - p_0)ds - (p_1 - p_0) \\
            = & \int_0^1 (D^2_{yx}c(x_s, y_0)-D^2_{yx}c(x_s, y_t))(-D^2_{yx}c(x_s, y_0))^{-1}(p_1 - p_0) ds,
        \end{aligned}
    \end{equation*}
    where $x_s \in \cseg{y_0}[x_0, x_1]$. Hence, using \eqref{eqn: alpha beta} and \eqref{eqn: mod. of conti. Hessian}, and noting $\| y_t - y_0 \| \leq Lt\| q_1 - q_0 \| \leq L \| q_1 - q_0 \|$ by \eqref{eqn: Lipschitz}, we obtain 
    \begin{equation*}
        \| (p_{1,t}-p_{0,t}) - (p_1 - p_0) \| \leq \beta \omega(L\| q_1 - q_0 \| )\|p_1 - p_0\|.
    \end{equation*}
    Then using \eqref{eqn: mod. of conti. inverse Hessian} and \eqref{eqn: alpha beta}, we compute
    \begin{equation}\label{eqn: Cy1-Cyt est 1}
        \begin{aligned}
            & \left| \langle p_{1,t} - p_{0,t}, M_t^{-1}(q_1 - q_0) \rangle - \langle p_1 - p_0, M_0^{-1}(q_1 - q_0) \rangle \right| \\
            \leq & \left| \langle p_{1,t} - p_{0,t}, M_t^{-1}(q_1 - q_0) \rangle - \langle p_1 - p_0, M_t^{-1} (q_1 - q_0) \rangle \right| \\
            & + \left| \langle p_1 - p_0, M_t^{-1} (q_1 - q_0) \rangle - \langle p_1 - p_0 , M_0^{-1} (q_1 - q_0) \rangle \right| \\
            \leq & \| (p_{1,t} - p_{0,t}) - (p_1 - p_0) \| \| M_t^{-1} \|_F \| q_1 - q_0 \| \\
            & + \| p_1 - p_0 \| \| M_t^{-1} - M_0^{-1} \|_F \| q_1 - q_0 \| \\
            \leq & \beta \omega(L\| q_1 - q_0 \|) \|p_1 - p_0\| \times \beta \times \|q_1 - q_0 \| \\
            & + \|p_1 - p_0 \| \times \beta^2\omega(L\| q_1 - q_0 \|) \times \| q_1 - q_0 \| \\
            = & 2\beta^2 \| p_1 - p_0 \| \|q_1 - q_0 \| \omega (L\|q_1 - q_0 \|).
        \end{aligned}
    \end{equation}
    We fix $\rho_\theta>0$ that satisfies
    \begin{equation}\label{eqn: rho th}
        \omega(\rho_\theta) \leq \frac{\cos(\theta)}{4\alpha\beta^2}.
    \end{equation}
    Note that $\theta\in (0, \frac{\pi}{2})$ so that $\cos(\theta)>0$, and therefore there exists $\rho_\theta$ satisfying \eqref{eqn: rho th}. Also, $\rho_{\theta}$ depends only on $\theta$ and the universal constants. Let 
    \begin{equation*}
        \mu(\theta) = \frac{\rho_\theta}{L\alpha}.
    \end{equation*}
    Then $\| v \| \leq \mu(\theta)$ yields
    \begin{equation*}
        \omega(L\| q_1 -q_0 \| ) = \omega(L \| M_0 v\| ) \leq \omega (\rho_{\theta}) \leq \frac{\cos(\theta)}{4\alpha\beta^2},
    \end{equation*}
    so that
    \begin{equation}\label{eqn: Cy1-Cyt est 2}
        \begin{aligned}
            2\beta^2 \| p_1 - p_0 \| \|q_1 - q_0 \| \omega (L\|q_1 - q_0 \|) &\leq 2\beta^2 \| p_1 - p_0 \| \| M_0 \| \| v \| \frac{\cos(\theta)}{4\alpha\beta^2} \\
            & \leq \frac{1}{2}\|p_1 - p_0\| \|v\| \cos(\theta).
        \end{aligned}
    \end{equation}
    Combining \eqref{eqn: Cy1-Cyt est 1} and \eqref{eqn: Cy1-Cyt est 2}, we observe
    \begin{equation*}
        |\langle p_{1,t}-p_{0,t}, M_t^{-1}(q_1-q_0)\rangle  - \langle p_1 - p_0, M_0^{-1} (q_1 - q_0) \rangle| \leq \frac{1}{2} \| p_1 - p_0 \|\| v \|.
    \end{equation*}
    Therefore, if $p_1 \in (\K_{\theta, v} +p_0)\cap [\X]_{y_0}$, then we obtain
    \begin{equation*}
        \begin{aligned}
             \frac{d}{dt} \C(p_1,y_t)  =& \langle p_{1,t} - p_{0,t}, M_t^{-1}(q_1-q_0) \rangle \\
             \geq& \langle p_1 - p_0 , M_0^{-1} (q_1 - q_0) \rangle \\
             &- |\langle p_{1,t}-p_{0,t}, M_t^{-1}(q_1-q_0)\rangle  - \langle p_1 - p_0, M_0^{-1} (q_1 - q_0) \rangle| \\
             \geq& \| p_1 - p_0 \| \|v\| \cos(\theta) - \frac{1}{2} \| p_1 - p_0 \|\| v \| \cos(\theta) \\
             =& \frac{1}{2} \| p_1 - p_0 \| \| v \| \cos(\theta)
        \end{aligned}
    \end{equation*}
    so that
    \begin{equation*}
        \C(p_1, y_1) - \C(p_1, y_0) = \int_0^1 \frac{d}{dt}\C(p_1, y_t) dt \geq \frac{1}{2}\| p_1 - p_0 \| \|v \| \cos(\theta).
    \end{equation*}
    We also have
    \begin{equation*}
        \begin{aligned}
             \frac{d}{dt} \C(p_1,y_t)  =& \langle p_{1,t} - p_{0,t}, M_t^{-1}(q_1-q_0) \rangle \\
             \leq& \langle p_1 - p_0 , M_0^{-1} (q_1 - q_0) \rangle \\
             & + |\langle p_{1,t}-p_{0,t}, M_t^{-1}(q_1-q_0)\rangle  - \langle p_1 - p_0, M_0^{-1} (q_1 - q_0) \rangle| \\
             \leq& \| p_1 - p_0 \| \|v\| \cos(\theta) + \frac{1}{2} \| p_1 - p_0 \|\| v \| \cos(\theta) \\
             =& \frac{3}{2} \| p_1 - p_0 \| \| v \| \cos(\theta),
        \end{aligned}
    \end{equation*}
    so that
    \begin{equation*}
        \C(p_1, y_1) - \C(p_1, y_0) = \int_0^1 \frac{d}{dt}\C(p_1, y_t) dt \leq \frac{3}{2}\| p_1 - p_0 \| \|v \| \cos(\theta).
    \end{equation*}
    This concludes the first part of the lemma. Similarly, if $p_1 \in (\K_{\theta, -v} +p_0)\cap [\X]_{y_0}$, then noting that $\langle p_1 - p_0, v \rangle \leq -\| p_1 - p_0 \| \| v \| \cos(\theta)$, we compute
    \begin{equation*}
        \begin{aligned}
            \frac{d}{dt} \C(p_1, y_t)  =& \langle p_{1,t} - p_{0,t}, M_t^{-1}(q_1-q_0) \rangle \\
             \leq& \langle p_1 - p_0 , M_0^{-1} (q_1 - q_0) \rangle \\
             &+ |\langle p_{1,t}-p_{0,t}, M_t^{-1}(q_1-q_0)\rangle  - \langle p_1 - p_0, M_0^{-1} (q_1 - q_0) \rangle| \\
             \leq& - \| p_1 - p_0 \| \| v \| \cos(\theta) + \frac{1}{2} \| p_1 - p_0 \| \|v \| \cos(\theta) \\
             =& - \frac{1}{2} \| p_1 - p_0 \| \| v \| \cos(\theta)
        \end{aligned}
    \end{equation*}
    and hence we obtain
    \begin{equation*}
        \C(p_1, y_1) - \C(p_1, y_0) = \int_0^1 \frac{d}{dt}\C(p_1, y_t)dt \leq -\frac{1}{2}\| p_1 - p_0 \| \|v\| \cos(\theta).
    \end{equation*}
    Also, 
    \begin{equation*}
        \begin{aligned}
            \frac{d}{dt} \C(p_1, y_t)  =& \langle p_{1,t} - p_{0,t}, M_t^{-1}(q_1-q_0) \rangle \\
             \geq& \langle p_1 - p_0 , M_0^{-1} (q_1 - q_0) \rangle \\
             &- |\langle p_{1,t}-p_{0,t}, M_t^{-1}(q_1-q_0)\rangle  - \langle p_1 - p_0, M_0^{-1} (q_1 - q_0) \rangle| \\
             \geq& - \| p_1 - p_0 \| \| v \| \cos(\theta) - \frac{1}{2} \| p_1 - p_0 \| \|v \| \cos(\theta) \\
             =& - \frac{3}{2} \| p_1 - p_0 \| \| v \| \cos(\theta)
        \end{aligned}
    \end{equation*}
    and therefore
    \begin{equation*}
        \C(p_1, y_1) - \C(p_1, y_0) = \int_0^1 \frac{d}{dt}\C(p_1, y_t)dt \geq -\frac{3}{2}\| p_1 - p_0 \| \|v\| \cos(\theta).
    \end{equation*}
    This yields the second part of the lemma.
\end{proof}

If we have Loeper's property, then the sections of $c$-affine functions are $c$-convex with respect to the corresponding $y$.

\begin{Lem}\label{lem: c-affine section c-convex}
Let the cost function $c$ satisfy Loeper's property \eqref{eqn: Loeper}, and we use the notations from Lemma \ref{lem: section non empty interior}. 
Then the section $\csec{f_{0}}{f_{1}}$ is $c$-convex with respect to $y_1$.
\end{Lem}
\begin{proof}
Note first that $f_{0}(x') = f_{1} (x') = 0$ so that $x' \in \csec{f_{0}}{f_{1}} \neq \emptyset$. Let $x_0, x_1 \in \csec{f_{0}}{f_{1}}$, and let $x_t \in \cseg{y_1}[x_0, x_1]$. Then \eqref{eqn: Loeper} implies 
\begin{align*}
	f_{1}(x_t) - f_{0}(x_t) &= -c(x_t, y_1) +c(x_t, y_0)\\
	& \leq \max \{ -c(x_0, y_1) + c(x_0, y_0), -c(x_1, y_1) + c(x_1, y_0) \} \\
	& = \max \{ f_{1}(x_0) - f_{0}(x_0) , f_{1}(x_1) - f_{0}(x_1) \} \\
	& \leq 0,
\end{align*}
where we have used $x_0, x_1 \in \csec{f_{0}}{f_{1}}$ in the last inequality. Hence $x_t \in \csec{f_{0}}{f_{1}}$. 
\end{proof}

\subsection{$c$-convex functions}
In this subsection, we define the $c$-convex functions and show some properties of them. The $c$-affine functions and the $c$-convex functions are very important concepts in the optimal transport theory. It is well known in the optimal transport literature that the solution $T$ of the optimal transport problem is given by $T(x) = \cexp{x}{\nabla \phi(x)}$ for some $c$-convex function $\phi$. Hence, it is essential to study properties of the $c$-convex functions in the optimal transport research. 

\begin{Def}
    We define \emph{$c$-convex functions} to be the supremum envelopes of $c$-affine functions, i.e. a function $\phi: \X \to \R$ is $c$-convex function if
\begin{equation}\label{eqn: c-conv function}
	\phi(x) = \sup_{y \in \Y} \{ -c(x,y) + \psi(y) \}
\end{equation}
for some $\psi: \Y \to \R\cup\{ -\infty \}$ that is not identically $-\infty$.
\end{Def}

\begin{Rmk}
The $c$-affine functions are $c$-convex by definition.
\end{Rmk}

We first prove Lipschitz regularity of $c$-convex functions.

\begin{Lem}\label{lem: Lipschitz c-conv}
Let $\phi : \X \to \R$ be a $c$-convex function. Then $\phi$ is Lipschitz with the Lipschitz constant $\clip$. 
\end{Lem}
\begin{proof}
Let $x_0, x_1 \in \X$, and let $\epsilon >0$. By \eqref{eqn: c-conv function}, there exists $y_\epsilon \in \Y$ such that
\begin{equation*}
\phi(x_1) \leq -c(x_1, y_\epsilon) + \psi(y_\epsilon) + \epsilon.
\end{equation*}
Then we compute
\begin{align*}
\phi(x_1) - \phi(x_0) & \leq -c(x_1, y_\epsilon) + \psi(y_\epsilon) + \epsilon - \sup_{y \in \Y} \{ -c(x_0,y) + \psi(y) \} \\
& \leq  -c(x_1, y_\epsilon) + \psi(y_\epsilon) + \epsilon +c(x_0, y_\epsilon) - \psi(y_\epsilon) \\
& = -c(x_1, y_\epsilon) + c(x_0, y_\epsilon) + \epsilon \\
& \leq \clip \| x_1 - x_0 \| + \epsilon.
\end{align*}
Taking $\epsilon \to 0$ in the last line, we obtain 
\begin{equation}\label{eqn: phi lip oneside}
\phi(x_1) - \phi(x_0) \leq \clip \| x_1 - x_0 \|.
\end{equation}
Using the same method with the role of $x_0$ and $x_1$ changed, we can also obtain $\phi(x_0) - \phi(x_1) \leq \clip \| x_1 - x_0 \|$. Combining with \eqref{eqn: phi lip oneside}, we obtain 
\begin{equation*}
| \phi(x_1) - \phi(x_0) | \leq \clip \| x_1 - x_0 \|,
\end{equation*}
which yields the Lipschitz continuity of $\phi$ with the Lipschitz constant $\clip$.
\end{proof}

$c$-convex functions has an analogy of subdifferentials of the convex functions which we call $c$-subdifferentials

\begin{Def}\label{lem: c-subdifferential}
    Let $\phi: \X \to \R$ be a $c$-convex function and let $x_0 \in \X$. If there exists $y_0 \in \Y$ such that 
    \begin{equation*}
        \begin{aligned}
            -c(x, y_0) + c(x_0, y_0) + \phi(x_0) \leq \phi(x) \quad \forall x \in \X,
        \end{aligned}
    \end{equation*}
    then we call that $y_0$ is a \emph{$c$-subdifferential of $\phi$ at $x_0$}.
\end{Def}

Any $c$-convex function has a $c$-subdifferential at any point in the domain.

\begin{Lem}
    Let $\phi: \X \to \R$ be a $c$-convex function. Then, for any $x_0 \in \X$, there exists $y_0 \in \Y$ which is a $c$-subdifferential of $\phi$ at $x_0$.
\end{Lem}
\begin{proof}
    By the definition of $c$-convex functions, there exists a function $\psi : \Y \to \R \cup \{ -\infty \}$, not identically $-\infty$, such that
    \begin{equation*}
        \phi(x) = \sup\{ -c(x,y) + \psi(y) | y \in \Y \}.
    \end{equation*}
    Note first that, since $\psi$ is not identically $-\infty$, we have
    \begin{equation*}
        \phi(x) \geq -c(x, y')+\psi(y'),
    \end{equation*}
    where $\psi(y') > -\infty$. In particular, $\phi$ is bounded from below and Lemma \ref{lem: Lipschitz c-conv} yields that $\phi$ is a bounded function. Let $y_k \in \Y$ be a sequence that approximates $\phi(x_0)$ in the above supremum:
    \begin{equation}\label{eqn: c-subdiff approx}
        \lim_{k \to \infty} -c(x_0, y_k) + \psi(y_k) = \phi(x_0).
    \end{equation}
    The above equality implies that $\psi(y_k)$ is bounded. Using the compactness of $\Y$, up to a subsequence (not relabeled), we obtain $y_k \to y_0$ for some $y_0 \in \Y$ and $\psi(y_k) \to h$ for some $h \in \R$. Then passing the limit $k \to \infty$ in \eqref{eqn: c-subdiff approx}, we obtain
    \begin{equation*}
        -c(x_0, y_0) + h = \phi(x_0).
    \end{equation*}
    Moreover, for any $x \in \X$, we have
    \begin{equation*}
            -c(x, y_0) + h  = \lim_{k \to \infty}-c(x,y_k)+\psi(y_k) \leq \phi(x).
    \end{equation*}
    Therefore, $y_0$ is a subdifferential of $\phi$ at $x_0$.
\end{proof}

\section{c-chord}\label{sec: c-chord}
In this section, we define the $c$-chord which will be used to define the alternative $c$-convex function in the next section. In subsection \ref{subsec: c-chord no mtw}, we define the $c$-chord and prove a few properties of the $c$-chord without Loeper's property \eqref{eqn: Loeper}. In subsection \ref{subsec: c-chord with mtw}, we provide properties of the $c$-chord which can be derived from Loeper's property.

\subsection{$c$-chord and basic properties (without Loeper's property)}\label{subsec: c-chord no mtw}

\begin{Def}\label{def: c-chord}
Let $X_i = (x_i, u_i) \in \X \times \R$ for $i = 0,1$. We define $\Fx : \X \to \R$ by
\begin{equation}\label{eqn: c-chord}
	\Fx(x) = \sup \{ -c(x,y) +h | y \in \Y, h \in \R, -c(x_i,y) + h \leq u_i \}.
\end{equation}
We call $\Fx$ the \emph{$c$-chord between $X_0$ and $X_1$}. If 
\begin{equation*}
\Fx(x_i) = u_i
\end{equation*}
for $i = 0,1$, then we also say that \emph{$\Fx$ is the $c$-chord connecting $X_0$ and $X_1$}.
\end{Def}

The $c$-chord is an analogy of the segment that connects the two points in the graph of a convex function, as explained in the introduction. The graph of the function \eqref{eqn: c-chord} will be the segment connecting the two points $X_0$ and $X_1$ if $c(x,y) = 
-\langle x,y \rangle$ and $\Y = \Rn$. The segment connecting the two points used for defining the convex function is sometimes called a 'chord'. Hence, inspired by the chord, we have used the name '$c$-chord' for the function \eqref{eqn: c-chord}.

\begin{Rmk}\label{rmk: Lipschitz c-chord}
the $c$-chords are, by definition, $c$-convex functions. Hence, by Lemma \ref{lem: Lipschitz c-conv}, every $c$-chord is Lipschitz with Lipschitz constant $\clip$.
\end{Rmk}

The next lemma shows that, to find the value of the $c$-chord $\Fx$, we only need to consider the $c$-affine functions which pass through either $X_0$ or $X_1$.

\begin{Lem}\label{lem: c-chord with c-affine passing end points}
    Let $X_i =(x_i, u_i) \in \X \times \R$ for $i = 0,1$. Then we have
    \begin{equation*}
        \Fx(x) =\sup \left( \bigcup_{i=0,1}\{-c(x, y)+c(x_i, y) + u_i | -c(x_{i+1}, y)+c(x_i, y) + u_i \leq u_{i+1} \}  \right),
    \end{equation*}
    where we have used the convention $1+1=0$ in the above equation.
\end{Lem}
\begin{proof}
    Observe first that we have
    \begin{equation*}
        \begin{aligned}
            &\bigcup_{i=0,1}\{-c(x, y)+c(x_i, y) + u_i | -c(x_{i+1}, y)+c(x_i, y) + u_i \leq u_{i+1} \} \\
            \subset & \{ -c(x,y) +h | y \in \Y, h \in \R, -c(x_i,y) + h \leq u_i \},
        \end{aligned}
    \end{equation*}
    and therefore 
    \begin{equation*}
        \Fx(x) \geq \sup \left( \bigcup_{i=0,1}\{-c(x, y)+c(x_i, y) + u_i | -c(x_{i+1}, y)+c(x_i, y) + u_i \leq u_{i+1} \}  \right).
    \end{equation*}
    On the other hand, suppose $-c(\cdot,y)+h$ is such that $-c(x_i, y) +h \leq u_i$ for $i = 0,1$. Let $\lambda(y,h) = \min\{ u_i +c(x_i, y) - h| i=0,1\}\geq 0$. If $\lambda(\lambda,h) = u_i +c(x_i, y)-h$, then we have
    \begin{equation*}
        -c(x,y) + h \leq -c(x,y)+h +\lambda  = -c(x, y) +c(x_i, y) + u_i.
    \end{equation*}
    Therefore, we obtain
    \begin{equation*}
        \begin{aligned}
            \Fx(x) = & \sup \{ -c(x,y) +h | y \in \Y, h \in \R, -c(x_i,y) + h \leq u_i \} \\
            \leq &\sup \left\{ -c(x, y)+h+\lambda(y,u) \bigg| \begin{matrix} -c(x_i,y) + h \leq u_i \end{matrix} \right\} \\
            = & \sup \left( \left\{ -c(x, y)+h+\lambda(y,u) \bigg| \begin{matrix} -c(x_i,y) + h \leq u_i \\ \lambda = u_0 +c(x_0, y) -h \end{matrix}  \right\} \right) \\
            & \cup \sup \left( \left\{ -c(x, y)+h+\lambda(y,u) \bigg| \begin{matrix} -c(x_i,y) + h \leq u_i \\ \lambda = u_1 +c(x_1, y) -h \end{matrix}  \right\} \right) \\
            = & \sup \left( \bigcup_{i=0,1}\{-c(x, y)+c(x_i, y) + u_i | -c(x_{i+1}, y)+c(x_i, y) + u_i \leq u_{i+1} \}  \right).
        \end{aligned}
    \end{equation*}
    This concludes the proof.
\end{proof}

A simple comparison property of the $c$-chord can be derived from the definition.

\begin{Lem}\label{lem: ordered F}
Let $X_i = (x_i, u_i)$ and $X_i' = (x_i, u_i')$ for some $x_i \in \X$, $u_i, u_i' \in \R$, for $i=0,1$. Suppose $u_i' \leq  u_i$ for $i = 0, 1$, then we have
\begin{equation*}
	F_{X_0' X_1'}(x) \leq F_{X_0 X_1}(x)
\end{equation*}
for any $x \in \X$. If $u_i' + \lambda = u_i$ for some $\lambda \geq 0$, for $i=0,1$, then 
\begin{equation*}
	F_{X_0' X_1'}(x) +\lambda = F_{X_0, X_1}(x) \quad \forall x \in \X.
\end{equation*}
In particular, if $u_i' < u_i$ for $i=0,1$, then $\Fxp(x) < \Fx(x)$ for any $x \in X$.
\end{Lem}
\begin{proof}
Note that if $-c(x_i, y) +h \leq u_i'$, then $-c(x_i, y) +h \leq u_i$. This implies
\begin{equation*}
	\{ -c(x,y) + h | -c(x_i, y) + h \leq u_i', i = 0,1\} \subset \{ -c(x,y) + h | -c(x_i, y) + h \leq u_i, i = 0,1\}.
\end{equation*}
Hence, taking supremum over the above sets, the above inclusion induces 
\begin{equation*}
F_{X_0' X_1'}(x) \leq F_{X_0 X_1}(x). 
\end{equation*}
If $u_i' +h= u_i, i = 0,1$, then
\begin{align*}
	F_{X_0' X_1'}(x)+ \lambda & = \sup\{ -c(x,y) + h | -c(x_i, y) + h \leq u_i', i = 0,1\}+ \lambda \\
	& = \sup\{ -c(x,y) + h + \lambda | -c(x_i, y) + h \leq u_i', i = 0,1\} \\
	& = \sup\{ -c(x,y) + h' | -c(x_i, y) + h' - \lambda \leq u_i' = u_i-\lambda, i = 0,1\} \\
	& = F_{X_0 X_1}(x).
\end{align*}
If $u_i' < u_i$ for $i=0,1$, then let $\lambda = u_i - u_i'>0$ and we obtain
\begin{equation*}
	\Fxp(x) < \Fxp(x) + \lambda = \Fx(x).
\end{equation*}
\end{proof} 

The $c$-chord $\Fx$ is also a Lipschitz function by Remark \ref{rmk: Lipschitz c-chord}. Thus it is not true that, for any $X_i = (x_i, u_i) \in \R^{n+1}$ for $i =0,1$, the $c$-chord $\Fx$ connects $X_0$ and $X_1$: If $u_1 - u_0>0$ is too big, then we can have 
\begin{equation*}
	\Fx(x_1) < u_1,
\end{equation*}
so that $X_1 \not\in \graph{\Fx}$, i.e. $\Fx$ does not pass through $X_1$. In this case, we can find another point $X_1'  =(x_1, \Fx(x_1))$, and observe that $\Fx$ connects $X_0$ and $X_1'$.

\begin{Lem}\label{lem: actual F}
Let $X_i = (x_i, u_i) \in \X \times \R$, $u_i' = \Fx(x_i)$ and $X_i' = (x_i, u_i')$ for $i = 0,1$. Then $\Fx = \Fxp$.
\end{Lem}
\begin{proof}
By definition of $\Fx$ \eqref{eqn: c-chord}, we have $u_i' =\Fxp(x_i) \leq u_i$ for $i=0,1$. Therefore Lemma \ref{lem: ordered F} yields that we have $\Fxp \leq \Fx$.\\
On the other hand, by definition of $\Fx$, if $y \in \Y$ and $h \in \R$ satisfy $-c(x_i,y) +h \leq u_i$, then 
\begin{equation}\label{eqn: -cxi leq ui'}
	-c(x_i,y)+h \leq \Fx(x_i) = u_i'.
\end{equation}
Therefore $-c(x, y )+h \leq \Fxp(x)$ by the definition of $c$-chord \eqref{eqn: c-chord}. Hence, taking sup on \eqref{eqn: -cxi leq ui'} over $y$ and $h$ satisfying $-c(x_i,y) +h \leq u_i$, we obtain
\begin{align*}
	\Fx(x) &= \sup\{-c(x,y)+h | y \in \Y, h \in \R, -c(x_i,y) +h \leq u_i, i=0,1 \} \\
	& \leq \sup \{-c(x,y)+h | y \in \Y, h \in \R, -c(x_i,y)+h \leq u_i', i=0,1 \} \\
	& = \Fxp(x).
\end{align*} 
\end{proof}

Lemma \ref{lem: actual F} yields that for $X_i = (x_i, u_i) \in \R^{n+1}$, if we set $u_i' = \Fx(x_i)$, and $X_i' = (x_i, u_i') \in \R^{n+1}$, then $\Fx$ is a $c$-chord connecting $X_0'$ and $X_1'$. In fact, if $\Fx$ is a $c$-chord connecting $X_0$ and $X_1$, then there exists a $c$-affine function that passes through both $X_0$ and $X_1$.

\begin{Lem}\label{lem: existence of touching c}
Let $X_i = (x_i, u_i) \in \X \times \R$ for $i=0,1$. Suppose $\Fx(x_i) = u_i$, then there exist $y \in \Y$ and $h \in \R$ such that 
\begin{equation}\label{eqn: the touching c}
	-c(x_i, y) +h = u_i.
\end{equation}
\end{Lem}
\begin{proof}
Since $\Fx(x_i) = u_i$, by the definition of a $c$-chord \eqref{eqn: c-chord}, we obtain sequences $(y_{j,k},h_{j,k}) \in \Y  \times \R$, where $j = 0,1$, such that
\begin{align}
	&-c(x_i, y_{j,k}) + h_{j,k} \leq u_i, i=0,1, \label{eqn: exist touching c: below ui at xi}\\
	&\lim_{k \to \infty} -c(x_j, y_{j,k} ) + h_{j,k} = u_j. \label{eqn: exist touching c: approaching uj at xj}
\end{align} 
Note that \eqref{eqn: exist touching c: below ui at xi} and \eqref{eqn: exist touching c: approaching uj at xj} implies that, for any $\epsilon>0$, and for big enough $k$, we have
\begin{equation}\label{eqn: exist touching c epsilon error}
	 u_j +c(x_j, y_{j,k}) -h_{j,k} \leq \epsilon.
\end{equation}
\eqref{eqn: exist touching c: below ui at xi} and \eqref{eqn: exist touching c epsilon error} implies
\begin{equation*}
	\inf_{\X \times \Y} c + u_j-\epsilon \leq h_{j,k} \leq \sup_{\X \times \Y} c + u_j.
\end{equation*}
Hence $h_{j,k}$ is bounded. Since $\Y$ is also compact, up to a subsequence of each sequence (not relabeled), we obtain
\begin{equation*}
    \lim_{k \to \infty} y_{j,k} = y_j \quad \textrm{and} \quad \lim_{k \to \infty} h_{j,k} = h_j.
\end{equation*}
Note that \eqref{eqn: exist touching c: below ui at xi} and \eqref{eqn: exist touching c: approaching uj at xj} yields that we have $-c(x_{j+1}, y_j) +h_j \leq u_{j+1}$ and $-c(x_j, y_j) + h_j = u_j$ (with convention $1+1=0$ in $j$ index). Then, by the intermediate value theorem, there exists $x' \in \cseg{y_0}[x_0, x_1]$ such that 
\begin{equation*}
	-c(x', y_0) + h_0 = -c(x', y_1) + h_1 =: \lambda.
\end{equation*}
Let $y_t \in \cseg{x'}[y_0, y_1]$, and define
\begin{equation*}
	g_j(t) = -c(x_j, y_j) + h_j + c(x_j, y_t) - c(x', y_t) - \lambda.
\end{equation*}
Then we have 
\begin{align*}
	g_0(0) &= -c(x_0, y_0) +h_0 + c(x_0, y_0) - c(x', y_0) - \lambda\\
	& = 0 \\
	& \leq -c(x_1, y_1) + h_1 +c(x_1, y_0) - h_0 \\
	& =  -c(x_1, y_1) + h_1 + c(x_1, y_0) - c(x', y_0) - \lambda = g_1(0)
\end{align*}
and similarly $g_0 (1) \geq g_1(1)$. Hence, using the intermediate value theorem for $g_1 - g_0$, we obtain that there exists $t' \in [0,1]$ such that
\begin{equation*}
	g_0(t') = g_1(t') =: u.
\end{equation*}
Consider $f(x) = -c(x, y_{t'}) +c(x', y_{t'}) + \lambda + u$. We observe that
\begin{align*}
	f(x_0) & = -c(x_0, y_{t'}) + c(x', y_{t'}) +\lambda+ u \\
	& = -c(x_0, y_{t'}) + c(x', y_{t'})+\lambda -c(x_0, y_0) + h_0 + c(x_0, y_{t'}) - c(x', y_{t'}) - \lambda \\
	& = -c(x_0, y_0) + h_0 = u_0.
\end{align*}
Similarly, $f(x_1) = u_1$. Therefore, \eqref{eqn: the touching c} holds with $y=y_{t'}$ and $h=c(x', y_{t'}) + \lambda + u$.
\end{proof}

\subsection{Properties of $c$-chord using Loeper's property}\label{subsec: c-chord with mtw}
If the cost function satisfies Loeper's property, then $c$-chords have several more nice properties. The first lemma in this subsection shows that a $c$-chord is equal to some $c$-affine function that satisfies \eqref{eqn: the touching c} on a $c$-segment.

\begin{Lem}\label{lem: F=-c on segment}
Let $c$ satisfy Loeper's property. Let $X_i = (x_i,u_i) \in \X \times \R$ for $i=0,1$. Suppose $y \in \Y$ be such that $-c(x_i, y) + h = u_i$, $i=0,1$ for some $h \in \R$. Let $x_t \in \cseg{y}[x_0,x_1]$. Then we have
\begin{equation}\label{eqn: F=-c on segment}
F_{X_0 X_1}(x_t) = -c(x_t, y) + h.
\end{equation}
\end{Lem}
\begin{proof}
Let $y' \in \Y$, $h' \in \R$ and suppose $-c(x_i, y')+h' \leq u_i$ for $i = 0,1$. Then, using \eqref{eqn: Loeper}, we obtain
\begin{align*}
-c(x_t, y') + c(x_t, y) & \leq \max\{ -c(x_0, y') + c(x_0, y), -c(x_1, y') + c(x_1, y) \} \\
& = \max\{-c(x_0, y') + h-u_0, -c(x_1, y') + h -u_1 \}\\
& \leq \max\{ u_0 - h' + h - u_0, u_1 - h' +h -u_1 \} \\
& = h-h',
\end{align*} 
where we have used $-c(x_i, y) +h= u_i$ at the second line and $-c(x_i, y')+h' \leq u_i$ at the third line in the above inequalities. Rearranging terms, we obtain
\begin{equation*}
-c(x_t, y') + h' \leq -c(x_t, y) + h.
\end{equation*}
Therefore, by the definition of $F_{X_0 X_1}$, we obtain
\begin{equation*}
\Fx(x_t) = \sup\{ -c(x_t, y') + h' | -c(x_i, y')+h' \leq u_i, i=0,1 \} \leq -c(x_t, y) + h.
\end{equation*}
Noting that $-c(x_i, y)+h$ also in $\{ -c(x_t, y') + h' | -c(x_i, y')+h' \leq u_i, i=0,1 \}$, we obtain the equality \eqref{eqn: F=-c on segment}
\end{proof}

In fact, if we have \eqref{eqn: F=-c on segment} for any $X_i$, $y$ and $h$ satisfying the assumptions of Lemma \ref{lem: F=-c on segment}, then $c$ satisfies Loeper's property \eqref{eqn: Loeper}.

\begin{Lem}\label{lem: F=-c on segment implies Loeper}
Suppose that for any $X_i =(x_i, u_i) \in \X \times \R$ for $i = 0,1$, and for any $y \in \Y$ and $h \in \R$ satisfying $-c(x_i, y) + h = u_i$ for $i = 0,1$, we have \eqref{eqn: F=-c on segment} where $x_t \in \cseg{y}[x_0, x_1]$. Then $c$ satisfies Loeper's property \eqref{eqn: Loeper}.
\end{Lem}
\begin{proof}
Fix $y_0 \in \Y$, $x_0, x_1 \in \X$. Let $u_i = -c(x_i, y_0)$ for $i = 0, 1$, and let $x_t \in \cseg{y_0}[x_0,x_1]$. For any $y$, we define 
\begin{equation*}
h' = \min\{c(x_i, y)+u_i| i=0,1 \}.
\end{equation*}
Then $-c(x_i, y) +h' \leq u_i$ for $i = 0,1$. Therefore, by Definition \ref{def: c-chord}, we observe
\begin{equation*}
F_{X_0 X_1} (x) \geq -c(x,y) + h'. 
\end{equation*}
By the above inequality and the assumption \eqref{eqn: F=-c on segment}, we obtain
\begin{align*}
-c(x_t, y_0) + h &= \Fx(x_t) \\
&\geq -c(x_t,y) + h' = -c(x_t, y) + \min\{c(x_i, y)+u_i| i=0,1 \}.
\end{align*}
Multiplying $-1$ and rearranging the terms, we obtain
\begin{align*}
-c(x_t, y) + c(x_t, y_0)& \leq \max\{ -c(x_i, y) - u_i| i=0,1\}+h\\
& = \max\{ -c(x_i, y)+c(x_i, y_0)-h| i=0,1 \} +h \\
& = \max\{ -c(x_i, y)+c(x_i, y_0)| i=0,1 \},
\end{align*}
which is \eqref{eqn: Loeper}.
\end{proof}

Combining Lemma \ref{lem: F=-c on segment} and Proposition \ref{lem: F=-c on segment implies Loeper}, we obtain another expression of the MTW condition.

\begin{Prop}
The cost function $c$ satisfies Loeper's property \eqref{eqn: Loeper} if and only if, we have \eqref{eqn: F=-c on segment} for any $X_i = (x_i,u_i) \in \X \times \R$, $i=0,1$, $y \in \Y$, and $h \in \R$ that satisfy $-c(x_i, y) + h = u_i$ for $i = 0,1$.
\end{Prop}
\begin{proof}
The straight direction is Lemma \ref{lem: F=-c on segment} and the opposite direction is Lemma \ref{lem: F=-c on segment implies Loeper}.
\end{proof}

Another implication of Loeper's property is strict comparison of $c$-chords on $c$-segments. In Lemma \ref{lem: ordered F}, we have $\Fxp(x) < \Fx(x)$ if $u'_i < u_i$ for $i=0,1$. In the next lemma, however, we only require one of the strict inequalities $u'_i < u_i$, and obtain the strict equality $\Fxp(x) < \Fx(x)$, but only in a $c$-segment.

\begin{Lem}\label{lem: strictly ordered F}
Suppose $c$ satisfies Loeper's property. Let $X_i = (x_i, u_i), X_i' = (x_i, u_i')$, for $x_i \in \X$, $u_i, u_i' \in \R$, $i = 0,1$. Also assume $\Fx(x_i) = u_i$ for $i = 0,1$. Suppose $u_0' = u_0$ and $u_1' < u_1$. Then, for any $y \in \Y$ and $h \in \R$ such that 
\begin{equation*}
	-c(x_i , y) +h = u_i,
\end{equation*}
we have 
\begin{equation*}
	\Fxp(x_t) < \Fx(x_t)
\end{equation*}
for any $x_t \in \cseg{y}(x_0, x_1)$.
\end{Lem}
\begin{proof}
Suppose $y \in \Y$ and $h \in \R$ satisfy $-c(x_i , y) +h = u_i$. Then by Lemma \ref{lem: F=-c on segment}, we have $\Fx(x_t) = -c(x_t,y)+h$ where $x_t \in \cseg{y}(x_0, x_1)$. Also, by Lemma \ref{lem: ordered F}, we have $\Fxp \leq \Fx$. To prove by contradiction, suppose $\Fxp(x_t) = \Fx(x_t)$ for some $x_t \in \cseg{y}(x_0,x_1)$. We claim that there exists $y' \in \Y$ and $h' \in \R$ such that
\begin{equation}\label{eqn: strictly ordered F: y' h'}
	-c(x_i, y') + h' \leq u_i', \quad -c(x_t, y') + h' = \Fxp(x_t).
\end{equation}
Let $y_k \in \Y$ and $h_k \in \R$ be sequences such that 
\begin{equation}\label{eqn: strictly ordered F: yk, hk}
	-c(x_i, y_k) + h_k \leq u_i', \quad \lim_{k \to \infty} -c(x_t,y_k) + h_k = \Fxp(x_t).
\end{equation}
Then, for big enough $k$, we have $|-c(x_t, y_k) + h_k -\Fxp(x_t) | \leq 1$. Since $\Fxp(x_t) < \infty$ thanks to Remark \ref{rmk: Lipschitz c-chord}, we obtain that for big enough $k$, 
\begin{equation*}
	\Fxp(x_t) +\min_{\X \times \Y} c -1 \leq h_k \leq \Fxp(x_t) + \max_{\X \times \Y} c + 1,
\end{equation*}
In particular, $h_k$ is a bounded sequence. Since $\Y$ is also compact, possibly passing to subsequences of $y_k$ and $h_k$ (not relabeled), we obtain
\begin{equation*}
	\lim_{k \to \infty} y_k = y' \in \Y, \quad \lim_{k \to \infty} h_k = h' \in \R.
\end{equation*} 
Then, taking limit $k \to \infty$ in \eqref{eqn: strictly ordered F: yk, hk}, we obtain \eqref{eqn: strictly ordered F: y' h'}, and the claim is proved. \\
Since we assumed $-c(x_1, y) + h = u_1$, we have
\begin{equation}\label{eqn: cy'+h'<cy+h at x1}
	-c(x_1, y') + h' \leq u_1' < u_1 = -c(x_1, y) + h
\end{equation}
Let 
\begin{equation*}
f(x) = -c(x,y)+h, \quad f'(x) = -c(x,y')+h'
\end{equation*}
Then $x_0, x_1 \in \csec{f'}{f}$ and Lemma \ref{lem: c-affine section bdy} implies $x_t \in \partial \csec{f'}{f}$. Moreover, Lemma \ref{lem: c-affine section c-convex} yields that $\csec{f'}{f}$ is convex with respect to $y$. Let $p_i = -D_y c(x_i, y)$. Then we have $p_t = -D_y c(x_t, y) = t p_1 + (1-t)p_0$ and Corollary \ref{lem: boundary and interior of c-affine section} shows $p_t \in \partial \csec{f'}{f}$. On the other hand, \eqref{eqn: cy'+h'<cy+h at x1} and Corollary \ref{lem: boundary and interior of c-affine section} imply $x_1 \in \Int{\csec{f'}{f}}$ which yields $p_1 \in \Int{\left[\csec{f'}{f}\right]_y}$. Then there exists $r >0$ such that $B_r (p_1) \subset \Int{\left[\csec{f'}{f}\right]_y}$, and the convexity of $\left[\csec{f'}{f}\right]_y$ implies
\begin{equation*}
	B_{tr}(p_t) = \{ tp + (1-t) p_0 | p \in B_r(p_1) \} \subset \left[\csec{f'}{f}\right]_y.
\end{equation*}
Therefore, we also have $p_t \in \Int{\left[\csec{f'}{f}\right]_y}$ that contradicts to  $p_t \in \partial \left[\csec{f'}{f}\right]_y$. Thus, we must have the strict inequality $\Fxp(x_t) < \Fx(x_t)$ for any $x_t \in \cseg{y}(x_0, x_1)$.
\end{proof}

\section{Alternative c-convex functions}\label{sec: alternative c-conv}
In this section, we define the alternative $c$-convex function using $c$-chord defined in Definition \ref{def: c-chord}. We study several properties of alternative $c$-convex function with or without Loeper's property. 

\subsection{Alternative $c$-convex functions}\label{subsec: Alternative c-conv}

We start with defining alternative $c$-convex function.

\begin{Def}\label{def: alternative c-conv}
Let $\phi:\X \to \R$. $\phi$ is called \emph{alternative $c$-convex} if for any $X_i = (x_i, \phi(x_i)), x_i \in \Dom{\phi}, i = 0,1$, we have
\begin{equation}\label{eqn: alt c conv}
\phi(x) \leq F_{X_0 X_1}(x), \forall x \in \X.
\end{equation}
\end{Def}

The alternative $c$-convex function is an analogy of the definition of a convex function using the following inequality, as explained in the introduction:
\begin{equation}\label{eqn: convex inequality}
	\phi(t x_1 + (1-t) x_0) \leq t \phi(x_1) + (1-t) \phi(x_0),
\end{equation}
or equivalently,
\begin{equation*}
    \phi(x) \leq \left\{ \begin{matrix}
        t\phi(x_1) + (1-t) \phi(x_0) & x = tx_1 + (1-t)x_0 \\
        \infty & \textrm{otherwise}
    \end{matrix}\right. .
\end{equation*}
Replacing the right hand side of the above inequality with the $c$-chord $\Fx$, we obtain the definition of the alternative $c$-convex function.
 
In equation \eqref{eqn: convex inequality}, we use a segment that 'connects' two points in the graph of $\phi$, but the $c$-chord $\Fx$ does not necessarily connect the two points $X_0$ and $X_1$. Still, if we pick two points in the graph of an alternative $c$-convex function, then the $c$-chord between the two points in the graph of the alternative $c$-convex function connects those two points.

\begin{Lem}\label{lem: F=phi at end pts}
Let $\phi :\X \to \R$ be an alternative $c$-convex function. Let $X_i = (x_i, \phi(x_i)) \in \graph{\phi}$, $i = 0,1$. Then 
\begin{equation*}
	\Fx(x_i) = \phi(x_i).
\end{equation*}
In other words, the $c$-chords between any two point in the graph of an alternative $c$-convex function connects the two points. 
\end{Lem}
\begin{proof}
By the definition of alternative $c$-convexity, we have $\phi(x_i) \leq \Fx(x_i)$. On the other hand, by the definition of $\Fx$ and our choice of $X_i$, $i = 0,1$, we obtain 
\begin{align*}
	\Fx(x_i) & = \sup \{ -c(x_i, y) + h | y \in \Y, h \in \R, -c(x_j, y) + h \leq \phi(x_j), j = 0,1 \} \\
	& \leq \phi(x_i).
\end{align*}
Therefore, we obtain the equality.
\end{proof}

Theorem \ref{thm: main theorem} states that the alternative $c$-convex function and the $c$-convex function are equivalent definition if and only if the cost function $c$ satisfies Loeper's property (or other equivalent MTW condition). However, from the definitions of alternative $c$-convex functions and $c$-convex functions, we can observe that the $c$-convex functions are alternative $c$-convex functions even if we do not assume Loeper's property.

\begin{Lem}\label{lem: c-conv is alt c-conv}
A $c$-convex function $\phi$ is alternative $c$-convex. 
\end{Lem}
\begin{proof}
Let $x_i \in \X$ and $X_i = (x_i, \phi(x_i))$, $i = 0,1$ and Let $x' \in \X$. Then, by $c$-convexity of $\phi$, there exists $y \in \Y$ such that 
\begin{equation*}
	\phi(x) \geq -c(x,y) +c(x',y) + \phi(x'), \quad \forall x \in \X.
\end{equation*}
In particular, the above inequality holds with $x = x_0$ and $x = x_1$. Hence by Definition \ref{def: c-chord},
\begin{equation*}
	F_{X_0 X_1}(x) \geq -c(x,y) +c(x',y) + \phi(x'), \quad \forall x \in \X
\end{equation*}
Evaluating the above inequality at $x = x'$, we obtain $F_{X_0 X_1} (x') \geq \phi(x')$ for any $x' \in \X$. Since $X_0$ and $X_1$ were arbitrary points in $\graph{\phi}$, $\phi$ is an alternative $c$-convex function.
\end{proof}

Alternative $c$-convex functions are also Lipschitz with Lipschitz constant $\clip$.

\begin{Lem}\label{lem: Lipschitz alt. c-conv}
Alternative $c$-convex functions are Lipschitz with Lipschitz constant $\clip$.
\end{Lem}
\begin{proof}
Let $\phi: \X \to \R$ be an alternative $c$-convex function and let $x_0, x_1 \in \X$. Then we have
\begin{equation*}
	\phi(x) \leq \Fx(x), \quad \forall x \in \X,
\end{equation*}
where $X_i = (x_i, \phi(x_i))$. Lemma \ref{lem: F=phi at end pts} implies $\Fx(x_i) = \phi(x_i)$ and Lemma \ref{lem: existence of touching c} shows that there exist $y \in \Y$ and $h \in \R$ such that 
\begin{equation*}
	-c(x_i,y) + h = \Fx(x_i), \quad i =0,1.
\end{equation*}
Then we observe
\begin{align*}
	|\phi(x_0) - \phi(x_1)| = | -c(x_0,y) +c(x_1, y) | \leq \clip \|x_1 - x_0\|.
\end{align*}
\end{proof}

We will also be using sections of alternative $c$-convex functions in the proof of the main theorem. Under Loeper's property, we can generalize Lemma \ref{lem: section non empty interior} and \ref{lem: c-affine section c-convex} to the alternative $c$-convex function case.

\begin{Lem}\label{lem: convex section}
Let $c$ satisfy Loeper's property and let $\phi: \X \to \R$ be an alternative $c$-convex function. Fix $y \in \Y$ and $h \in \R$, and let
\begin{equation*}
	f(x) = -c(x,y) + h.
\end{equation*}
Then the section $\csec{\phi}{f}$ is $c$-convex with respect to $y$. Moreover, if $\phi(x') < f(x')$ for some $x' \in \X$, then $\csec{\phi}{f}$ has a non-empty interior, and 
\begin{equation}\label{eqn: < in Int section}
	\Int{\csec{\phi}{f}} \subset \{ x \in \X | \phi(x) < f(x) \}.
\end{equation}
\end{Lem}
\begin{proof}
Let $x_0, x_1 \in \csec{\phi}{f}$. To show the $c$-convexity, we only need to show that the $c$-segment $\cseg{y}[x_0, x_1]$ is in $\csec{\phi}{f}$. Let $u_i = -c(x_i, y) + h$ and
\begin{equation*}
	X_i = (x_i, u_i), \quad X_i' = (x_i, \phi(x_i)),
\end{equation*}
for $i = 0,1$. $x_i \in \csec{\phi}{f}$ implies $\phi(x_i) \leq u_i$ and Lemma \ref{lem: ordered F} implies $\Fxp \leq \Fx$. Then the alternative $c$-convexity of $\phi$ implies that we have
\begin{equation*}
	\phi \leq \Fxp \leq \Fx.
\end{equation*}
Let $x_t \in \cseg{y}[x_0, x_1]$. Lemma \ref{lem: F=-c on segment} implies that $\Fx(x_t) = -c(x_t, y) + h$, and we obtain
\begin{equation*}
	\phi(x_t) \leq \Fx(x_t) = -c(x_t, y) + h.
\end{equation*}
Hence $x_t \in \csec{\phi}{f}$.\\
Next, suppose $\phi(x') < -c(x',y) + h$ for some $x' \in \X$. Since $-c(\cdot,y)+h$ is continuous and $\phi$ is also continuous by Lemma \ref{lem: Lipschitz alt. c-conv}, there exists $r>0$ such that 
\begin{equation*}
	\phi(x) < -c(x,y)+h, \quad \forall x \in B_r (x').
\end{equation*} 
Then, by Definition \ref{def: section}, $B_r(x') \subset \csec{\phi}{f}$ and $\Int{\csec{\phi}{f}} \neq \emptyset$. Also, for any $x'' \in \Int{\csec{\phi}{f}}\setminus\{x'\}$, the $c$-convexity of $\csec{\phi}{f}$ with respect to $y$ and the compactness of $\X$ imply that there exists $x_1 \in \partial \csec{\phi}{f}$ such that 
\begin{equation*}
-D_y c(x'', y) = -tD_x c(x_1, y) - (1-t) D_x c(x', y),
\end{equation*}
for some $t \in (0,1)$. In other words, $x'' \in \cseg{y}(x', x_1)$. Using Lemma \ref{lem: ordered F} when $\phi(x_1)< -c(x_1,y)+h$ and Lemma \ref{lem: strictly ordered F} when $\phi(x_1)=-c(x_1,y)+h$ with $x_0 = x'$, $X_i =(x_i, -c(x_i,y)+h)$ and $X_i' = (x_i, \phi(x_i))$, we obtain $\phi(x'') < \Fx(x'')$. Then Lemma \ref{lem: F=-c on segment} shows that we have
\begin{equation*}
\phi(x'') < -c(x'',y)+c(x',y)+h = f(x'').
\end{equation*}
Since $x'' \in \Int{\csec{\phi}{f}}$ was arbitrary, we obtain the desired inclusion \eqref{eqn: < in Int section}.
\end{proof}

\section{Proof of the main theorem}\label{sec: proof}
We prove the main theorem in this section. In subsection \ref{subsec: fisrt part}, we prove the first part of the main theorem. In subsection \ref{subsec: second part}, we show a few lemmas for the second part of the main theorem and prove the second part of the main theorem.

\subsection{The first part of the main theorem}\label{subsec: fisrt part}

Before we state prove the first part of the main theorem, we prove a simple lemma about the normal cones of a convex body.

\begin{Lem}\label{lem: conv set in normal cones}
    Let $A \in \Rn$ be a compact convex set with non-empty interior. Let $p_0 \in \partial A$ and let $S \subset \Rn$ be another convex set with non-empty interior. If $S \subset \N(A;p_0) \cup -\N(A;p_0)$, then we have
    \begin{equation*}
        S \subset \N(A;p_0) \quad \textrm{or} \quad S \subset -\N(A;p_0).
    \end{equation*}
\end{Lem}
\begin{proof}
    Since we assume that $A$ has non-empty interior, there exists $p_1 \in A$ and $r>0$ such that $B_r(p_1) \subset A$. Then we observe
    \begin{equation*}
        \N(A;p) = -\K^*(A-p) \subset -K^*(B_r(p_1)-p_0). 
    \end{equation*}
    If $v\in\Rn$ is perpendicular to $p_1 -p_0$: $\langle v, p_1 - p_0 \rangle =0$, then we observe that, for $\epsilon> 0$ such that $\epsilon \| v \| < r$, we have $p_1 + \epsilon v \in B_r(p_1)$, but
    \begin{equation*}
        \begin{aligned}
            & \langle p_1 + \epsilon v - p_0, v \rangle = \epsilon \| v \|^2 >0, \\
            & \langle p_1 - \epsilon v - p_0, v \rangle = - \epsilon \| v \|^2 < 0,
        \end{aligned}
    \end{equation*}
    so that $v \not \in -K^*(B_r(p_1)-p_0)$ and $v \not\in K^*(B_r(p_1)-p_0)$. In particular, we have $v \not\in \N(A;p_0)$ and $v \not\in -\N(A;p_0)$. \\
    Now, note that we assumed that $S$ has non-empty interior and $S \subset \N(A;p_0) \cup -\N(A;p_0)$. If $\Int{S} \subset \N(A;p_0)$, then the closedness of $\N(A;p_0)$ yields $S \subset \N(A;p_0)$, and similarly $\Int{S} \subset -\N(A;p_0)$ implies $S \subset -\N(A;p_0)$. On the other hand, if we have that $\Int{S}$ intersect with both $\N(A;p_0)$ and $-\N(A;p_0)$, then we can find $p_- , p_+ \in \Int{S}$ and $\delta>0$ such that $ B_{\delta}(w_-) \subset \Int{S} \cap \N(A;p_0)$ and $ B_{\delta}(w_+) \subset \Int{S} \cap -\N(A, p_0)$. Let 
    \begin{equation*}
        \begin{aligned}
            w_+ =  a_+ (p_1 - p_0) + b_+ v_+, \\
            w_- =  a_- (p_1 - p_0) + b_- v_-,
        \end{aligned}
    \end{equation*}
    where $v_+$ and $v_-$ are perpendicular to $p_1 - p_0$. Since $p_1 \in \Int A$, the definition of $\N(A;p_0)$ yields
    \begin{equation*}
        \begin{aligned}
            \langle w_+  , p_1 - p_0 \rangle = a_+ \| p_1 - p_0\|^2 <0, \\
            \langle w_- , p_1 - p_0 \rangle = a_- \| p_1 - p_0\|^2 > 0.
        \end{aligned}
    \end{equation*}
    Therefore, $a_+ >0$ and $a_- < 0$. Let $t = a_+ / (a_+ - a_-) \in (0,1)$. Then the convexity of $S$ implies that, for any $\rho \in (-\delta, \delta)$, we have
    \begin{equation*}
        (1-t)(w_+  + \rho v_+)+ tw_- = (b_+ + \rho)v_+ + b_- v_- \in S.
    \end{equation*}
    Hence, choosing $\rho \in (-\delta, \delta)$ such that $v=(b_+ + \rho)v_+ + b_- v_- \neq 0 $, we obtain $ v \in S \in \N(A;p_0) \cup -\N(A;p_0)$, which is impossible since $v$ is perpendicular to $p_1 - p_0$. 
\end{proof}

Now we are ready to prove the first part of Theorem \ref{thm: main theorem}.

\begin{Thm}\label{thm: main 1}
    Let $c$ satisfy Loeper's property. Then a function $\phi : \X \to \R$ is $c$-convex if and only if $\phi$ is alternative $c$-convex.
\end{Thm}
\begin{proof}
Thanks to Lemma \ref{lem: c-conv is alt c-conv}, we only need to show that any alternative $c$-convex function is a $c$-convex function. We divide the proof into four steps.\\

\noindent\emph{Step 1)} Let $\phi :\X \to \R$ be an alternative $c$-convex function. Then Lemma \ref{lem: Lipschitz alt. c-conv} yields that $\phi$ is differentiable almost everywhere in $\X$. Let $x_0 \in \Int{\X}$ be a differentiability point of $\phi$. \\
We show that there exists $y_0 \in \Y$ such that $D_x \phi(x_0) = -D_x c(x_0, y_0)$ in this step. Suppose, to prove by contradiction, $D_x \phi(x_0) \not\in [\Y]_{x_0}$. Let $q' \in[\Y]_{x_0}$ be the point that is closest to $D_x \phi(x_0)$:
\begin{equation*}
	\inf_{q \in [\Y]_{x_0}} \| D_x \phi(x_0) - q \| = \| D_x \phi(x_0) - q' \| > 0.
\end{equation*}
Let $v = D_x \phi(x_0) - q'$. Then $v$ is an outward normal vector of the convex set $[\Y]_{x_0}$ at $q'$. In particular, we have
\begin{equation}\label{eqn: v outward normal}
	\langle q-q', v \rangle \leq 0
\end{equation}
for any $q \in [\Y]_{x_0}$. On the other hand, since $x_0 \in \Int{\X}$, there exists $\epsilon>0$ such that $x_0 + t v \in \X$ for any $t \in (-\epsilon, \epsilon)$. Then, using Lemma \ref{lem: F=phi at end pts} with Lemma \ref{lem: existence of touching c} with $x_0 + tv$ instead of $x_1$ in the lemmas yields that there exist $y_t \in \Y$ and $h_t \in \R$ such that 
\begin{equation*}
	-c(x_0 , y_t) + h_t = \phi(x_0), \quad -c(x_0 + tv, y_t) + h_t = \phi(x_0+tv).
\end{equation*}
Then we observe
\begin{equation}\label{eqn: diff quo of phi and c}
	\frac{\phi(x_0+tv) - \phi(x_0)}{t} = \frac{-c(x_0+tv,y_t)+c(x_0,y_t)}{t} = \langle -D_x c(x_0, y_t),v\rangle +\mathrm{Err}( t ).
\end{equation} 
We can estimate $\mathrm{Err}$ as follows:
\begin{align*}
	|\mathrm{Err}(t)| & = \left| \frac{-c(x_0+tv,y_t)+c(x_0,y_t)}{t} - \langle -D_x c(x_0, y_t),v\rangle \right| \\
	& = |\langle -D_x c(x_0 + sv,y_t), v \rangle - \langle -D_x c(x_0, y_t),v\rangle| \\
	& = |\langle -D_x c(x_0 + sv,y_t) + D_x c(x_0, y_t), v \rangle| \\
	& \leq L |v|^2 s
\end{align*}
for some $s \in [0,t]$, where we have used the mean value theorem in the second equality and \eqref{eqn: Lipschitz} to obtain the last inequality. Since $s \in [0,t]$, $s \to 0$ as $t \to 0$. Therefore $\mathrm{Err}(t) \to 0$ as $t \to 0$. Using compactness of $\Y$, we obtain a sequence $t_k$ such that $t_k \to 0$ and $y_{t_k} \to y_0$ for some $y_0 \in \Y$ as $k \to \infty$. Hence, taking the limit $k \to \infty$ in \eqref{eqn: diff quo of phi and c} with $t=t_k$, we obtain 
\begin{equation*}\label{eqn: inner phi,v = inner q0,v}
	\langle D_x \phi(x_0) , v \rangle = \langle -D_x c(x_0, y_0), v \rangle.
\end{equation*}
Denote $q_0 = -D_x c(x_0, y_0) \in [\Y]_{x_0}$. Subtracting $\langle q', v \rangle $ on each side of the above equality, we obtain
\begin{equation}\label{eqn: contradicting inner products}
	\langle D_x \phi (x_0) - q', v \rangle = \langle q_0 - q', v \rangle.
\end{equation}
The left hand side of \eqref{eqn: contradicting inner products} is strictly positive since 
\begin{equation*}
	\langle D_x \phi (x_0) - q', v \rangle = \| D_x \phi (x_0) - q' \|^2.
\end{equation*}
The right hand side of \eqref{eqn: contradicting inner products}, however, is non-positive due to \eqref{eqn: v outward normal}. Hence we obtained a contradiction. Therefore there must exists $y_0 \in \Y$ such that $D_x \phi(x_0) = -D_x c (x_0, y_0)$.\\

\noindent\emph{Step 2)} In \emph{step 2} and \emph{step 3}, we prove 
\begin{equation}\label{eqn: cy0 supporting phi}
	-c(x, y_0) +c(x_0, y_0) + \phi(x_0) \leq \phi(x)
\end{equation}
for any $x \in \X$, where $y_0 \in \Y$ is from the \emph{step 1}. We again use a proof by contradiction. Suppose 
\begin{equation}\label{eqn: main 1: cy0>phi}
-c(x', y_0) +c(x_0, y_0) + \phi(x_0) > \phi(x')
\end{equation}
for some $x' \in \X$. Let
\begin{equation*}
	f(x) = -c(x,y_0) +c(x_0, y_0) + \phi(x_0).
\end{equation*}
Then, by Lemma \ref{lem: convex section}, the section $\csec{\phi}{f}$ has a non-empty interior. Let $q_0 = -D_x c(x_0, y_0)$ and $M_0 = -D^2_{xy} c(x_0, y_0)$. We claim that we must have the following:
\begin{equation}\label{eqn: main 1: section = normal cone}
	\left[ \csec{\phi}{f} \right]_{y_0} =\left( M_0^T \N([\Y]_{x_0} ; q_0)+p_0 \right) \cap [\X]_{y_0}.
\end{equation}
\emph{Case 1.} Here, we show 
\begin{equation}\label{eqn: main 1: section in normal cones}
    \left[ \csec{\phi}{f} \right]_{y_0} \subset \left( M_0^T\N([\Y]_{x_0} ; q_0)+p_0 \right) \cup \left( -M_0^T\N([\Y]_{x_0};q_0)+p_0\right).
\end{equation}
If we have $ \left[ \csec{\phi}{f} \right]_{y_0} \subset \left( M_0^T \N([\Y]_{x_0} ; q_0)+p_0 \right)$, then \eqref{eqn: main 1: section in normal cones} is obvious. Hence we only need to see the follow case: 
\begin{equation}\label{eqn: main 1: section - normal cone = empty}
	\left[ \csec{\phi}{f} \right]_{y_0} \setminus \left( M_0^T \N([\Y]_{x_0} ; q_0)+p_0\right) \neq \emptyset.
\end{equation}
Then, since $M_0^T \N([\Y]_{x_0} ; q_0)+p_0$ is closed, $\left[ \csec{\phi}{f} \right]_{y_0} \setminus \left(M_0^T \N([\Y]_{x_0} ; q_0)+p_0\right)$ has a non-empty interior. Let $p_1 \in \Int{\left[ \csec{\phi}{f} \right]_{y_0} \setminus \left( M_0^T \N([\Y]_{x_0} ; q_0)+p_0 \right)}$ and $x_1 = \cexp{y_0}{p_1}$. Then we have
\begin{equation}\label{eqn: main 1: v1}
	v_0 := \frac{d}{dt} x_t \bigg|_{t=0} = M_0^{-T} (p_1 - p_0) \not\in \N([\Y]_{x_0} ; q_0),
\end{equation}
where $x_t \in \cseg{y_0}[x_0, x_1]$. In other words, $v_0$ is not an outward normal vector of $[\Y]_{x_0}$ at $q_0$. We claim that $v_0 \in -\N([\Y]_{x_0};q_0)$, i.e. $v_0$ is an inward normal vector of $[\Y]_{x_0}$ at $q_0$. Since $x_1 \in \Int{\left[ \csec{\phi}{f} \right]_{y_0}}$, Lemma \ref{lem: convex section} implies
\begin{equation*}
	\phi(x_1) < f(x_1) = -c(x_1, y_0) +c(x_0, y_0) + \phi(x_0).
\end{equation*}
Then using the continuity of $\phi$ and $c$, we obtain $r_x >0$ and $r_y>0$ such that $B_{r_x}(x_1) \subset \X$, and if $\| x - x_1 \| \leq r_x$ and $y \in \Y$, $\| y - y_0 \| \leq r_y$, then
\begin{equation}\label{eqn: y near y0 then c>phi}
	-c(x, y) +c(x_0, y) + \phi(x_0) > \phi(x).
\end{equation}
On the other hand, \eqref{eqn: main 1: v1} implies
\begin{equation*}
	x_1 = \cexp{y_0}{p_1} = \cexp{y_0}{p_0 + M_0^T v_0} = \cexp{y_0}{-D_y c(x_0, y_0) -D^2_{yx} c(x_0, y_0) v_0 }.
\end{equation*}
Hence, there exists $\delta >0$ such that if $\| y_0 - y \| \leq \delta$, then
\begin{equation}\label{eqn: x1' near x1}
	\| -D_y c(x_1,y) - (-D_y c(x_0, y) -D^2_{xy} c (x_0, y) v_0 ) \| \leq \frac{r_x}{L}.
\end{equation}
Note that $B_{r_x}(x_1) \subset \X$ with \eqref{eqn: Lipschitz} implies $B_{r_x / L}(-D_y c(x_0, y)) \subset [\X]_y $. In particular, \eqref{eqn: x1' near x1} yields that
\begin{equation*}
    \cexp{y}{-D_y c(x_0, y) -D^2_{xy} c (x_0, y) v_0}
\end{equation*}
is well-defined. If $v_0$ is not an inward normal vector of $[\Y]_{x_0}$ at $q_0$, then there exists $q' \in [\Y]_{x_0}$ such that
\begin{equation}\label{eqn: y in v direction}
	\langle q'-q_0 , v_0 \rangle < 0.
\end{equation}
Choose $s \in (0, \min \{ r_y, \delta \} / L\| q' - q_0\|) $, and define
\begin{equation*}
	q_1 = q_0 + s (q' - q_0).
\end{equation*}
Note that $q_1 \in [\Y]_{x_0}$ due to the convexity of $[\Y]_{x_0}$. Also,
\begin{equation}\label{eqn: main 1: q1 length}
	\| q_1 - q_0 \| \leq \frac{\min\{r_y, \delta\}}{L}, 
\end{equation}
and
\begin{equation}\label{eqn: main 1: q1 direction}
	\langle q_1 - q_0, v_0 \rangle < 0.
\end{equation}
Let $y_1 = \cexp{x_0}{q_1}$ and $x_1' = \cexp{y_1}{-D_y c(x_0, y_1) -D^2_{xy} c (x_0, y_1) v_0 }$. \eqref{eqn: Lipschitz} and \eqref{eqn: main 1: q1 length} imply $\| y_1 - y_0 \| \leq \delta$, then \eqref{eqn: x1' near x1} yields
\begin{equation*}
	\| x_1 - x_1' \| \leq r_x.
\end{equation*}
\eqref{eqn: Lipschitz} and \eqref{eqn: main 1: q1 length} also imply $\| y_1 - y_0 \| \leq r_y$. Then \eqref{eqn: y near y0 then c>phi} implies
\begin{equation}\label{eqn: c>phi at x1'}
	-c(x_1', y_1) +c(x_0, y_1) + \phi(x_0) > \phi(x_1').
\end{equation}
Let $x'_t \in \cseg{y_1}[x_0, x_1']$. Using $-D_x c(x_0, y_0) = D_x \phi(x_0)$ and \eqref{eqn: main 1: q1 direction}, we obtain
\begin{align*}
	&\frac{d}{dt} \bigg|_{t=0} (-c(x_t', y_1) +c(x_0, y_1) + \phi(x_0)-\phi(x_t')) \\
	=&\langle -D_x c(x_0, y_1) -D_x \phi(x_0), \frac{d}{dt}\bigg|_{t=0} x_t' \rangle \\
	=&\langle -D_x c(x_0, y_1)+D_x c(x_0, y_0), v_0 \rangle \\
	<&0.
\end{align*}
Hence, for small enough $t'>0$ we have
\begin{equation}\label{eqn: c<phi at xt'}
	-c(x_{t'}', y_1) +c(x_0, y_1) + \phi(x_0)<\phi(x_{t'}').
\end{equation}
\eqref{eqn: c<phi at xt'} and \eqref{eqn: c>phi at x1'} imply that there exists $s \in (t', 1)$ such that 
\begin{equation}\label{eqn: c=phi at xs'}
	-c(x_s', y_1) +c(x_0, y_1) + \phi(x_0)=\phi(x_s').
\end{equation}
Let $X_s = (x_s', \phi(x_s'))$. Since $\phi$ is an alternative $c$-convex function, we must have $\phi(x) \leq F_{X_0 X_s}(x)$. However, \eqref{eqn: c=phi at xs'} and Lemma \ref{lem: F=-c on segment} yields 
\begin{equation*}
	F_{X_0 X_s} (x_t') = -c(x_{t}', y_1) +c(x_0, y_1) + \phi(x_0)
\end{equation*} 
for $t \in [0,s]$. Evaluating at $t=t'$, we obtain
\begin{equation*}
	\phi(x_{t'}') \leq F_{X_0 X_s} (x_{t'}') = -c(x_{t'}', y_1) +c(x_0, y_1) + \phi(x_0)<\phi(x_{t'}')
\end{equation*}
where the first inequality follows from the alternative $c$-convexity of $\phi$ and the last inequality is from \eqref{eqn: c<phi at xt'}. We obtained a contradiction, and therefore $v_0$ must be an inward normal vector:
\begin{equation}\label{eqn: main 1: section in normal cone}
	v_0 \in -\N([Y]_{x_0};q_0)
\end{equation}
Noting that $v_0 = M_0^{-T}(p_1 - p_0)$, we obtain 
\begin{equation*}
    p_1 \in -M_0^T \N([Y]_{x_0};q_0) + p_0.
\end{equation*}
Since $p_1 \in \Int{\left[ \csec{\phi}{f} \right]_{y_0} \setminus \left( M_0^T \N([\Y]_{x_0} ; q_0)+p_0\right)}$ was arbitrary, we obtain 
\begin{equation*}
    \Int{\left[ \csec{\phi}{f} \right]_{y_0} \setminus  \left( M_0^T\N([\Y]_{x_0} ; q_0)+p_0\right)} \subset  -M_0^T \N([Y]_{x_0};q_0) + p_0,
\end{equation*}
and the closedness of $ -M_0^T \N([Y]_{x_0};q_0) + p_0$ yields \eqref{eqn: main 1: section in normal cones}.\\
\emph{Case 2.} Next, we show 
\begin{equation}\label{eqn: main 1: normal cone in section}
	\left( M_0^T \N([\Y]_{x_0};q_0) + p_0 \right) \cap [\X]_{y_0} \subset \left[ \csec{\phi}{f} \right]_{y_0}.
\end{equation}
Suppose, otherwise, 
\begin{equation*}
	p_1 \in \left( (M_0^T \N([\Y]_{x_0};q_0) + p_0) \setminus \left[ \csec{\phi}{f} \right]_{y_0} \right) \cap [\X]_{y_0}.
\end{equation*}
 Again, let $x_1 = \cexp{y_0}{p_1}$ and $x_t \in \cseg{y_0}[x_0,x_1]$. Since $x_1 \not\in \csec{\phi}{f}$, we have
\begin{equation*}
	f(x_1) < \phi(x_1).
\end{equation*}
Moreover, the alternative $c$-convexity of $\phi$ implies $\phi(x_t) \leq F_{X_0 X_1}(x_t)$ where $X_i = (x_i, \phi(x_i) )$ for $i=0,1$. Lemma \ref{lem: existence of touching c} and Lemma \ref{lem: F=phi at end pts} yield that there exists $y' \in \Y$ such that
\begin{equation*}
	-c(x_1,y') +c(x_0, y') +\phi(x_0) = \phi(x_1).
\end{equation*}
Let $q' = -D_x c(x_0, y') $ and $\C (x,q) = -c(x,\cexp{x_0}{q}) +c(x_0, \cexp{x_0}{q}) +\phi(x_0)$, then we have $\C(x_1,q') > f(x_1)$. Since $\C$ is continuous in $q$ variable, there exists $r'>0$ such that $\C(x_1,q) > f(x_1)$ for any $q \in B_{r'}(q') \cap [\Y]_{x_0}$. Then the set 
\begin{equation*}
	(B_{r'}(q') \cap [\Y]_{x_0}) \setminus \{ q \in [\Y]_{x_0} | \langle q - q_0, M_0^{-T}(p_1-p_0) \rangle = 0 \}
\end{equation*}
is not an empty set because $\{ q \in [\Y]_{x_0} | \langle q - q_0, M_0^{-T}(p_1-p_0) \rangle = 0 \}$ has an empty interior. Hence we can choose $q_1$ from the above set:
\begin{equation}\label{eqn: main 1: q1}
	q_1 \in (B_{r'}(q') \cap [\Y]_{x_0}) \setminus \{ q \in [\Y]_{x_0} | \langle q - q_0, M_0^{-T}(p_1-p_0) \rangle = 0 \}.
\end{equation}
Let $f_1(x)  =\C(x, q_1)$. Then $f_1$ is a $c$-affine function and we have $f_1 (x_1) > f(x_1)$ and $f_1(x_0) = f(x_0)$. Lemma \ref{lem: convex section} yields 
\begin{equation*}
	f_1 (x_t) \geq f(x_t). 
\end{equation*}
Then we obtain
\begin{align*}
	0 & \leq \frac{d}{dt} ( f_1(x_t)- f(x_t) ) \bigg|_{t=0} \\
	& = \langle q_1 - q_0 , M_0^{-T}(p_1 - p_0) \rangle.
\end{align*}
On the other hand, thanks to $p_1 \in  (M_0^T \N([\Y]_{x_0};q_0) + p_0)$, we also have $M_0^{-T}(p_1 - p_0) \in \N([\Y]_{x_0};q_0)$, that is,
\begin{equation*}
\langle q_1 - q_0, M_0^{-T}(p_1 - p_0) \rangle \leq 0.
\end{equation*}
Therefore, we obtain $\langle q_1 - q_0, M_0^{-T}(p_1 - p_0) \rangle = 0$, which contradicts to \eqref{eqn: main 1: q1}. Hence, we must have \eqref{eqn: main 1: normal cone in section}. Now, Lemma \ref{lem: conv set in normal cones} with \eqref{eqn: main 1: section in normal cones} and \eqref{eqn: main 1: normal cone in section} yields that we must have \eqref{eqn: main 1: section = normal cone}.\\

\noindent \emph{Step 3)} \eqref{eqn: main 1: section = normal cone} in particular yields that $\N([\Y]_{x_0};q_0)$ has a non-empty interior. Let $v_1 \in \Int{\N([\Y]_{x_0};q_0)}$, then we have 
\begin{equation}\label{eqn: main 1: q0 exposed}
\langle v_1, q-q_0 \rangle < 0, \forall q \in [\Y]_{x_0} \setminus \{ q_0 \}.
\end{equation}
In other words, $q_0$ is an exposed point of $[\Y]_{x_0}$. Let $p_1 = p_0 + s_0 M_0^T v_1$, where $s_0>0$ is chosen such that $p_1 \in [\X]_{y_0}$, which is possible due to $p_0 \in \Int{[\X]_{y_0}}$. \eqref{eqn: main 1: section = normal cone} then implies $p_1 \in \left[ \csec{\phi}{f} \right]_{y_0}$.  Let $x_1 = \cexp{y_0}{p_1}$ and $x_t \in \cseg{y_0}[x_1, x_0]$. Also, let $X_i = (x_i,f(x_i))$ and $X_i' = (x_i, \phi(x_i))$. Then Lemma \ref{lem: ordered F} with $x_i \in \csec{\phi}{f}$ for $i = 0,1$ implies
\begin{equation*}
\Fxp(x) \leq \Fx(x), \quad \forall x \in \X.
\end{equation*}
On the other hand, the alternative $c$-convexity of $\phi$ implies $\phi(x) \leq \Fxp(x)$ for all $x \in \X$. Then, Using Lemma \ref{lem: F=-c on segment}, evaluating the above inequality at $x=x_t$ implies
\begin{align*}
\phi(x_t) \leq \Fxp(x_t) \leq \Fx(x_t) = f(x_t).
\end{align*}
Also, noting that $\phi(x_0) = \Fxp(x_0) = f(x_0)$, we obtain
\begin{equation*}
\frac{d}{dt} \phi(x_t) \leq \frac{d}{dt} \Fxp(x_t) \leq \frac{d}{dt} f(x_t).
\end{equation*}
Since we have $D_x \phi(x_0) = -D_x c(x_0, y_0)$ from the \emph{step 1}, the above inequality yields
\begin{equation}\label{eqn: main 1: dt fxp t=0}
\frac{d}{dt} \Fxp(x_t) = \frac{d}{dt} f(x_t) = \langle q_0, M_0^{-T}(p_1 - p_0) \rangle.
\end{equation}
We claim 
\begin{equation}\label{eqn: main 1: squeezed limit}
    \langle -D_x c(x_0, y'), M_0^{-T} (p_1 - p_0) \rangle \geq \langle q_0, M_0^{-T}(p_1 - p_0) \rangle
\end{equation}
for some $y' \in \Y$. The idea is to approximate the $c$-chord $\Fxp$ with $c$-affine functions. Lemma \ref{lem: c-chord with c-affine passing end points} yields that we can approximate $\Fxp$ using a sequence of $c$-affine functions that pass through one of the end points $X_0'$ or $X_1'$. We divide the approximation into two cases. \\
\emph{Case 1}. Suppose that there exists sequences $\epsilon_j \in (0,1)$ and $y_{j,k} \in \Y$ such that $\epsilon_j \to 0$ as $j \to \infty$, and
\begin{equation}\label{eqn: main 1: yk leq Fxp}
    \begin{aligned}
        & -c(x_1, y_{j,k}) + c(x_0, y_{j,k}) + \phi(x_0) \leq \Fxp(x_1),\\
        & -c(x_{\epsilon_j}, y_{j,k}) + c(x_0, y_{j,k}) + \phi(x_0) +O(k) \geq \Fxp(x_{\epsilon_j})  \quad \textrm{as} \quad k \to \infty.
    \end{aligned}
\end{equation}
i.e. $y_{j,k}$ produces a sequence of $c$-affine functions that pass through $X_0'$, and pass below $X_1'$ while approximating $\Fxp(\epsilon_{j})$. Then for each $\epsilon_j$, there exists $k(j)$ such that 
\begin{equation*}
-c(x_{\epsilon_j}, y_j,{k(j)}) + c(x_0, y_{j,k(j)}) + \phi(x_0) +(\epsilon_j)^2 \geq \Fxp(x_{\epsilon_j}).
\end{equation*}
Rearranging the terms and dividing the above inequality by $\epsilon_j$, we obtain
\begin{equation*}
\frac{-c(x_{\epsilon_j}, y_{k(j)}) + c(x_0, y_{k(j)})}{\epsilon_j} + \epsilon_j \geq \frac{\Fxp(x_{\epsilon_j}) - \phi(x_0)}{\epsilon_j}.
\end{equation*}
Noting that we have $\Fxp(x_0) = \phi(x_0)$ and \eqref{eqn: main 1: dt fxp t=0}, the above inequality yields
\begin{equation*}
\langle -D_x c(x_0, y_{k(j)}), \frac{d}{dt} x_t \bigg|_{t=0} \rangle + o(\|x_{\epsilon_j} - x_0\|)+\epsilon_j \geq \langle q_0, M_0^{-T}(p_1 - p_0) \rangle + o(\|x_{\epsilon_j} - x_0\|).
\end{equation*}
Thanks to $\| x_{\epsilon_j} - x_0 \| \leq \epsilon_j L\|p_1 - p_0 \|$, and using $x_t \in \cseg{y_0}[x_0,x_1]$, we deduce
\begin{equation*}
\langle -D_x c(x_0, y_{k(j)}), M_0^{-T} (p_1 - p_0) \rangle + O(\epsilon_j) \geq \langle q_0, M_0^{-T}(p_1 - p_0) \rangle + o(\epsilon_j).
\end{equation*}
The compactness of $\Y$ yields that, up to a subsequence (not relabeled), $y_{k(j)} \to y'$ for some $y' \in \Y$ as $j \to \infty$. Then taking the limit $j \to \infty$ in the above inequality yields \eqref{eqn: main 1: squeezed limit}.\\
\emph{Case 2}. Suppose that there is no sequences that satisfies the assumptions of \emph{Case 1}. Then, Lemma \ref{lem: c-chord with c-affine passing end points} yields that we must have $\epsilon_0 \in (0,1)$ and a sequence $y_{\epsilon, k}'$ such that, for any $\epsilon \in (0,\epsilon_0)$, we have
\begin{equation}\label{eqn: main 1: yk' leq Fxp}
    \begin{aligned}
        & -c(x_0, y_{\epsilon, k}') + c(x_1, y_{\epsilon,k}') + \phi(x_1) \leq \Fxp(x_0),\\
        & -c(x_{\epsilon}, y_{\epsilon,k}') + c(x_1, y_{\epsilon,k}') + \phi(x_1) +O(k) \geq \Fxp(x_{\epsilon})  \quad \textrm{as} \quad k \to \infty.
    \end{aligned}
\end{equation}
i.e. $y_{\epsilon, k}'$ yields a sequence of $c$-affine functions that pass through $X_1'$, and pass below $X_0'$ while approximating $\Fxp(\epsilon)$ for any $\epsilon \in (0, \epsilon_0)$. Then, for each $\epsilon \in (0,\epsilon_0)$, there exists $k(\epsilon)$ such that, denoting $y_{\epsilon, k(\epsilon)}'=y_\epsilon'$,
\begin{equation*}
    -c(x_\epsilon, y_{\epsilon}') + c(x_1, y_\epsilon') + \phi(x_1) + \epsilon^2 \geq \Fxp(x_\epsilon).
\end{equation*}
We subtract the first inequality in \eqref{eqn: main 1: yk' leq Fxp} with $k=k(\epsilon)$ from the above inequality and obtain
\begin{equation*}
    -c(x_\epsilon, y_{\epsilon}') + c(x_0, y_{\epsilon}') + \epsilon^2 \geq \Fxp(x_\epsilon) -\Fxp(x_0).
\end{equation*}
Dividing the above inequality by $\epsilon$ and using \eqref{eqn: main 1: dt fxp t=0} yields
\begin{equation*}
    \langle -D_x c(x_0, y_\epsilon') , \frac{d}{dt} x_t \bigg|_{t=0} \rangle + o(\| x_\epsilon - x_0 \|) + \epsilon \geq \langle q_0, M_0^{-T}(p_1 - p_0) \rangle + o(\| x_\epsilon - x_0 \| ).
\end{equation*}
The compactness of $\Y$ yields that, passing to a subsequence (not relabeled), we can assume $y_\epsilon' \to y'$ for some $y' \in \Y$. Hence, taking the limit in the above inequality implies \eqref{eqn: main 1: squeezed limit} and the claim is proved.\\
Denoting $q' = -D_x c(x_0, y')$, the above inequality implies
\begin{equation*}
\langle q' - q_0 , M_0^{-T} (p_1 - p_0) \rangle \geq 0.
\end{equation*}
Thanks to $M_0^{-T} (p_1 - p_0) = s_0 v_1$, the above inequality contradicts to \eqref{eqn: main 1: q0 exposed} unless $q' = q_0$. If $q' = q_0$, then the bi-twist condition yields $y' = y_0$. However, we also have \eqref{eqn: main 1: yk leq Fxp} that implies
\begin{align*}
-c(x_1, y') + c(x_0, y') + \phi(x_0) \leq \Fxp(x_1) & < f(x_1) \\
& = -c(x_1, y_0) + c(x_0, y_0) + \phi(x_0).
\end{align*}
Thus $q' \neq q_0$, leading to the final contradiction. As a consequence, the first assumption \eqref{eqn: main 1: cy0>phi} cannot hold, and we must have \eqref{eqn: cy0 supporting phi}.

\noindent\emph{Step 4)} Finally, Let $N \subset \X$ be the set of points where $\phi$ is not differentiable, and we define $\phi' : \X \to \R$ by
\begin{equation*}
\phi'(x) = \sup \left\{ -c(x,y_0) +h \bigg| \begin{matrix}D \phi(x_0) = -D_x c(x_0,y_0), h = c(x_0, y_0) + \phi(x_0) \\ \textrm{ for some } x_0 \in \Int{\X}\setminus N \end{matrix} \right\}.
\end{equation*}
Note that $\phi'$ is a $c$-convex function as it is defined to be a supremum over a family of $c$-affine functions. In particular, $\phi'$ is a continuous function by Lemma \ref{lem: Lipschitz c-conv}. Moreover, for any $x_0 \in \Int{\X}\setminus N$, and $y_0 \in \Y$ and $h \in \R$ that satisfy $ D_x \phi(x_0) = -D_x c(x_0,y_0), h = c(x_0, y_0) + \phi(x_0)$ (which exist due to the step 1), \eqref{eqn: cy0 supporting phi} yields
\begin{equation*}
-c(x,y_0) +h \leq \phi(x), \quad -c(x_0, y_0) + h = \phi(x_0).
\end{equation*}
Hence we have $\phi'(x_0) = \phi(x_0)$ for any $x_0 \in \Int{\X}\setminus N$. Then by the continuity of $\phi$ and $\phi'$, and that $\Int{\X}\setminus N $ is dense in $\X$, we obtain $\phi'(x) = \phi(x)$ for any $x \in \X$. Therefore, $\phi$ is a $c$-convex function.
\end{proof}

\subsection{Second part of the main theorem}\label{subsec: second part}
In this subsection, we will prove the second part of the main theorem by constructing a function that is alternative $c$-convex but not $c$-convex. Since we will not assume Loeper's property in this section, the lemmas that assume Loeper's property cannot be used in the proof of the second part of the main theorem. Hence we start with a few lemmas that will be used in the proof.

To prove the second part of the main theorem, we will assume that the cost function $c$ does not satisfy Loeper's property. In particular, there exist $y, y_0 \in \Y$, $x_0, x_1 \in \X$ and $x_t \in \cseg{y_0}(x_0,x_1)$ that do not satisfy \eqref{eqn: Loeper}. In the next lemma, we show that $y, y_0$ can be chosen from $\Int{\Y}$.

\begin{Lem}\label{lem: non Loeper interior Y}
Suppose that $c$ does not satisfy Loeper's property \eqref{eqn: Loeper}. Then there exist $y_0, y_1 \in \Int{\Y}$, $x_0, x_1 \in \Int{\X}$, $h_1' \in \R$ and $x_{t_0} \in \cseg{y_0}(x_0,x_1)$ such that
\begin{align*}
&-c(x_i, y_1)+h_1' < -c(x_i, y_0), \quad i = 0,1 \\
&-c(x_{t_0}, y_1)+h_1 > -c(x_{t_0}, y_0).
\end{align*}
\end{Lem}
\begin{proof}
If $c$ does not satisfy Loeper's property, then there exist $y_1', y_0' \in \Y$, $x_0', x_1' \in \X$ and $x_{t'}' \in \cseg{y_0}(x_0',x_1')$ that does not satisfy \eqref{eqn: Loeper}:
\begin{equation*}
-c(x_{t'}', y_1') +c(x_{t'}', y_0') > \max\{ -c(x_i', y_1') +c(x_i',y_0') | i = 0,1 \}.
\end{equation*}
We define
\begin{equation*}
    z(x,x',y,t) = \cexp{y}{-(1-t)D_y c(x, y)-tD_yc(x',y)},     
\end{equation*}
and
\begin{equation*}
    \begin{aligned}
        &g(x,x'y,y',t) \\
        =&  -c(z(x,x',y,t), y') +c(z(x,x',y,t), y) 
        - \max\{ -c(x, y') +c(x,y), -c(x', y') +c(x',y)  \} .
    \end{aligned}
\end{equation*}
Then $g: \X \times \X \times \Y \times \Y \times [0,1] \to \R$ is a continuous function. Hence, $g^{-1}(0,\infty)$ is an open set (relative to $\X \times \X \times \Y \times \Y \times [0,1]$) and $g^{-1}(0,\infty) \neq \emptyset$ as $(x_0', x_1', y_0', y_1'. t') \in g^{-1}(0,\infty)$. Choose $(x_0, x_1, y_0, y_1, t_0) \in g^{-1}(0,\infty) \cap \Int{\X \times \X \times \Y \times \Y \times [0,1]}$. Then $x_0,x_1 \in \Int{\X}$, $y_0, y_1 \in \Int{\Y}$, and $(x_0,x_1,y_0, y_1,_0') \in g^{-1}(0,\infty)$ implies
\begin{equation}\label{eqn: not Loeper int Y: not Loeper int Y}
-c(x_{t_0}, y_1) +c(x_{t_0}, y_0) > \max\{ -c(x_i, y_1) +c(x_i,y_0) | i = 0,1 \},
\end{equation}
where $x_{t_0}=z(x_0,x_1,y_0,t_0) \in \cseg{y_0}[x_0,x_1] $. Define
\begin{equation*}
\lambda = - \max\{ -c(x_i, y_1) +c(x_i,y_0) | i = 0,1 \},
\end{equation*}
and
\begin{equation*}
h_1' = \lambda - \frac{1}{2} (-c(x_{t_0}, y_1) + \lambda + c(x_{t_0}, y_0)).
\end{equation*}
Then
\begin{align*}
	&-c(x_0, y_1) + h_1' + c(x_0, y_0) \\
	=& -c(x_0,y_1) + c(x_0,y_0) + \frac{1}{2}(\lambda + c(x_{t_0},y_1) - c(x_{t_0}, y_0)) \\
	<& -c(x_0,y_1) + c(x_0,y_0) \\
	& - \max\{ -c(x_i, y_1) +c(x_i,y_0) | i = 0,1 \} \\
	\leq & 0
\end{align*}
where we have used \eqref{eqn: not Loeper int Y: not Loeper int Y} to obtain the strict inequality above. Similarly, we also obtain $-c(x_1, y_1) + h_1' < - c(x_1, y_0)$, and we obtain the first part of the lemma. On the other hand, we have
\begin{align*}
	&-c(x_{t_0}, y_1) + h_1' + c(x_{t_0}, y_0) \\
	=& -c(x_{t_0},y_1) + c(x_{t_0},y_0) \\
	& + \frac{1}{2}(\lambda + c(x_{t_0},y_1) - c(x_{t_0}, y_0)) \\
	=& \frac{1}{2}(-c(x_{t_0},y_1) +c(x_{t_0}, y_0)) + \lambda) \\
	>& 0
\end{align*}
where we have used \eqref{eqn: not Loeper int Y: not Loeper int Y} for the strict inequality above. This shows the second part of the lemma.
\end{proof}

In the following lemma, we choose the points that does not satisfy Loeper's property with extra structures that we need for constructing an alternative $c$-convex function which has a section that is disconnected. 

\begin{Lem}\label{lem: non Loeper set}
    Suppose that $c$ does not satisfy Loeper's property \eqref{eqn: Loeper}. Then there exist $y_0, y_1 \in \Int{\Y}$, $z_0, z_1 \in \Int{\X}$, $h_1 \in \R$, $\theta\in(0,\frac{\pi}{2})$ and $\rho>0$ that satisfies the following:
    \begin{itemize}
        \item [1] $-c(z_i,y_0) = -c(z_i, y_1)+h_1$, $i=0,1$.
        \item [2] For $w_j = -D_y c(z_j, y_0)$, we have 
         \begin{equation*}
             \begin{aligned}
                 \langle (-D^2_{xy}c(z_0,y_0))^{-1}(-D_x c(z_0, y_1) + D_x c(z_0, y_0)), w_1 - w_0 \rangle >0, \\ 
                 \langle (-D^2_{yx}c(z_1,y_0))^{-1}(-D_x c(z_1, y_1) + D_x c(z_1, y_0)), w_0 - w_1 \rangle >0.
             \end{aligned}
         \end{equation*}
        \item [3]  $\cos \theta > \frac{7}{8}$ and
        \begin{equation*}
            [w_0,w_1] \subset \left( \K^{\rho}_{\theta,\xi}+w_0\right) \cup \left( \K^{\rho}_{\theta, -\xi}+w_1 \right) \subset \left[ \csec{f_{0}}{f_{1}} \right]_{y_0},
        \end{equation*}
        where $\xi= w_1-w_0$, $f_{0}(x) = -c(x, y_0)$, and $f_{1}(x) = -c(x,y_1) + h_1$.
    \end{itemize}
\end{Lem}
\begin{proof}
    Let $(x_0,x_1,y_0,y_1,t_0)$ and $h_1'$ be from Lemma \ref{lem: non Loeper interior Y}. Then the function 
    \begin{equation*}
        \C(p)= -c(\cexp{y_0}{p}, y_1) + c(\cexp{y_0}{p}, y_0) 
    \end{equation*}
    is not quasi-convex in $p$ variable, since, for $p_0 = -D_y c(x_0, y_0)$ and $ p_1 =-D_y c(x_1, y_0) $, Lemma \ref{lem: non Loeper interior Y} yields that we have
    \begin{equation*}
        \begin{aligned}
            &\C(p_i) < -h_1', \quad i = 0,1 \\
            &\C(p_{t_0}) > -h_1',
        \end{aligned}
    \end{equation*}
    where $p_{t_0} = t_0 p_1 + (1-t_0) p_0 $. We claim that there exists $\tau \in (0,t_0)$ such that 
    \begin{equation}\label{eqn: increasing at t0}
        \frac{d}{dt}\C(p_t) \bigg|_{t=\tau} >0, \quad \C(p_{\tau}) > -h_1'.
    \end{equation}
    Define
    \begin{equation*}
	   \T = \{ t \in [0,t_0] | \C(p_s) > -h_1', \forall s \in [t, t_0] \}.
    \end{equation*}
    Note that $t_0 \in \T$, so that $\T \neq \emptyset$. Let $\tau_0 = \inf\T$. Then by definition of $\T$, we have
    \begin{equation}\label{eqn: main 2: t''}
	   \C(p_{\tau_0}) \geq -h_1'.
    \end{equation}
    If we have the strict inequality in \eqref{eqn: main 2: t''}, the continuity of $\C$ implies that there exists $\delta_\T >0$ such that
    \begin{equation*}
	   \C(p_t)>-h_1', \quad \forall t \in (\tau_0-\delta_\T, \tau_0+\delta_\T),
    \end{equation*}
    which contradicts to $\tau_0 = \inf\T$. Therefore, we must have the equality in \eqref{eqn: main 2: t''}
    \begin{equation*}
	   \C(p_{\tau_0})=-h_1'.
    \end{equation*}
    In particular, we have $\tau_0 < t_0$. Then we apply the mean value theorem for the function $\C(p_t)$ on the interval $[\tau_0, t_0]$ and obtain $\tau \in (\tau_0, t_0)$ such that
    \begin{align*}
	   \frac{d}{dt}\C(p_t)\bigg|_{t=\tau} & = \frac{\C(p_{t_0})-\C(p_{\tau_0})}{t_0-\tau_0} \\
	   & = \frac{\C(p_{t_0}) +h_1'}{t_0-\tau_0} \\
	   & > 0.
    \end{align*}
    Moreover, by definition of $\T$, we have $\tau \in \T$. Hence \eqref{eqn: increasing at t0} is satisfied at $\tau$.\\
    Let $h_1 = - \C(p_{\tau})$ and define 
    \begin{equation*}
        \sigma = \inf\{ s > \tau | \C(p_s)=-h_1\}.
    \end{equation*}
    Note that \eqref{eqn: increasing at t0} implies $\sigma > \tau$ and 
    \begin{equation}\label{eqn: p' segment in sup level set}
        (p_{\tau}, p_{\sigma}) \subset \{ p \in [\X]_{y_0}| \C(p) > -h_1\}.
    \end{equation}
    In addition, since $x_0$, $x_1$ are interior points of $\X$, $p_0=-D_y c(x_0, y_0)$ and $p_1 = -D_y c(x_1, y_0)$ are also interior points of $[\X]_{y_0}$, and therefore, by the convexity of $[\X]_{y_0}$, so are $p_{\tau}$ and $p_{\sigma}$. Also, the definition of $\sigma$ if we have 
    \begin{equation*}
        \frac{d}{dt}\C(p_t) \bigg|_{t=\sigma} > 0,
    \end{equation*}
    then $\C(p_{s_0}) < -h_1$ for some $s_0< \sigma$, and the intermediate value theorem yields that there exists $s_1 \in (\tau, s_0)$ such that $\C(p_{s_1}) = -h_1$. This contradicts to the definition of $\sigma$, and hence we have
    \begin{equation}\label{eqn: C' at s''}
        \frac{d}{dt}\C(p_t)\bigg|_{t=\sigma} \leq 0.
    \end{equation}
    If the strict inequality holds in \eqref{eqn: C' at s''}, then let $w_0 = p_{\tau}$ and $w_1 = p_{\sigma}$. Then \eqref{eqn: increasing at t0}, \eqref{eqn: C' at s''} and the definition of $p_{\sigma}$ imply the first and the second part of the lemma with $z_i = \cexp{y_0}{w_i}$ for $i=0,1$.\\    
    Otherwise, we have the equality in \eqref{eqn: C' at s''}. Then, up to an isometry, we may assume $p_{\tau} = 0$, $p_{\sigma} = a e_n$, and $D_p \C(p_{\sigma})=b e_1$ for some $a, b>0$. Then, \eqref{eqn: increasing at t0} implies $\langle D_p \C(0), e_n \rangle = \frac{\partial}{\partial p^n} \C(0)>0$. 
    Using that $\C$ is $C^1$, there exists $\delta_0>0$ such that 
    \begin{equation*}
        \|p\|<\delta_0  \Rightarrow \| D_p \C(p) - D_p \C(0) \|< \frac{1}{2}\langle D_p \C(0), e_n \rangle, 
    \end{equation*}
    In particular, if $\| p \|< \delta_0$, then $\langle D_p \C(p), e_n \rangle>0$. Similarly, we can also obtain $\delta_1>0$ such that
    \begin{equation*}
        \|p-ae_n\| < \delta_1 \Rightarrow \| D_p \C(p) - D_p\C(ae_n)\| < \frac{1}{2}\|D_p\C(ae_n)\|,
    \end{equation*}
    which implies that, if $\| p - ae_n\| < \delta_1$, then $\langle D_p \C(p), e_1\rangle >0$. Let 
    \begin{equation*}
        \begin{aligned}
            &\epsilon_0 = \min\{ \frac{1}{2}\delta_0 , \frac{1}{4}a\}, \\
            &\epsilon_1 = \min\{ \frac{1}{2}\delta_1, \frac{1}{4}a\}.
        \end{aligned}
    \end{equation*}
    Then we have $\epsilon_0 e_n \in B_{r_0}(0)$, $(a-\epsilon_1)e_n \in B_{r_1}(ae_n)$, and $\epsilon_0 < (a-\epsilon_1)$. Moreover, the definitions of $\tau$ and $\sigma$ imply $[\epsilon_0 e_n , (a-\epsilon_1)e_n] \subset \{p \in [\X]_{y_0} | \C(p) > -h_1 \}$. Then the openness (relative to $[\X]_{y_0}$) of the set $\{p \in [\X]_{y_0}| \C(p) > -h_1 \}$ together with the compactness of $[\epsilon_0 e_n , (a-\epsilon_1)e_n]$ and $p_\tau, p_\sigma \in \Int{[\X]_{y_0}}$ implies that there exists $\delta_2>0$ such that 
    \begin{equation}\label{eqn: nbhd of p' segment in sup level set}
        B_{\delta_2}([\epsilon_0 e_n , (a-\epsilon_1)e_n]) \subset \{p \in [\X]_{y_0}| \C(p) > -h_1 \}.
    \end{equation}
     Noting $\frac{\partial}{\partial p^n} \C(0)>0$, we can use the implicit function theorem to obtain $l>0$ and a $C^1$ function $g$ defined on a ($n-1$)-dimensional ball $B^{n-1}_{l}(0)$ such that $\C(\hat{p}, g(\hat{p}) )= -h_1$ for $\hat{p} = (p^1, \cdots , p^{n-1}) \in B^{n-1}_l(0)$. Let $\delta_3>0$ satisfy $\delta_3< \min\{\epsilon_0, \epsilon_1, \delta_2, l\}$ and satisfy the following:
     \begin{equation}\label{eqn: p0 close to 0}
        \begin{aligned}
            & \| (s\hat{e}_1, g(s\hat{e}_1)) \| < \epsilon_0, \quad \forall s \in (0, \delta_3), \\
            & \| (\delta_3\hat{e}_1, g(\delta_3\hat{e}_1))\| < \frac{a\langle D_p \C(0),e_n\rangle}{3\| D_p \C(0) \|}.
        \end{aligned}
     \end{equation}
     Note that it is possible to choose such $\delta_3$ since $g$ is continuous and $g(0)=0$. Let $w_0 = (\delta_3 \hat{e}_1, g(\delta_3 \hat{e}_1))$ where $\hat{e}_1 = (1,0, \cdots, 0) \in \R^{n-1}$, and let $w_1 = ae_n = p_{\sigma}$. Then we have $\C(w_0)=\C(w_1)=-h_1$, which gives the first part of the lemma with $z_i = \cexp{y_0}{w_i}$. In addition, the first equation of the second part of the lemma follows from the following computation:
     \begin{align*}
         \langle D_p\C(w_0), ae_n - \pi_0 \rangle & = a \langle D_p \C(w_0), e_n\rangle +\langle D_p\C(w_0) - D_p \C (p_{\tau}), w_0 \rangle + \langle D_p \C(p_{\tau}), w_0 \rangle \\
         & \geq \frac{a}{2} \langle D_p\C(p_{\tau}), e_n \rangle - \frac{1}{2}\langle D_p\C(p_{\tau}), e_n \rangle \|w_0\| - \| D_p \C(p_{\tau})\|\|w_0\| \\
         & \geq \frac{a}{2} \langle D_p \C(p_{\tau}), e_n \rangle -\frac{3}{2}\| D_p \C(p_{\tau})\|\|w_0\| \\
         & > 0,
     \end{align*}
     where we have use \eqref{eqn: p0 close to 0} in the last strict inequality. The second equation of the second part of the lemma follows from $D_p\C(w_1)=be_1$ with $b>0$.\\
     Next, we claim 
     \begin{equation}\label{eqn: open p segment in sup level set}
     (w_0, w_1) \subset \{ p \in [\X]_{y_0} | \C(p) > -h_1 \}.
     \end{equation}
     Note first that for $w_t \in (w_0, w_1)$, we have
     \begin{equation*}
          w_t=((1-t)\delta_3 \hat{e}_1, (1-t)g(\delta_3\hat{e}_1)+ at). 
     \end{equation*}
     We divide the proof of the claim into three cases depending on the value of the last coordinate of $w_t$:  $(1-t)g(\delta_3\hat{e}_1)+ at$.\\
     \emph{Case 1}. If $(1-t)g(\delta_3\hat{e}_1) + at < \epsilon_0$, then \eqref{eqn: p0 close to 0} and $\epsilon_0 < \frac{1}{2}\delta_0$ imply $w_t \in B_{\delta_0}(0)$. Then, thanks to the definition of $\delta_0$, we observe
     \begin{equation*}
         \C(w_t) +h_1 = \int_0^1 \langle D_p \C(sw_t + (1-s) G_t , \|w_t- G_t\|e_n \rangle >0.
     \end{equation*}
     where $G_t = ((1-t)\hat{e}_1, g((1-t)\hat{e}_1))$. Hence, $w_t \in \{ p \in [X]_{y_0} | \C(p) > -h_1 \}$ for any $t\in (0,1)$ such that $(1-t)g(\delta_3\hat{e}_1) + at < \epsilon_0$.\\
     \emph{Case 2}. If $(1-t)g(\delta_3\hat{e}_1) + at > (a-\epsilon_1)$, then $\delta_3< \frac{1}{2}\delta_1$ implies $w_t \in B_{\delta_1}(w_1) = B_{\delta_1}(ae_n)$. Noting that the point $P_t=((1-t)g(\delta_3 \hat{e}_1) + at)e_n$ is also in the ball $B_{\delta_1}(w_1)$, \eqref{eqn: p' segment in sup level set} and the definition of $\delta_1$ imply
     \begin{equation*}
         \C(w_t)-\C(P_t) = \int_0^1 \langle D_p \C((1-s)w_t + sP_t), (1-t)\delta_3 e_1 \rangle ds >0,
     \end{equation*}
     and we obtain $\C(w_t) > \C(P_t) > -h_1$.\\
     \emph{Case 3}. Otherwise, we have $ \epsilon_0 \leq (1-t)g(\delta_3\hat{e}_1)+ at \leq a-\epsilon_1$. Then $\delta_3 < \delta_2$ imply $w_t \in B_{\delta_2}([\epsilon_0 e_n, (a-\epsilon_1) e_n])$, we obtain $\C(w_t) > -h_1$ from \eqref{eqn: nbhd of p' segment in sup level set}, and we conclude the claim.\\
     Now, we take $\theta' \in (0, \frac{\pi}{2})$ such that 
     \begin{equation*}
         \cos \theta' < \min\left\{ \frac{\langle D_p\C(w_0), w_1 - w_0\rangle}{\| D_p \C(w_0)\|\|w_1 - w_0\|}, \frac{\langle D_p\C(w_1), w_0 - w_1\rangle}{\| D_p \C(w_1)\|\|w_0 - w_1\|} \right\}.
     \end{equation*}
     Then by Lemma \ref{lem: cone in c-affine section} there exist $\rho_0, \rho_1>0$ such that
     \begin{equation*}
         \left( \K^{\rho_i}_{\theta', v_i} + p_i \right) \subset \{ p \in [\X]_{y_0} | \C(p) \geq -h_1\}
     \end{equation*}
     for $i=0,1$ where $v_i = D_p\C(w_i)$. Let us use the convention $w_2=w_0$, and let $\theta_i$ be in $[0,\frac{\pi}{2})$ such that
     \begin{equation*}
         \langle D_p \C(w_i) , w_i - w_{i+1} \rangle = \| D_p \C(w_i) \| \| w_i - w_{i+1} \| \cos \theta_i
     \end{equation*}
     for $i=0,1$. Note that the left hand side of the above equality is positive so that $\theta_i$ is indeed in $[0, \frac{\pi}{2})$. Then we define $\theta'' = \theta' - \max\{\theta_0, \theta_1\}$. Thanks to the definition of $\theta'$, we have $\theta'' \in (0, \frac{\pi}{2})$. Also, letting $\rho' = \min\{ \rho_0, \rho_1\}$ and $\xi=w_1-w_0$, we have
     \begin{equation*}\label{eqn: small cones in section}
     \begin{aligned}
         & \left( \K^{\rho'}_{\theta'', \xi} +w_0 \right) \subset \left( \K^{\rho_0}_{\theta', v_0} + w_0 \right) \subset \{ p \in[\X]_{y_0} | \C(p) \geq -h_1 \}, \\
         & \left( \K^{\rho'}_{\theta'', -\xi} +w_1 \right) \subset \left( \K^{\rho_1}_{\theta', v_1} + p_1 \right) \subset \{ p \in [\X]_{y_0} | \C(p) \geq -h_1 \}.
     \end{aligned}
     \end{equation*}
     Then, we choose $\epsilon>0$ such that $\|w_0 - w_{\epsilon}\| = \| w_{1-\epsilon}-w_1 \| < \rho'$. Note that $w_{\epsilon} \in \Int{\K^{\rho'}_{\theta'', \xi}}$ and $w_{1-\epsilon} \in \Int{\K^{\rho'}_{\theta'', -\xi}}$. Hence, there exists $r_1>0$ such that
     \begin{equation*}
         B_{r_1}(w_{\epsilon}) \subset \K^{\rho'}_{\theta'', \xi} \quad \textrm{and} \quad B_{r_1}(w_{1-\epsilon}) \subset \K^{\rho'}_{\theta'', -\xi}.
     \end{equation*}
     Also, by the claim \eqref{eqn: open p segment in sup level set}, we obtain $[w_{\epsilon}, w_{1-\epsilon}] \subset \{ p \in [\X]_{y_0} | \C(p) > -h_1 \}$. Then using the openness of $\{ p \in [\X]_{y_0} | \C(p) > -h_1 \}$ together with the compactness of $[w_{\epsilon}, w_{1-\epsilon}]$, we obtain $r_2>0$ such that 
     \begin{equation*}
         B_{r_2}([w_{\epsilon}, w_{1-\epsilon}]) \subset \{ p \in [\X]_{y_0} | \C(p)>-h_1 \}.
     \end{equation*} 
     Let $r = \min\{ r_1, r_2 \}$. Then, together with \eqref{eqn: small cones in section}, we obtain
     \begin{equation*}
        \begin{aligned}
            \left[w_0, w_1\right] & \subset \left( \K^{\rho'}_{\theta'', \xi} +w_0 \right) \cup \left( \K^{\rho'}_{\theta'', -\xi} +w_1 \right) \cup B_r([w_{\epsilon}, w_{1-\epsilon}]) \\ & \subset \{ p \in [\X]_{y_0} | \C(p) \geq -h_1 \}.
        \end{aligned}
     \end{equation*}
     Now, we define $\theta$ by 
     \begin{equation}\label{eqn: choosing theta}
         \cos(\theta) = \max\left\{ \frac{\langle w_{1/2}-w_0, \frac{1}{2}rv + \frac{1}{2}\xi \rangle}{\|w_{1/2}-w_0 \| \| \frac{1}{2}rv + \frac{1}{2}\xi\|},  \frac{7}{8} \right\} = \max\left\{ \frac{\|\xi\|}{\| rv+\xi \|}, \frac{7}{8}\right\},
     \end{equation}
     where $v$ is a unit vector which is perpendicular to $\xi$. Observe that $\theta$ does not change if we replace $v$ with another unit vector that is perpendicular to $\xi$. Let
     \begin{equation}\label{eqn: choosing rho}
         \rho = \left\| \frac{1}{2}rv + \frac{1}{2}\xi \right\|.
     \end{equation}
     Let $p \in (\K^{\rho}_{\theta, \xi}+w_0)$, then we may write $p = t \xi + s v + w_0$ for some $v$ that is perpendicular to $\xi$. Then we observe 
     \begin{equation*}
         \frac{\langle p-w_0, \xi\rangle}{\| p-w_0 \| \| \xi \|} = \frac{\tau \| \xi \| }{\| t \xi + s v\|} \geq \cos(\theta) \geq  \frac{\| \xi \|}{ \|  rv + \xi \| },
     \end{equation*}
     which implies $\frac{s}{t} \leq r$. Hence, if $0 < t < \epsilon$, we obtain
     \begin{equation*}
        \begin{aligned}
             \frac{\langle p-w_0, \xi\rangle}{\| p-w_0 \| \| \xi \|} &= \frac{ \| \xi \|}{\|  \xi + \frac{s}{t} v \| } \\
            &\geq \frac{\| \xi \|}{\| \xi + r v \|} = \frac{\langle \epsilon \xi + \epsilon rv, \xi \rangle}{ \| \epsilon \xi + \epsilon rv \| \| \xi \|} = \frac{\langle (w_\epsilon +\epsilon rv) - w_0, \xi \rangle }{\| (w_\epsilon +\epsilon rv) - w_0 \| \| \xi\|}.
        \end{aligned}
     \end{equation*}
     Since $B_r(w_\epsilon) \subset \K^{\rho'}_{\theta'', \xi}+w_0$, we obtain 
     \begin{equation*}
         \frac{\langle p-w_0, \xi \rangle}{\| p-w_0 \| \| \xi \|} \geq \cos(\theta'').
     \end{equation*}
     In particular, we have $(\K^\rho_{\theta, \xi}+w_0) \subset (\K^{\rho'}_{\theta'', \xi}+w_0)$.\\
     On the other hand, if $\epsilon \leq t \leq \frac{1}{2}$, then 
     \begin{equation*}
         p=(t \xi +w_0) + s v \in B_{r}([w_{\epsilon}, w_{1-\epsilon}])
     \end{equation*}
     as $t\xi + w_0 = w_t$ and $s \leq tr \leq \frac{1}{2}r$. Therefore, we have
     \begin{equation*}
         \K^{\rho}_{\theta,\xi }+w_0 \subset \left(\K^{\rho'}_{\theta'',\xi}+w_0 \right) \cup B_{r}([w_{\epsilon}, w_{1-\epsilon}]).
     \end{equation*}
     Similarly, we can also obtain 
     \begin{equation*}
         \K^{\rho}_{\theta,-\xi }+w_1 \subset \left( \K^{\rho'}_{\theta'',-\xi}+w_1 \right) \cup B_{r}([w_{\epsilon}, w_{1-\epsilon}]).
     \end{equation*}
     Then, noting $\rho > \frac{1}{2} \| \xi \|$, we obtain $[w_0, w_{\frac{1}{2}}] \subset (\K^\rho_{\theta, \xi} + w_0)$ and $[w_{\frac{1}{2}}, w_1] \subset (\K^\rho_{\theta, -\xi} + w_1)$, and thus
     \begin{equation}\label{eqn: double cone in the stick}
        \begin{aligned}
            {[w_0, w_1]} & \subset \left(\K^{\rho}_{\theta, \xi}+w_0\right) \cup \left(K^{\rho}_{\theta,-\xi}+w_1 \right) \\
            & \subset \left( \K^{\rho'}_{\theta'', \xi} +w_0 \right) \cup \left( \K^{\rho'}_{\theta'', -\xi} +w_1 \right) \cup B_r([w_\epsilon, w_{1-\epsilon}])\\
            & \subset \left[ \csec{f_0}{f_1}\right]_{y_0},
        \end{aligned}
     \end{equation}
     which yields the inclusion in the last part of the lemma.
\end{proof}

\begin{Rmk}\label{rmk: outside the double cone}
    Thanks to the choice of $\theta$ \eqref{eqn: choosing theta} and $\rho$ \eqref{eqn: choosing rho}, we have
    \begin{equation*}
        [\X]_{y_0} \setminus \left( \left( \K^\rho_{\theta, \xi}+w_0 \right) \cup \left( K^{\rho}_{\theta,-\xi} + w_1 \right) \right) \subset \left( [\X]_{y_0}\setminus\left( \K_{\theta,\xi}+w_0\right) \right) \cup \left( [\X]_{y_0} \setminus (\K_{\theta, -\xi}+w_1)\right).
    \end{equation*}
    Note that the cones in the right hand side of the above equation do not have the radius $\rho$. Moreover, \eqref{eqn: double cone in the stick} yields
    \begin{equation*}
        \left( \left(\K^{\rho}_{\theta, \xi}+w_0\right) \cup \left(K^{\rho}_{\theta,-\xi}+w_1 \right) \right) \setminus\{w_0,w_1\} \subset \Int{\left[ \csec{f_{0}}{f_{1}} \right]_{y_0}}.
    \end{equation*}
\end{Rmk}

Now we prove the second part of Theorem \ref{thm: main theorem}. Assuming that the cost function $c$ does not satisfy Loeper's property, we will construct an alternative $c$-convex function that is not a $c$-convex function. The idea is to construct a $c$-convex function $\phi$ which have a disconnected section $\csec{\phi}{f}$ for some $c$-affine function $f$ using \ref{lem: non Loeper set} and Lemma \ref{lem: small tilting}. Then we define a function $\phi'$ using $\phi$ which is equal to $f$ in some connected component of $\csec{\phi}{f}$, but not in entire section $\csec{\phi}{f}$, and equal to $\phi$ outside the section. With a careful construction, $\phi'$ can be an alternative $c$-affine function, but $\phi'$ cannot have a $c$-subdifferential in the interior of the set $\{x|\phi'(x)=f(x) \}$.

\begin{Thm}\label{thm: main 2}
Suppose that $c$ does not satisfy Loeper's property. Then there exists an alternative $c$-convex function that is not $c$-convex.
\end{Thm}
\begin{proof}
    We will construct an alternative $c$-convex function that does not have a $c$-subdifferential at some point. Let $z_0, z_1, y_0, y_1, w_0, w_1, \xi, \theta, \rho$ and $h_1$ be from Lemma \ref{lem: non Loeper set}. Also, fix $l>0$ such that $ B_l(y_i) \subset \Y$ for $i = 0,1$. We divide the proof into several steps.

    \noindent \textbf{\emph{Step 1)}} Let 
    \begin{equation}\label{eqn: main 2: C}
        \C_{ij}(p,q) = -c(\cexp{y_j}{p}, \cexp{z_i}{q}) + c(z_i,\cexp{z_i}{q})
    \end{equation}
    for $i=0,1$, $j=0,1$. Let $\tilde{\theta} = \frac{\pi}{2}-\frac{\theta}{4}$ and $r<\min\{\frac{l}{2\alpha L}, \mu(\tilde{\theta})\}$, where $\mu(\tilde{\theta})$ is defined in Lemma \ref{lem: small tilting}. Then, for any $v$ such that $\| v \| < r$, we have
    \begin{equation}\label{eqn: main 2: Mv est}
        \| (-D^2_{xy}c(z_i, y_j) ) v \| \leq \| -D^2_{xy}c(z_i, y_j)\|_F \| v \| < \alpha \times \frac{l}{2\alpha L} =\frac{l}{2L}.
    \end{equation}
    Let $M_{ij} = -D^2_{xy}c(z_i,y_j)$. Noting that $B_l(y_j) \subset \Y$ and \eqref{eqn: Lipschitz}, we have $B_{\frac{l}{2L}}(q^i_j) \subset [\Y]_{z_i} $, where $q^i_j = -D_x c(z_i, y_j)$. Hence, we have
    \begin{equation}\label{eqn: main 2: q+Mv in Y}
        q^i_j+M_{ij}v \in [\Y]_{z_i}.
    \end{equation}
    Therefore, ${\C}_{ij}(p, q^i_j+M_{ij}v)$ is well-defined for $\|v\|<r$. We also have $\| v\| < \mu(\tilde{\theta})$ so that, by Lemma \ref{lem: small tilting}, we obtain
    \begin{equation}\label{eqn: main 2: small tilting est}
        \begin{aligned}
            &p \in (\K_{\tilde{\theta}, v}+w_i^j)\cap [\X]_{y_j} \\
            \Rightarrow& \frac{1}{2} \| p-w_i^j\|\|v\| \cos(\tilde{\theta}) \leq {\C}_{ij}(p, q^i_j+M_{ij}v) - {\C}_{ij}(p, q^i_j) \leq \frac{3}{2} \| p-w_i^j\|\|v\| \cos(\tilde{\theta}),
        \end{aligned}
    \end{equation}
    and
    \begin{equation}\label{eqn: main 2: small tilting est 2}
        \begin{aligned}
            &p \in (\K_{\tilde{\theta},-v}+w_i^j)\cap[\X]_{y_j} \\
            \Rightarrow & -\frac{3}{2} \| p-w_i^j\|\|v\| \cos(\tilde{\theta}) \leq {\C}_{ij}(p, q^i_j+M_{ij}v) - {\C}_{ij}(p, q^i_j) \leq -\frac{1}{2} \| p-w_i^j\|\|v\| \cos(\tilde{\theta}),
        \end{aligned}
    \end{equation}
    where $w_i^j = -D_y c(w_i,y_j)$. In particular $w_i = w^0_i$, and we have
    \begin{equation*}
        (\K_{\tilde{\theta},v}+w_i) \cap [\X]_{y_0} \subset \{p \in [\X]_{y_0}| {\C}_i(p,q^i_0+M_{i0}v) \geq {C}_i(p,q^i_0) \}.
    \end{equation*}
    Now, let $V_0 = \K_{\frac{\pi-\theta}{2},-\xi}^r$ and $V_1 = \K^r_{\frac{\pi-\theta}{2},\xi}$. Also, denote
    \begin{equation*}
        \tilde\K =  \left( \K^\rho_{\theta, \xi} + w_0 \right) \cup \left( \K_{\theta,-\xi}^\rho + w_1 \right).
    \end{equation*}
    Then, using Remark \ref{rmk: outside the double cone}, we observe
    \begin{equation}\label{eqn: main 2: cone covering}
        \left(\bigcup_{v\in V_0}(\K_{\tilde{\theta},v}+w_0)\right) \cup \left(\bigcup_{v \in V_1}(\K_{\tilde{\theta},v}+w_1) \right) \cup \tilde{\K} \supset [\X]_{y_0}.
    \end{equation}
    On the other hand, for any $v \in V_0$, we have
    \begin{equation}\label{eqn: main 2: pt in cone0}
        \begin{aligned}
            \langle w_t - w_0, -v \rangle & = t \langle -\xi , v \rangle \geq  t\| \xi \| \| v \| \cos (\frac{\pi-\theta}{2})\\
            & \geq \| w_t -w_0 \| \| v \| \cos(\frac{\pi}{2}-\frac{\theta}{4})
        \end{aligned}
    \end{equation}
    so that $[w_0,w_1] \subset (\K_{\tilde{\theta},v}+w_0)$. Similarly, for any $v \in V_1$, we have
    \begin{equation}\label{eqn: main 2: pt in cone1}
        \begin{aligned}
            \langle w_t - w_1, -v \rangle &= (1-t) \langle \xi, v \rangle \geq (1-t) \| \xi \| \| v \| \cos (\frac{\pi-\theta}{2}) \\
            & \geq \| w_t - w_1 \| \| v \| \cos( \frac{\pi}{2} - \frac{\theta}{4}),
        \end{aligned}
    \end{equation}
    so that $[w_0,w_1] \subset (\K_{\tilde{\theta},-v}+w_1)$. Therefore, we have
    \begin{equation*}
        [w_0, w_1] \subset \left( \bigcap_{v \in V_0} (\K_{\tilde{\theta}, -v}+w_0)\right) \cap \left( \bigcap_{v\in V_1} (\K_{\tilde{\theta},-v}+w_1)\right).
    \end{equation*}
    Moreover, when $t=1$ in \eqref{eqn: main 2: pt in cone0} and $t=0$ in \eqref{eqn: main 2: pt in cone1}, we obtain strict inequalities. Then, using the compactness of $V_0$ and $V_1$, we can obtain $ \rho_0 >0$ such that 
    \begin{equation*}
        B_{\rho_0}(w_0) \subset \bigcap_{v \in V_1}(\K_{\tilde{\theta}, -v}+w_1) \quad \textrm{and} \quad B_{\rho_0}(w_1) \subset \bigcap_{v \in V_0}(\K_{\tilde{\theta},-v}, w_0).
    \end{equation*}
    In addition, since $w_0 \in \Int{\K_{\theta, -\xi}+w_1}$, there exists $\rho_1$ such that 
    \begin{equation}\label{eqn: main 2: ball at p0 in theta cone at p1}
        B_{\rho_1}(w_0) \subset \K_{\theta, -\xi}+w_1.
    \end{equation}

    \noindent \textbf{\emph{Step 2)}} Let $\mathcal{F}_i(p) = \{ {\C}_{i0}(p,q^i_0+M_{i0}v)-c(z_i, y_0) | v \in V_i \}$. Denote $x(p) = \cexp{y_0}{p}$ for $p \in [\X]_{y_0}$. We define a function $\Phi:[\X]_{y_0}\to \R$ by
    \begin{equation*}
        \Phi(p) = \sup (\mathcal{F}_0(p) \cup \mathcal{F}_1(p)\cup \{ -c(x(p), y_1)+h_1\}).
    \end{equation*}
    By the definition of ${\C}_{ij}$ \eqref{eqn: main 2: C}, we have $\mathcal{F}_i(w_0) = -c(z_0, y_0)$ for $i=0,1$. In addition, since $z_0, z_1, y_0, y_1$ are from Lemma \ref{lem: non Loeper set}, we also have $-c(z_i, y_1) + h_1 = -c(z_i, y_0)$. Therefore, we have
    \begin{equation}\label{eqn: main 2: Phi=-c0 at xi}
        \Phi(w_i) = -c(w_i, y_0)
    \end{equation}
    for $i = 0,1$. On the other hand, \eqref{eqn: main 2: small tilting est}, \eqref{eqn: main 2: small tilting est 2}, \eqref{eqn: main 2: cone covering}, and the third part of Lemma \ref{lem: non Loeper set} shows 
    \begin{equation*}
        \Phi(p) \geq -c(x(p), y_0),
    \end{equation*}
    and the equality hold only when $p = w_i$, $i=0,1$. Let $\delta<\rho$, and consider $\tilde{\K}\setminus (B_\delta(w_0) \cup B_\delta(w_1))$. Noting Remark \ref{rmk: outside the double cone}, $\tilde{\K}\setminus (B_\delta(w_0) \cup B_\delta(w_1))$ is compactly contained in $\Int{\left[ \csec{f_{0}}{f_{1}}\right]_{y_0}}$, where $f_{i}$ are from Lemma \ref{lem: non Loeper set}. Therefore, we have
    \begin{equation}\label{eqn: main 2: epsilon delta}
        \epsilon_\delta := \inf\{ f_{1}(x)-f_{0}(x) | -D_yc(x,y_0) \in \tilde{\K}\setminus (B_\delta(w_0) \cup B_\delta(w_1)) \} >0.
    \end{equation}
    Also, \eqref{eqn: main 2: small tilting est} yields 
    \begin{equation*}
        \begin{aligned}
            & p \in \left( \bigcup_{v\in V_i} \K_{\tilde{\theta},v} \right) \cap [\X]_{y_0} \setminus B_\delta(w_i) \\
        \Rightarrow & -c(x(p),y^i_v) +c(w_i, y_v^i)-c(w_i, y_0)+ c(x(p), y_0) \geq \frac{\delta r}{2} \cos(\tilde{\theta}),
        \end{aligned}
    \end{equation*} 
    where $y_v^i = \cexp{w_i}{q^i_0+M_{i0}v}$. Therefore, for any $0<\epsilon < \min\{\epsilon_\delta, \frac{\delta r}{2}\cos(\tilde{\theta}) \}$, we have
    \begin{equation*}
        \overline{\{ p | \Phi(p)< -c(x(p), y_0) + \epsilon \}} \subset B_\delta(w_0) \cup B_\delta(w_1). 
    \end{equation*}
    Hence, if we define
    \begin{equation*}
        \Phi_\epsilon(p) = \left\{ \begin{matrix}
            \max\{\Phi(p), -c(x(p),y_0)+\epsilon\} & \textrm{for }p \in B_\delta(w_1) \\
            \Phi(p) & \textrm{otherwise}
        \end{matrix}\right.,
    \end{equation*}
    and
    \begin{equation*}
        \phi_\epsilon(x) = \Phi_\epsilon(-D_y c(x, y_0) ),
    \end{equation*}
    then $\phi_\epsilon$ is a continuous function. Moreover, the sublevel set $\{x \in \X | \phi_\epsilon(x) \leq -c(x,y_0)+\epsilon \}$ has non-empty interior. Also, if $\phi_\epsilon(x)< -c(x,y_0) +\epsilon$, then $-D_y c(x,y_0) \in B_\delta(w_0)$ by \eqref{eqn: main 2: epsilon delta}. In the following steps, we show that the function $\phi_\epsilon$ is an alternative $c$-convex function with appropriate choice of $\delta>0$ and $\epsilon>0$.

    \noindent \textbf{\emph{Step 3)}} In this step and the following step, we choose $\tilde{\delta}$ and $\tilde{\epsilon}$, the respective values of $\delta$ and $\epsilon$ used in defining $\phi_\epsilon$ in the previous step. In the proof, we may use the notation $v$ and $y_v$ to denote several different points in $\Y$ for convenience, depending on the context.\\
    Note first that if $v \in \Rn$ is such that $ \| v \| \leq \frac{l}{2 \alpha L}$, then \eqref{eqn: main 2: Mv est} yields that we still have \eqref{eqn: main 2: q+Mv in Y}. Define $\eta_0$ by
    \begin{equation*}
        \cos(\eta_0) := \frac{\langle  M_{10}^{-1}(q^1_1 - q^1_0),w_0 -w_1\rangle}{\| M_{10}^{-1}( q_1^1 - q_0^1 )\| \|(w_0-w_1) \|}.
    \end{equation*}
    Then the second part of Lemma \ref{lem: non Loeper set} shows that $\eta_0 \in [0, \frac{\pi}{2})$. Note that the following map is continuous in $p$ and $y$ variables away from $p_0$ and $y_0$ respectively:
    \begin{equation*}
        (p,y) \mapsto \frac{\langle (-D^2_{xy}c(x(p),y_0))^{-1}( -D_x c(x(p), y)+D_x c(x(p), y_0)), w_0 -w_1\rangle}{\| (-D^2_{xy}c(x(p),y_0))^{-1}( -D_x c(x(p), y)+D_x c(x(p), y_0)) \| \| w_0-w_1 \|}.
    \end{equation*}
    Using the continuity, we can obtain $\delta_0<\frac{1}{4}\|w_1 - w_0\|$ and $\sigma_0$ such that if $p \in B_{\delta_0}(w_1)$ and $y \in B_{\sigma_0}(y_1)$, then
    \begin{equation}\label{eqn: main 2: angle between gradient and -w}
         \frac{\langle (-D^2_{xy}c(x(p),y_0))^{-1}( -D_x c(x(p), y)+D_x c(x(p), y_0)), w_0 -w_1\rangle}{\| (-D^2_{xy}c(x(p),y_0))^{-1}( -D_x c(x(p), y)+D_x c(x(p), y_0)) \| \| w_0-w_1 \|} > \cos(\frac{\eta_0}{2}+\frac{\pi}{4}).
    \end{equation}
    Let $\eta_1 = \frac{\eta_0}{2}+ \frac{\pi}{4}$, and note that $\eta_0 < \eta_1 < \frac{\pi}{2}$. On the other hand, for $q^0 = -D_x c(z_0, y)$, we have
    \begin{equation*}\label{eqn: main 2: q dist compare}
        \begin{aligned}
            \| y - y_1\| & =\left \| \int_0^1 \frac{d}{ds}y_s' ds \right \|\\
            &=\left\| \int_0^1 \left( -D^2_{xy}c(z_0, y_s') \right)^{-1}(q^0 - q_1^0)ds \right\| \\
            & \leq \beta \| q^0 - q^0_1\|,
        \end{aligned}
    \end{equation*}
    where $y_s' \in \cseg{z_0}[y_1, y]$. Therefore if $\| v \| < \min\{\frac{l}{2\alpha L}, \frac{\sigma_0}{\alpha \beta} \}$, letting $y_v = \cexp{z_0}{q_1^0 + M_{01}v}$, we obtain
    \begin{equation*}
            \| y_v - y_1 \| \leq \beta \| M_{01} v \| < \alpha \beta \times \frac{\sigma_0}{\alpha \beta} = \sigma_0. 
    \end{equation*}
    In addition, by \eqref{eqn: Lipschitz}, we have
    \begin{equation*}
        \frac{1}{L}\| y_1 - y_0 \| \leq   \| q^1_1 - q^1_0 \| = \| -D_x c(z_1, y_1) + D_x c(z_1, y_0) \| \leq L \|y_1 - y_0 \|.
    \end{equation*}
    Therefore, the triangle inequality yields
    \begin{equation*}
        \frac{1}{L} \| y_1 - y_0 \| - \sigma_0 \leq \| y_v - y_0\| \leq L \|y_1 - y_0 \| + \sigma_0.
    \end{equation*}
    Let $q_v^p = -D_x c(x(p), y_v)$ and $q_0^p = -D_x c(x(p),y_0)$. Then \eqref{eqn: Lipschitz} implies
    \begin{equation*}
        \frac{1}{L^2} \| y_1 - y_0 \| - \frac{1}{L} \sigma_0 \leq \| q_v^p - q_0^p \| \leq L^2 \| y_1 - y_0 \| +L\sigma_0.
    \end{equation*}
    Then, letting
    \begin{equation*}
        \begin{aligned}
            l_0 & = \frac{1}{L^2} \| y_1 - y_0 \| - \frac{1}{L} \sigma_0, \\
            l_1 & = L^2 \| y_1 - y_0 \| +L\sigma_0, \\  
            \eta_2 & = \frac{\eta_1}{2} + \frac{\pi}{4}, \\ 
            \xi^p_v & = (-D^2_{xy}c(x(p),y_0))^{-1}(q_v^p - q^p_0), \\        
        \end{aligned}
    \end{equation*}
    and using Lemma \ref{lem: cone in c-affine section}, we obtain $\tilde{\rho}>0$ such that 
    \begin{equation*}
        \K^{\tilde\rho}_{\eta_2, \xi_v^p}+ p_1 \subset \{ p \in [\X]_{y_0} | -c(x(p), y_v)+ c(x(p_1),y_v) \geq -c(x(p), y_0)+ c(x(p_1),y_0) \},
    \end{equation*}
    for any $p_1 \in B_{\delta_0}(w_1)$ (recall Remark \ref{rmk: dependency of rho}). Then \eqref{eqn: main 2: angle between gradient and -w} and the above inclusion yield that, letting $\tilde{\eta} = \eta_2 - \eta_1$, we have
    \begin{equation}\label{eqn: main 2: cone in -w direction in sections}
        \begin{aligned}
            \K^{\tilde\rho}_{\tilde{\eta}, -\xi}+p_1 & \subset \K^{\tilde{\rho}}_{\eta_2, \xi_v^p}+ p_1 \\
            & \subset \{ p \in [\X]_{y_0} | -c(x(p), y_v)+ c(x(p_1),y_v) \geq -c(x(p), y_0)+ c(x(p_1),y_0) \}.
        \end{aligned} 
    \end{equation}
    for any $p_1 \in B_{\delta_0}(w_1)$ and $\| v \| < \min\{\frac{l}{2\alpha L}, \frac{\sigma_0}{\alpha \beta} \}$. Moreover, since $\delta_0 < \frac{1}{4}\| w_1 - w_0 \|$, for any $p_1 \in B_{\delta_0}(w_1)$ and $p = p_1-d \frac{\xi}{\| \xi \|}$ where $d > \frac{1}{4} \| w_0 - w_1 \|$, we have
    \begin{equation*}
        B_{\gamma'}(p) \subset \K_{\frac{\tilde{\eta}}{4}, -\xi} + p_1,
    \end{equation*}
    where $\gamma' = \frac{1}{2}\| w_1 - w_0 \| \sin(\frac{\tilde{\eta}}{4})$. Fix $\gamma_0<  \frac{1}{2} \gamma'$. Then, for any $p_1 \in B_{\min\{\delta_0, \gamma_0 \}}(w_1)$ and $p_0 \in B_{\min\{ \delta_0, \gamma_0 \}}(w_0)$, we have
    \begin{equation}\label{eqn: main 2: gamma 1}
        B_{\min\{\delta_0, \gamma_0 \}}(p_0) \subset B_{\gamma'} (w_0) \subset \K_{\frac{\tilde{\eta}}{4}, -w}+p_1 .
    \end{equation}
    Next, let $\xi^1 = w^1_1 - w_0^1  = -D_y c(z_1, y_1) +D_y c(z_0, y_1)$. Note that \eqref{eqn: Lipschitz} implies
    \begin{equation*}
        \| \xi^1 \| = \| w^1_1 - w^1_0\| = \| -D_y c(x_1, y_1) +D_y c(x_0, y_1)\| \geq \frac{1}{L}\|z_1 - z_0\|.
    \end{equation*}
    Fix $\eta \in (0,\frac{\pi}{2})$ and let $\delta_1 >0$ be such that $\delta_1 < \frac{1}{2L} \| z_1 - z_0 \|$ and
    \begin{equation}\label{eqn: main 2: delta1}
        B_{L^2 \delta_1}(w_1^1) \subset \K_{\eta, \xi^1}+w_0^1.
    \end{equation}
    Note that $-D_y c(x({B_{\delta_1}(w_1)}),y_1) \subset B_{L^2 \delta_1}(w_1^1)$ by \eqref{eqn: Lipschitz} (recall $x(p) = \cexp{y_0}{p}$). Let $0< r' < \min\{ \frac{l}{2 \alpha L}, \mu(\eta), \frac{\sigma_0}{\alpha\beta} \}$. Then, for $ \tilde{v}= \frac{r'\xi^1}{2\|\xi^1\|}$ and $p^1 \in B_{L^2 \delta_1}(w_1^1)$, Lemma \ref{lem: small tilting} implies that we have
    \begin{equation}\label{eqn: main 2: small tilting y1 positive direction}
        \begin{aligned} 
            \C_{01}(p^1, q^0_1 + M_{01}\tilde{v}) - \C_{01}(p^1, q_1^0) & \geq \frac{r'}{4}\cos(\eta) \| p^1 - w^1_0\| \\
            & \geq \frac{r'}{8L} \cos(\eta) \| z_1 - z_0 \|, 
       \end{aligned}
    \end{equation}
    and
    \begin{equation}\label{eqn: main 2: small tilting y2 negative direction}
        \begin{aligned}
            \C_{01}(p^1, q^0_1 - M_{01}\tilde{v}) - \C_{01}(p^1, q_1^0) & \leq -\frac{r'}{4}\cos(\eta) \| p^1 - w^1_0\| \\
            & \leq -\frac{r'}{8L} \cos(\eta) \| z_1 - z_0 \|.
        \end{aligned}
    \end{equation}
    Define
    \begin{equation*}
        \lambda(\delta) = \sup \left\{ |\C_{10}(p,q^1_1) - \C_{10}(p, q^1_0)| \big| p \in \overline{B_{\delta}(w_1)} \right\}.
    \end{equation*}
    Then $\lambda$ is a continuous function that decreases as $\delta$ decreases. In addition, the first part of Lemma \ref{lem: non Loeper set} implies $\lim_{\delta \to 0}\lambda(\delta) = 0$. Therefore, there exists $\delta_2>0$ such that 
    \begin{equation*}
        \lambda(\delta_2) < \frac{r'}{16L} \cos(\eta) \| z_1 -z_0 \|.
    \end{equation*}
    Then the above inequality together with \eqref{eqn: main 2: small tilting y1 positive direction} and \eqref{eqn: main 2: small tilting y2 negative direction} implies that for any $p \in B_{\min\{ \delta_1, \delta_2\}}(w_1)$, denoting $p^1 = -D_y c(x(p), y_1) $, we have
    \begin{equation*}
        \begin{aligned}
            \C_{01}(p^1, q_1^0 + M_{01}\tilde{v}) -\C_{01}(p^1, q^0_0) & \geq \C_{01}(p^1, q^0_1 + M_{01}\tilde{v}) - \C_{01}(p^1, q_1^0) - \lambda(\delta_2) \\
            & \geq \frac{r'}{8L} \cos(\eta) \| z_1 - z_0 \| - \frac{r'}{16L} \cos(\eta) \| z_1 -z_0 \| \\
            & = \frac{r'}{16L} \cos(\eta) \| z_1 -z_0 \|,
        \end{aligned}
    \end{equation*}
    and
    \begin{equation*}
        \begin{aligned}
            \C_{01}(p^1, q_1^0 - M_{01}\tilde{v}) -\C_{01}(p^1, q^0_0) & \leq \C_{01}(p^1, q^0_1 - M_{01}\tilde{v}) - \C_{01}(p^1, q_1^0) + \lambda(\delta_2) \\
            & \leq -\frac{r'}{8L}\cos(\eta) \| z_1 - z_0\| + \frac{r'}{16L} \cos(\eta) \| z_1 -z_0 \| \\
            & = -\frac{r'}{16L} \cos(\eta) \| z_1 -z_0 \|.
        \end{aligned}
    \end{equation*}
    Then, denoting $y_{v} = \cexp{x_0}{q_1^0 + M_{01}{v}}$, and using the first part of Lemma \ref{lem: non Loeper set}, the above inequalities become
    \begin{equation}\label{eqn: main 2: small tilting y1 up}
        -c(x(p), y_{\tilde{v}} ) + c(z_0, y_{\tilde{v}}) + c(x(p), y_0) -c(z_0, y_1)+h_1 \geq  \frac{r'}{16L}\cos(\eta) \| z_1 - z_0 \|,
    \end{equation}
    and
    \begin{equation}\label{eqn: main 2: small tilting y1 down}
        -c(x(p) , y_{-\tilde{v}}) + c(z_0, y_{-\tilde{v}}) + c(x(p), y_0) - c(z_0, y_1)+h_1 \leq -\frac{r'}{16L} \cos(\eta) \| z_1 - z_0 \|.
    \end{equation}
    Then, noting that the functions
    \begin{equation*}
        p_0 \mapsto -c(x(p_1), y_{\tilde{v}} ) + c(x(p_0), y_{\tilde{v}}) + c(x(p_1), y_0) -c(x(p_0), y_1), 
    \end{equation*}
    and
    \begin{equation*}
        p_0 \mapsto -c(x(p_1) , y_{-\tilde{v}}) + c(x(p_0), y_{-\tilde{v}}) + c(x(p_1), y_0) - c(x(p_0), y_1)
    \end{equation*}
    are (Lipschitz) continuous in $p_0$ variable, \eqref{eqn: main 2: small tilting y1 up} and \eqref{eqn: main 2: small tilting y1 up} imply that we can obtain $\gamma_1>0$ such that if $p_0 \in B_{\gamma_1}(w_0)$ and $p_1 \in B_{\min\{ \delta_1, \delta_2 \}}(w_1)$, then
    \begin{equation*}
        -c(x(p_1), y_{\tilde{v}} ) + c(x(p_0), y_{\tilde{v}}) + c(x(p_1), y_0) -c(x(p_0), y_1)+h_1 \geq \frac{r'}{32L}\cos(\eta) \| z_1 - z_0 \|,
    \end{equation*}
    and 
    \begin{equation*}
        -c(x(p_1) , y_{-\tilde{v}}) + c(x(p_0), y_{-\tilde{v}}) + c(x(p_1), y_0) - c(x(p_0), y_1)+h_1 \leq - \frac{r'}{32L}\cos(\eta)\| z_1 - z_0 \|.
    \end{equation*}
    In particular the intermediate value theorem yields that if $\epsilon< \frac{r'}{32L}\cos(\eta) \| z_1 - z_0 \|$, then for any $p_0 \in B_{\gamma_1}(w_0)$ and $p_1 \in B_{\min\{ \delta_1, \delta_2\} }(w_1)$, there exists $t(p_0,p_1) \in [0,1]$ such that
    \begin{equation*}
        \begin{aligned}
            -c(x(p_1),y(p_0,p_1)) + c(x(p_0) , y(p_0,p_1)) +c(x(p_1), y_0) -c(x(p_0),y_1)+h_1 & = \epsilon,
        \end{aligned}
    \end{equation*}
    where $y(p_0,p_1) = \cexp{x_0}{q_1^0 + t(p_0,p_1)M_{01}\tilde{v}}$. We can write the above equality as
    \begin{equation}\label{eqn: main 2: t(p)}
        -c(x(p_1), y(p_0,p_1)) + c(x(p_0) , y(p_0,p_1)) - c(x(p_0),y_1) + h_1 = -c(x(p_1), y_0) +\epsilon.
    \end{equation}
    We choose 
    \begin{equation*}
        \tilde{\delta} = \frac{1}{4}\min\{ \delta_0, \delta_1, \delta_2, \rho_0, \rho_1, \tilde{\rho}, \gamma_0, \gamma_1 \}.
    \end{equation*}
    Note that the third part of Lemma \ref{lem: non Loeper set} yields $\| w_1 - w_0 \| < 2\rho$, and \eqref{eqn: main 2: delta1} implies $\tilde{\delta}< \rho$ which we required in \emph{Step 2}.

    \noindent \textbf{\emph{Step 4)}} We recall that we assumed $\epsilon < \min\{\epsilon_\delta, \frac{\delta r}{2}\cos(\tilde{\theta})\}$ in \emph{Step 1}. Therefore, since we have chosen a value for $\tilde{\delta}$, we use
    \begin{equation}\label{eqn: main 2: epsilon bound}
        \epsilon < \min\{\epsilon_{\tilde\delta}, \frac{\tilde{\delta} r}{2}\cos(\tilde{\theta})\}
    \end{equation}
    hereafter. To choose $\tilde{\epsilon}$, we will need to describe the idea for proving that $\phi_\epsilon$ is alternative $c$-convex for $\epsilon = \tilde{\epsilon}$. Note first that the function $\phi(x) = \Phi(-D_y c(x, y_0))$ is a $c$-convex function. Then Lemma \ref{lem: c-conv is alt c-conv} yields that $\phi$ is an alternative $c$-convex function. In addition, the function $x \mapsto \max\{\phi(x), -c(x,y_0) +\epsilon\}$ is also $c$-convex and alternative $c$-convex. \\
    To show that $\phi_\epsilon$ is alternative $c$-convex, we need to check that, for any $X_i=(x_i, \phi_\epsilon(x_i)) \in \graph{\phi_\epsilon}$, we have
    \begin{equation}\label{eqn: main 2: phi epsilon alt c-conv}
        \phi_\epsilon (x) \leq F_{X_0 X_1}(x),
    \end{equation}
    for any $x \in \X$. If $x_i \in \{ x \in \X | \phi_\epsilon(x) = \max\{ \phi(x), -c(x,y_0)+\epsilon\} \}$ for $i = 0,1$, then we have
    \begin{equation*}
        F_{X_0 X_1}(x) \geq \max\{\phi(x), -c(x,y_0)+\epsilon \} \geq \phi_\epsilon(x).
    \end{equation*}
    On the other hand, if $\phi_\epsilon(x) = \phi(x)$, then letting $X_i' =(x_i, \phi(x_i))$, by Lemma \ref{lem: ordered F}, we have
    \begin{equation*}
        F_{X_0 X_1}(x) \geq F_{X_0' X_1'}(x) \geq \phi(x) = \phi_\epsilon(x).
    \end{equation*}
    Hence, we only need to consider the cases where $\phi(x_0) < -c(x_0, y_0) + \epsilon$ and $\phi_\epsilon(x) = -c(x,y_0)+\epsilon$. we divide this into a few cases and apply different methods to show \eqref{eqn: main 2: phi epsilon alt c-conv} in each case. Let $p_i = -D_y c(x_i, y_0)$ and $p = -D_y c(x, y_0)$. Then, \eqref{eqn: main 2: epsilon delta} yields that, choosing $\epsilon < \epsilon_\delta$, we have $p_0 \in B_\delta(w_0)$ and $p \in B_\delta(w_1)$. Hence, to show \eqref{eqn: main 2: phi epsilon alt c-conv}, we need to show that, for any $p \in B_\delta(w_1)$, there exists $y \in Y$ and $h \in \R$ such that 
    \begin{equation*}
        \begin{aligned}
            &-c(x_i,y) +h \leq \phi_{\epsilon}(x_i), \quad i=0,1, \\
            &-c(x(p),y)+h \geq \phi_\epsilon(x(p)).
        \end{aligned}
    \end{equation*}\\
    \emph{Case 1: $p_1 \in B_{2\tilde{\delta}}(w_1)$}.\\
    In this case, since $2 \tilde{\delta} \leq \min\{ \delta_1, \delta_2\}$ and $\tilde{\delta}<\gamma_1$, we obtain $y'=y(p_0,p_1)$ that satisfies \eqref{eqn: main 2: t(p)}. Let 
    \begin{equation*}
        f'(x) = -c(x, y') +c(x(p_0), y') -c(x(p_0),y_1)+h_1,
    \end{equation*}
    then \eqref{eqn: main 2: t(p)} yields
    \begin{equation}\label{eqn: main 2: fy'=cy0+epsilon at x1}
        f'(x(p_1)) = -c(x(p_1), y_0) + \epsilon.
    \end{equation}
    Then, noting that $\tilde{v} = \frac{r'}{2} < \min\{ \frac{l}{2\alpha L}, \frac{\sigma_0}{\alpha \beta}\}$ and $\delta < \delta_0$, we can apply \eqref{eqn: main 2: cone in -w direction in sections} and obtain
    \begin{equation}\label{eqn: main 2: case 1: covering cone in -w direction}
        \begin{aligned}
            \K_{\tilde{\eta}, -\xi}^{\tilde{\rho}} + p_1 & \subset \{p \in [\X]_{y_0}|-c(x(p), y')+c(x(p_1), y') \geq - c(x(p), y_0) +c(x(p_1), y_0) \} \\
            & = \{p \in [\X]_{y_0}|f'(x(p)) \geq - c(x(p), y_0) + \epsilon\}.
        \end{aligned}
    \end{equation}
    On the other hand, let 
    \begin{equation*}
        V_{\textit{case1}} = \left\{ v | \langle -\xi, v \rangle \leq \| \xi \| \| v \| \cos(\frac{\pi}{2}+\frac{\tilde{\eta}}{2}), \| v \| = \mu(\frac{\pi}{2}-\frac{\tilde{\eta}}{4})\right\}.
    \end{equation*}
    Observe that for any $v \in V_{\textit{case1}}$, we have $\K_{\frac{\tilde{\eta}}{4},-\xi} \subset \K_{\frac{\pi}{2}-\frac{\tilde{\eta}}{4},-v}$. Then \eqref{eqn: main 2: gamma 1} yields that we have
    \begin{equation}\label{eqn: main 2: case 1: ball at p0 in cones}
        B_{\tilde{\delta}}(w_0) \subset B_{2\tilde{\delta}}(p_0) \subset \K_{\frac{\tilde{\eta}}{4},-\xi}+p_1 \subset \K_{\frac{\pi}{2} - \frac{\tilde{\eta}}{4}, -v}+p_1.
    \end{equation}
    Also, for any $p \not\in \K_{\tilde{\eta},-\xi}+p_1$, there exists $v \in V_{\textit{case1}}$ such that $p \in \K_{\frac{\pi}{2}-\frac{\tilde{\eta}}{4},v}+p_1$. Then Lemma \ref{lem: small tilting} yields that, for any $p_1 \in B_{2\tilde{\delta}}(w_1) \setminus \K_{\tilde{\eta},-\xi}$, there exists $v \in V_{\textit{case1}}$ such that, for any $p \in B_{\tilde{\delta}}(w_1) \subset B_{2\tilde{\delta}}(p_1)$, we have
    \begin{equation}\label{eqn: main 2: case 1: tilting y0}
        \begin{aligned}
            & -c(x(p), y_v) +c(x(p_1), y_v)+c(x(p), y_0)-c(x(p_1), y_0) \\
            \geq & \frac{1}{2} \| p-p_1\| \|v \| \cos(\frac{\pi}{2}-\frac{\tilde{\eta}}{4}) \\
            \geq & 0,
        \end{aligned}
    \end{equation}
    and
    \begin{equation}\label{eqn: main 2: case 1: tilting y0 at x0} 
        \begin{aligned}
            &-c(x(p), y_v) +c(x(p_1), y_v)+c(x(p), y_0)-c(x(p_1), y_0) \\
            \leq & -\frac{1}{2} \| p - p_1 \| \| v \| \cos(\frac{\pi}{2}-\frac{\tilde{\eta}}{4}) \\
            \leq & -\frac{1}{4}\| w_0 - w_1 \| \| v \| \cos (\frac{\pi}{2}- \frac{\tilde{\eta}}{4}).
        \end{aligned}
    \end{equation} 
    where $y_v = \cexp{z_1}{q^1_0+M_{10}v}$. Therefore, denoting 
    \begin{equation*}
        f_v(x) = -c(x,y_v) + c(x(p_1), y_v) -c(x(p_1), y_0)+\epsilon,
    \end{equation*}
    \eqref{eqn: main 2: case 1: tilting y0} implies
    \begin{equation}\label{eqn: main 2: case 1: fv section covering}
        B_{\tilde{\delta}}(w_1) \setminus (\K_{\tilde{\eta},-\xi} +p_1)\subset \bigcup_{v \in V_{\textit{case1}}} \left\{p \in [X]_{y_0} \big| f_v(x(p)) \geq -c(x(p), y_0)+\epsilon
         \right\}.
    \end{equation}
    Also, for any $v \in V_{\textit{case1}}$, \eqref{eqn: main 2: case 1: ball at p0 in cones} and \eqref{eqn: main 2: case 1: tilting y0 at x0} yield that, for any $p \in B_{\tilde{\delta}}(w_0)$, we have
    \begin{equation*}
        f_v(x(p)) \leq - c(x(p), y_0) + \epsilon -\frac{1}{4}\| p_0 - p_1 \| \| v \| \cos (\frac{\pi}{2}- \frac{\tilde{\eta}}{4}).
    \end{equation*}
    Therefore, if we have 
    \begin{equation}\label{eqn: main 2: case 1: epsilon}
        \epsilon < \frac{1}{4}\| p_0 - p_1 \| \| v \| \cos(\frac{\pi}{2}- \frac{\tilde{\eta}}{4} ),
    \end{equation}
    then we obtain
    \begin{equation}\label{eqn: main 2: case 1: fv around x0}
            f_v(x(p)) \leq - c(x(p), y_0)
    \end{equation}
    for any $p \in B_{\tilde{\delta}}(w_0)$. Using \eqref{eqn: main 2: case 1: covering cone in -w direction} and \eqref{eqn: main 2: case 1: fv section covering}, we obtain
    \begin{equation}\label{eqn: main 2: case 1: ball covered by section}
        \begin{aligned}
             B_{\tilde{\delta}}(w_1) \subset & \{p \in [\X]_{y_0}|f'(x(p)) \geq f_{0}(x(p)) + \epsilon\} \\
             &\cup \bigcup_{v \in V_{\textit{case1}}} \left\{p \in [X]_{y_0} \big| f_v(x(p)) \geq f_{0}(x(p))+\epsilon
         \right\}
        \end{aligned}
    \end{equation}
    (recall $f_{0}(x) = -c(x, y_0)$ defined in \ref{lem: non Loeper set}). In addition, by definition of $f_v$ and \eqref{eqn: main 2: case 1: fv around x0}, we observe that $f_v$ is a $c$-affine function such that 
    \begin{equation*}
        f_v(x(p)) \leq c(x(p), y_0) \leq \phi_\epsilon (x(p)), \quad \forall p \in B_{\tilde{\delta}}(w_0),
    \end{equation*}
    and
    \begin{equation*}
        f_v(x(p_1)) = -c(x(p_1), y_0) + \epsilon \leq \phi_{\epsilon}(x(p_1)).
    \end{equation*}
    Therefore, if $\phi_\epsilon(x) = -c(x,y_0)+\epsilon$, then noting $p=-D_y c(x, y_0) \in B_{\tilde{\delta}}(w_1)$, \eqref{eqn: main 2: case 1: ball covered by section} yields that we have
    \begin{equation*}
        F_{X_0' X_1'}(x) \geq \sup(\{ f'(x)\} \cup\{f_v(x)| v \in V_{\textit{case1}} \}) \geq f_{0}(x) + \epsilon = \phi_\epsilon(x).
    \end{equation*}\\
    \emph{Case 2: $p_1 \in \bigcup_{v \in V_1}(\K_{\tilde{\theta},v}+w_1) \setminus B_{2\tilde{\delta}}(w_1)$.}\\
    In this case, there exists $v \in V_1$ such that $p_1 \in \K_{\tilde{\theta},v}+w_1$. Since the cone $\K_{\tilde{\theta},v}$ does not change when we multiply a positive number to $v$, we can assume $\| v \| =\frac{\min\{1, \tilde{\delta} \}}{3\max\{1,L\diam{\X}\}} r$ (note that $\| v \| < r$). Then, using \eqref{eqn: c Lipschitz} and \eqref{eqn: Lipschitz}, for any $p \in B_{\tilde{\delta}}(w_1)$, we obtain
    \begin{equation}\label{eqn: main 2: case 2: Lipschitz}
        \left| -c(x(p), y_0) + c(x(p), y_v) -c(x(w_1), y_v) +c(x(w_1), y_0) \right| \leq 2\clip \| y_0 - y_v \| \leq \frac{1}{3} L r \alpha \clip,
    \end{equation}\
    where $y_v = \cexp{z_1}{q^1_0+M_{10}v}$. On the other hand, by the definition of $\Phi_\epsilon$ and \eqref{eqn: main 2: small tilting est}, we obtain 
    \begin{equation*}
        \begin{aligned}
            \Phi_\epsilon(p_1) +c(x(p_1), y_0) & \geq -c(x(p_1), y_v) +c(z_1, y_v) + c(x(p_1),y_0) - c(z_1, y_0) \\
            & \geq \frac{1}{2} \| p_1 - w_1 \| \| v \| \cos(\tilde{\theta}) \\
            & \geq \tilde{\delta}r\cos(\tilde{\theta}),
        \end{aligned}
    \end{equation*}
    where we have used $p_1 \not\in B_{2\tilde{\delta}}(w_1)$ and $\| v \| \leq r$ in the last inequality. Also, \eqref{eqn: main 2: small tilting est} with \eqref{eqn: Lipschitz} yields 
    \begin{equation*}
        \begin{aligned}
            &-c(x(p_1) , y_v ) + c(x(w_1), y_v) - c(x(w_1), y_0) +c(x(p_1), y_0) \\
            \leq& \frac{3}{2}\| p_1 - w_1 \| \| v \| \cos( \tilde{\theta}) \\
            \leq& \frac{1}{2}\frac{\diam{[\X]_{y_0}}}{\max\{1,L\diam{\X}\}}\min\{ 1, \tilde{\delta} \} r \cos(\tilde{\theta}) \\
            \leq& \frac{1}{2}\tilde{\delta}r\cos(\tilde{\theta}).
        \end{aligned}
    \end{equation*}
    Therefore, we obtain
    \begin{equation}\label{eqn: main 2: case 2: around p1}
        \begin{aligned}
            \Phi_\epsilon(p_1) & \geq \tilde{\delta}r \cos(\tilde{\theta}) - c(x(p_1), y_0) \\
            & \geq \frac{1}{2} \tilde{\delta}r \cos(\tilde{\theta}) -c(x(p_1) , y_v ) + c(x(w_1), y_v) - c(x(w_1), y_0).
        \end{aligned}
    \end{equation}
    We also observe that, since $p_0 \in B_{\tilde{\delta}}(w_0)$, \eqref{eqn: main 2: small tilting est 2} implies
    \begin{equation}\label{eqn: main 2: case 2: around p0}
        \begin{aligned}
            & -c(x(p_0), y_v)+c(x(w_1), y_v)+c(x(p_0), y_0) -c(x(w_1), y_0) \\
            \leq& -\frac{1}{2}\|p_0 - w_1\| \| v \| \cos ( \tilde{\theta}) \\
            \leq& - \frac{1}{2}\times 2\tilde{\delta} \times \frac{\min\{1, \tilde{\delta} \}}{3\max\{1,L\diam{\X}\}} r \cos( \tilde{\theta})  \\
            =& - \frac{\min\{1, \tilde{\delta} \}}{3\max\{1,L\diam{\X}\}} \tilde{\delta} r \cos(\tilde{\theta}).
        \end{aligned}
    \end{equation}
    If we have
    \begin{equation}\label{eqn: main 2: case 2: epsilon}
        \epsilon < \min\left\{ \frac{1}{3} Lr \alpha \clip, \frac{1}{4} \tilde{\delta}r\cos(\tilde{\theta}), \frac{\min\{1, \tilde{\delta} \}}{6\max\{1,L\diam{\X}\}} \tilde{\delta} r \cos(\tilde{\theta}) \right\},
    \end{equation}
    then \eqref{eqn: main 2: case 2: around p1} and \eqref{eqn: main 2: case 2: around p0} imply
    \begin{equation*}
        \Phi_\epsilon(p_1) \geq -c(x(p_1), y_v) + c(x(w_1), y_v) -c(x(w_1), y_0) + 2\epsilon,
    \end{equation*}
    and 
    \begin{equation*}
        \Phi(p_0) \geq -c(x(p_0), y_0) \geq -c(x(p_0), y_v)+c(x(w_1), y_v) - c(x(w_1), y_0) + 2\epsilon 
    \end{equation*}
    respectively. On the other hand, \eqref{eqn: main 2: case 2: Lipschitz} yields that for any $p \in B_{\tilde{\delta}}(w_0)$, we have
    \begin{equation*}
        -c(x(p), y_v) + c(x(w_1), y_v) -c(x(w_1), y_0) + 2\epsilon \geq -c(x(p), y_0) + \epsilon.
    \end{equation*}
    Thus, for any $p \in B_{\tilde{\delta}}(w_1)$ such that $\Phi_\epsilon (p) = -c(x(p), y_0)+ \epsilon$, we obtain
    \begin{equation*}
        \begin{aligned}
            F_{X_0 X_1}(x(p)) & \geq -c(x(p), y_v) +c(x(w_1), y_v) - c(x(w_1), y_0) + 2\epsilon \\
            & \geq -c(x(p), y_0)+\epsilon \\
            &= \phi_\epsilon(x(p)).
        \end{aligned}
    \end{equation*}
    This concludes \emph{Case 2}.\\
    \emph{Case 3: otherwise.}\\
    We have $p_1 \in \left( \K_{\theta, -\xi}+w_1 \right)\setminus B_{2\tilde{\delta}}(w_1)$. Note that we also have $p_0 \in  (\K_{\theta, -\xi}+w_1) \setminus B_{2\tilde{\delta}}(w_1)$ by \eqref{eqn: main 2: ball at p0 in theta cone at p1}. Let $w_s \in [w_0, w_1]$ be such that $\| w_1 - w_s \| = \frac{3}{2}\tilde{\delta}$. Then we have
    \begin{equation*}
        B_{\tilde{\delta}}(w_1) \subset \K_{ \nu_1, \xi}+w_s,
    \end{equation*}
    where $\sin(\nu_1) = \frac{2}{3}$. On the other hand, let $\cos(\nu_0) \leq \frac{\tilde{\delta}}{4L\diam{[\X]}}$. Then, for any $p \in \left( \K_{\theta, -\xi} +w_1 \right) \setminus B_{2\tilde{\delta}}(w_1)$, using the third part of Lemma \ref{lem: non Loeper set}, we observe
    \begin{equation*}
        \begin{aligned}
            \langle p-w_s , -\xi \rangle & = \langle p-w_1+w_1 -w_s , -\xi \rangle \\
            & = \langle p-w_1, -\xi \rangle -\frac{3}{2}\tilde{\delta}\| \xi \| \\
            & \geq 2\tilde{\delta}\| \xi \| \cos (\theta) - \frac{3}{2}\tilde{\delta} \| \xi \| \\
            & = \tilde{\delta}\| \xi \| (2 \cos(\tilde{\theta})-\frac{3}{2}) \\
            & \geq \frac{1}{4} \tilde{\delta} \| \xi \| \\
            & \geq \| \xi \| \| p - w_s \| \frac{\tilde{\delta}}{4\diam{[\X]_{y_0}}} \\
            & \geq \| \xi \| \| p - w_s \| \cos(\nu_0),
        \end{aligned}
    \end{equation*}
    where we have used \eqref{eqn: Lipschitz}. Therefore, taking $\nu = \max\{ \nu_0, \nu_1 \}$, and $v = \mu(\nu) \frac{\xi}{\| \xi \|}$, Lemma \ref{lem: small tilting} yields that for any $p \in B_{\tilde{\delta}}(w_1)$, we have
    \begin{equation*}
        \begin{aligned}
            -c(x(p), y_v)+c(x(w_s), y_v) -c(x(w_s), y_0) & \geq -c(x(p), y_0) + \frac{1}{2}\| p-w_s \| \| v \| \cos(\nu) \\
            & \geq -c(x(p), y_0) + \frac{1}{4} \tilde{\delta} \mu(\nu)\cos(\nu)
        \end{aligned}
    \end{equation*}
    and for any $p \in \left( \K_{\theta, -\xi}+w_1 \right) \setminus B_{2\tilde{\delta}}(w_1)$, we have
    \begin{equation*}
        \begin{aligned}
            -c(x(p), y_v) +c(x(w_s), y_v) -c(x(w_s), y_0) & \leq -c(x(p), y_0) - \frac{1}{2}\| p-w_s \| \| v \| \cos(\nu) \\
            & \leq -c(x(p), y_0).
        \end{aligned}
    \end{equation*}
    Thus, if we had 
    \begin{equation}\label{eqn: main 2: case 3: epsilon}
        \epsilon < \frac{1}{4} \tilde{\delta} \mu(\nu)\cos(\nu),
    \end{equation}
    then, noting $p_0, p_1 \in \left( \K_{\theta, -\xi}+w_1 \right) \setminus B_{2\tilde{\delta}}(w_1) $, we obtain
    \begin{equation*}
        \phi_\epsilon(x(p_i)) \geq -c(x(p_i), y_0) \geq -c(x(p_i), y_v) +c(x(w_s), y_v) -c(x(w_s), y_0)
    \end{equation*}
    and, for any $p \in B_{\tilde{\delta}}(w_1)$ such that $\Phi_\epsilon(p) = -c(x(p), y_0) + \epsilon$, we have
    \begin{equation*}
        \phi_\epsilon(x(p)) = -c(x(p), y_0) + \epsilon \geq -c(x(p), y_v)+c(x(w_s), y_v) -c(x(w_s), y_0).
    \end{equation*}
    This concludes \emph{Case 3}.

    \noindent \textbf{\emph{Step 5)}} We choose $\tilde{\epsilon}>0$ that satisfies \eqref{eqn: main 2: epsilon bound}, \eqref{eqn: main 2: case 1: epsilon}, \eqref{eqn: main 2: case 2: epsilon}, \eqref{eqn: main 2: case 3: epsilon}. Then, \emph{Step 4} yields that $\phi_\epsilon$ is an alternative $c$-convex function. However, the set
    \begin{equation*}
        \{ x \in \X | \phi_\epsilon (x) = -c(x, y_0) + \epsilon\} = \{ x \in \X | \phi(x) \leq -c(x, y_0)+\epsilon\}\cap\cexp{y_0}{B_{\tilde{\delta}}(w_1)}
    \end{equation*}
    has non-empty interior as we know $\phi(z_1) = -c(z_1, y_0) < -c(z_1, y_0) + \epsilon$ and $\phi_\epsilon$ is continuous. Let $x' \in \Int{\{ x \in \X | \phi_\epsilon (x) = -c(x, y_0) + \epsilon\}}$. Then $x'$ is a differentiable point of $\phi_\epsilon$ and $D_x \phi_\epsilon(x') = -D_x c(x', y_0)$. If $\phi_\epsilon$ is a $c$-convex function, then $\phi_\epsilon$ must have a $c$-subdifferential at $x'$ by Lemma \ref{lem: c-subdifferential}. Let $y'\in \Y$ be a $c$-subdifferential at $x'$. Then, by the definition of $c$-subdifferential, we obtain that
    \begin{equation*}
        -D_x c(x', y') = D \phi_\epsilon(x') = -D_x c(x', y_0).
    \end{equation*}
    Using twist condition, we obtain $y' = y_0$. The definition of $c$-subdifferential yields
    \begin{equation*}
        \phi_\epsilon(x) \geq -c(x, y_0) -c(x', y_0) + \phi_\epsilon(x') = -c(x, y_0) + \epsilon.
    \end{equation*}
    The construction of $\phi_\epsilon$, however, shows that we have 
    \begin{equation*}
        -\phi(x_0) = -c(z_0, y_0) < -c(z_0, y_0) + \epsilon.
    \end{equation*}
    Thus, $\phi_\epsilon$ cannot be a $c$-convex function, i.e. $\phi_\epsilon$ is an alternative $c$-convex function that is not a $c$-convex function.
\end{proof}

\bibliographystyle{amsalpha}

\bibliography{ref.bib}

\providecommand{\bysame}{\leavevmode\hbox to3em{\hrulefill}\thinspace}
\providecommand{\MR}{\relax\ifhmode\unskip\space\fi MR }
\providecommand{\MRhref}[2]{%
  \href{http://www.ams.org/mathscinet-getitem?mr=#1}{#2}
}
\providecommand{\href}[2]{#2}
\begin{thebibliography}{MTW05}

\bibitem[Caf90]{Caf90}
Luis~A Caffarelli, \emph{A localization property of viscosity solutions to the
  monge-ampere equation and their strict convexity}, Annals of mathematics
  \textbf{131} (1990), no.~1, 129--134.

\bibitem[Caf91]{Caf91}
\bysame, \emph{Some regularity properties of solutions of monge ampere
  equation}, Communications on pure and applied mathematics \textbf{44} (1991),
  no.~8-9, 965--969.

\bibitem[Caf92]{Caf92}
Luis~A. Caffarelli, \emph{The regularity of mappings with a convex potential},
  J. Amer. Math. Soc. \textbf{5} (1992), no.~1, 99--104. \MR{1124980}

\bibitem[CW16]{CW16}
Shibing Chen and Xu-Jia Wang, \emph{Strict convexity and c1, $\alpha$
  regularity of potential functions in optimal transportation under condition
  a3w}, Journal of Differential Equations \textbf{260} (2016), no.~2,
  1954--1974.

\bibitem[DPF13a]{PF13}
Guido De~Philippis and Alessio Figalli, \emph{Sobolev regularity for
  {M}onge-{A}mp\`ere type equations}, SIAM J. Math. Anal. \textbf{45} (2013),
  no.~3, 1812--1824. \MR{3066801}

\bibitem[DPF13b]{PF13monge}
\bysame, \emph{W 2, 1 regularity for solutions of the monge--amp{\`e}re
  equation}, Inventiones mathematicae \textbf{192} (2013), no.~1, 55--69.

\bibitem[FK10]{FK10}
Alessio Figalli and Young-Heon Kim, \emph{Partial regularity of brenier
  solutions of the monge-ampere equation}, Discrete and Continuous Dynamical
  Systems. Series A \textbf{28} (2010), no.~2, 559--565.

\bibitem[FKM13]{FKM13}
Alessio Figalli, Young-Heon Kim, and Robert~J. McCann, \emph{H\"{o}lder
  continuity and injectivity of optimal maps}, Arch. Ration. Mech. Anal.
  \textbf{209} (2013), no.~3, 747--795. \MR{3067826}

\bibitem[GK15]{GK15}
Nestor Guillen and Jun Kitagawa, \emph{On the local geometry of maps with
  {$c$}-convex potentials}, Calc. Var. Partial Differential Equations
  \textbf{52} (2015), no.~1-2, 345--387. \MR{3299185}

\bibitem[GK17]{GK17}
\bysame, \emph{Pointwise estimates and regularity in geometric optics and other
  generated {J}acobian equations}, Comm. Pure Appl. Math. \textbf{70} (2017),
  no.~6, 1146--1220. \MR{3639322}

\bibitem[KM10]{KM10}
Young-Heon Kim and Robert~J. McCann, \emph{Continuity, curvature, and the
  general covariance of optimal transportation}, J. Eur. Math. Soc. (JEMS)
  \textbf{12} (2010), no.~4, 1009--1040. \MR{2654086}

\bibitem[KZ20]{KZ20}
Gabriel Khan and Jun Zhang, \emph{The kähler geometry of certain optimal
  transport problems}, Pure Appl. Anal. \textbf{2} (2020), no.~2, 397--426.

\bibitem[Liu09]{Liu09}
Jiakun Liu, \emph{H\"{o}lder regularity of optimal mappings in optimal
  transportation}, Calc. Var. Partial Differential Equations \textbf{34}
  (2009), no.~4, 435--451. \MR{2476419}

\bibitem[Loe09]{Loe09}
Gr\'{e}goire Loeper, \emph{On the regularity of solutions of optimal
  transportation problems}, Acta Math. \textbf{202} (2009), no.~2, 241--283.
  \MR{2506751}

\bibitem[LT20]{LT07}
Gr{\'e}goire Loeper and Neil~S Trudinger, \emph{Weak formulation of the mtw
  condition and convexity properties of potentials}, arXiv preprint
  arXiv:2007.02665 (2020).

\bibitem[MTW05]{MTW}
Xi-Nan Ma, Neil~S. Trudinger, and Xu-Jia Wang, \emph{Regularity of potential
  functions of the optimal transportation problem}, Arch. Ration. Mech. Anal.
  \textbf{177} (2005), no.~2, 151--183. \MR{2188047}

\bibitem[Ran23]{Ran23}
Cale Rankin, \emph{A remark on the geometric interpretation of the a3w
  condition from optimal transport}, Bulletin of the Australian Mathematical
  Society \textbf{108} (2023), no.~1, 162–165.

\bibitem[Wan96]{W96}
Xu-Jia Wang, \emph{Regularity for monge-ampere equation near the boundary},
  Analysis \textbf{16} (1996), no.~1, 101--108.

\end{thebibliography}

\end{document}